\documentclass{amsart}
\usepackage{amsthm,amsmath,amssymb,amsfonts,amsbsy,bbm,mathrsfs,supertabular,
eurosym,graphicx,enumitem,xcolor,mathtools}

\RequirePackage[numbers]{natbib}
\usepackage[normalem]{ulem} 
\RequirePackage[colorlinks,citecolor=blue,urlcolor=blue]{hyperref}
\RequirePackage{graphicx}
\usepackage{kantlipsum} 
\calclayout

\usepackage{tikz}

\theoremstyle{plain}

\newtheorem{cor}{Corollary}
\newtheorem{prop}{Proposition}
\newtheorem{exa}{Example}
\newtheorem{thm}{Theorem}
\newtheorem{lem}{Lemma}
\theoremstyle{remark}
\newtheorem{df}{Definition}
\newtheorem{ass}{Assumption}

\newcommand{\pa}[1]{\left({#1}\right)}
\newcommand{\norm}[1]{\left\|{#1}\right\|}
\newcommand{\cro}[1]{\left[{#1}\right]}
\newcommand{\ab}[1]{\left|{#1}\right|}
\newcommand{\ac}[1]{\left\{{#1}\right\}}

\newcommand{\Var}{\mathop{\rm Var}\nolimits}

\newcommand{\B}{{\mathbb{B}}}

\newcommand{\E}{{\mathbb{E}}}
 
\renewcommand{\L}{{\mathbb{L}}}
\newcommand{\N}{{\mathbb{N}}}
\renewcommand{\P}{{\mathbb{P}}}
 
\newcommand{\R}{{\mathbb{R}}}

\newcommand{\sA}{{\mathscr{A}}}
\newcommand{\sB}{{\mathscr{B}}}
\newcommand{\sC}{{\mathscr{C}}}
\newcommand{\sD}{{\mathscr{D}}}
\newcommand{\sE}{{\mathscr{E}}}
\newcommand{\sF}{{\mathscr{F}}}

\newcommand{\sM}{{\mathscr{M}}}

\newcommand{\sO}{{\mathscr{O}}}
\newcommand{\sP}{{\mathscr{P}}}

\newcommand{\sT}{{\mathscr{T}}}

\newcommand{\sX}{{\mathscr{X}}}

\newcommand{\frD}{{\mathfrak{D}}}

\DeclareMathAlphabet{\mathscrbf}{OMS}{mdugm}{b}{n}
\newcommand{\cA}{{\mathcal{A}}}

\newcommand{\cD}{{\mathcal{D}}}
\newcommand{\cE}{{\mathcal{E}}}
\newcommand{\cF}{{\mathcal{F}}}

\newcommand{\cI}{{\mathcal{I}}}
\newcommand{\cJ}{{\mathcal{J}}}

\newcommand{\cL}{{\mathcal{L}}} 
\newcommand{\cM}{{\mathcal{M}}}

\newcommand{\cO}{{\mathcal{O}}}

\newcommand{\cR}{{\mathcal{R}}}
 
\newcommand{\cT}{{\mathcal{T}}}

\newcommand{\cV}{{\mathcal{V}}}

\newcommand{\gt}{{\mathbf{t}}}
\newcommand{\gu}{{\mathbf{u}}}
\newcommand{\gv}{{\mathbf{v}}}
\newcommand{\gw}{{\mathbf{w}}}
\newcommand{\gx}{{\mathbf{x}}}
\newcommand{\gy}{{\mathbf{y}}}

\newcommand{\gE}{{\mathbf{E}}}

\newcommand{\gI}{{\mathbf{I}}}

\newcommand{\gP}{{\mathbf{P}}}

\newcommand{\gT}{{\mathbf{T}}}

\newcommand{\gZ}{{\mathbf{Z}}}

\newcommand{\bs}[1]{\boldsymbol{#1}}

\newcommand{\bsX}{{\bs{X}}}

\newcommand{\eps}{\varepsilon}

\newcommand{\eref}[1]{(\ref{#1})}
\renewcommand{\ge}{\geqslant}
\renewcommand{\le}{\leqslant}
\renewcommand{\geq}{\geqslant}
\renewcommand{\leq}{\leqslant}
\newcommand{\1}{1\hskip-2.6pt{\rm l}}
\newcommand{\0}{{\bf 0}}

\newcommand{\scal}[2]{\langle #1,#2\rangle}

\newcommand{\etc}[1]{#1_1,\ldots,#1_n}

\newcommand{\dps}[1]{\displaystyle{#1}}

\newcommand{\on}{^{\otimes n}}

\newcommand{\et}{^{\star}}

\newcommand{\dTV}[2]{d\!\left(#1,#2 \right)}

\DeclarePairedDelimiter\ceil{\lceil}{\rceil}

\newcommand{\PES}[1]{\ceil*{#1}}

\begin{document}

\title[Robust density estimation with the $\L_{1}$-loss]{Robust density estimation with the $\L_{1}$-loss. Applications to the estimation of a density on the line satisfying a shape constraint}

\author{Yannick BARAUD}
\address{Department of Mathematics, University of Luxembourg, Maison du nombre, 
6 avenue de la Fonte, L-4364 Esch-sur-Alzette, Grand Duchy of Luxembourg}
\email{yannick.baraud@uni.lu}

\author{H\'el\`ene HALCONRUY}
\address{Samovar, Telecom SudParis, Institut Polytechnique de Paris, France}
\email{helene.halconruy@telecom-sudparis.eu}

\author{Guillaume MAILLARD}
\email{guillaume.maillard@ensai.fr}
\address{ENSAI
Campus de Ker Lann
51 Rue Blaise Pascal
BP 37203
35172 BRUZ Cedex, France}


\begin{abstract}
We tackle the problem of estimating the distribution of presumed i.i.d.\ observations for the total variation loss. Our approach is based on density models and is versatile enough to cope with many different ones, including some for which the Maximum Likelihood Estimator (MLE for short) does not exist. We mainly illustrate the properties of our estimator on models of densities on the line that satisfy a shape constraint. We show that it possesses some similar optimality properties, with regard to some global rates of convergence, as the MLE does when it exists. It also enjoys some adaptation properties with respect to some specific densities in the model for which our estimator is proven to converge at the parametric rate $1/\sqrt{n}$. More important is the fact that our estimator is robust, not only with respect to model misspecification, but also to contamination, the presence of outliers among the dataset and the equidistribution assumption we started from. This means that the estimator performs almost as well as if the data were i.i.d.\ with density $p$ in a situation where these data are only independent and most of their marginals are close enough in total variation to a distribution with density $p$. We also show that our estimator converges to the average density of the data, when this density belongs to the model, even when none of the marginal densities does. Our main result on the risk of the estimator takes the form of an exponential deviation inequality which is nonasymptotic and involves explicit numerical constants. We deduce from it several global rates of convergence, including some bounds for the minimax $\L_{1}$-risks over the sets of concave, log-concave or more generally $s$-concave densities with $s>-1$. These bounds derive from some approximation results of densities which are monotone, convex, concave or more generally $s$-concave. These results may be of independent interest. 
\end{abstract}


\thanks{This project has received funding from the European Union's Horizon 2020 research and innovation programme under grant agreement N\textsuperscript{o} 811017}

\maketitle

\section{Introduction}
Estimating a density under a shape constraint has been addressed by many authors since the pioneering papers by Grenander ~\cite{MR0093415,MR599175}, Rao~\cite{rao1969estimation}, Groeneboom~\cite{MR822052} and Birg\'e~\cite{MR1026298} for estimating a nonincreasing density on $(0,+\infty)$. It is well-known that this problem can be elegantly solved by the Grenander estimator -- see  Grenander~\cite{MR0093415} -- which is the maximum likelihood estimator (MLE for short) over the set of densities that satisfy this monotonicity constraint on $(0,+\infty)$. Practically, the MLE is obtained by taking the (right-hand) derivative of the least concave majorant of the cumulative distribution function of the data. Rao~\cite{rao1969estimation} and  Grenander ~\cite{MR0093415,MR599175} established the local asymptotic properties of the Grenander estimator while  Groeneboom~\cite{MR822052} and  Birg\'e~\cite{MR1026298} studied its 
global $\L_{1}$-estimation error over some functional classes of interest. Birg\'e proved that the worst-case risk of the Grenander estimator over  the set $\sF(H,L)$, that consists of nonincreasing densities bounded by $H>0$ and supported on $[0,L]$ with $L>0$, is of order $(\log(1+HL)/n)^{1/3}$. This rate turns out to be optimal in the minimax sense over this class of densities. The Grenander estimator therefore performs almost as well (apart maybe from numerical constants) as a minimax estimator over $\sF(H,L)$ that would know the values of $H$ and $L$ in advance. Even more surprising is the fact that the Grenander estimator converges at the parametric rate $1/\sqrt{n}$ when the target density is piecewise constant on the elements of partition of $(0,+\infty)$ into a finite number of intervals -- see Grenander ~\cite{MR599175}, Groeneboom~\cite{MR822052} and Birg\'e~\cite{MR1026298}. As a consequence, the Grenander estimator can nicely {\em adapt} to the specific features of the density to be estimated even though these features are {\em a priori} unknown. 

Because of its nice theoretical properties and the fact that it can easily be calculated in this specific situation,  the MLE has widely and almost exclusively been used to solve many other density estimation problems under shape constraints. We refer to  Groeneboom {\em et al}~\cite{MR1891741} for convex densities and to Balabdaoui and Wellner ~\cite{MR2435342} and Gao and Wellner~\cite{MR2520591} for $k$-monotone ones. Although the construction of the MLE is not based on smoothness assumptions in these cases, it still requires that some specific features of the target density, and in particular its support, be known. It was already the case for the monotonicity constraint since the left-endpoint of its support needed to be known to build the Grenander estimator. This kind of prior information might, however, not be available in practice and a more reasonable assumption would be that only an interval containing this support be known. Unfortunately, under this weaker assumption the MLE does not exist and a search for alternative  estimators becomes necessary to solve this issue. 

To our knowledge, the first attempt to solve it dates back to Wegman \cite{Wegman_1970}. He  designed an MLE-type estimator restricted to a class of unimodal densities that attain their modes
on an interval of length not smaller than some parameter $\varepsilon>0$. This parameter needs to be tuned by the statistician and its choice influences the performance of the resulting estimator. Birg\'e \cite{Birge_1997} proposed a different approach based on data-driven choice of a Grenander estimator among the collection of those associated to all possible modes. Up to an additional term of order $1/\sqrt{n}$, he proved that the $\L_{1}$-risk of the selected estimator is the same as that of the Grenander estimator that would know the value of the mode in advance. 

The situation is different when the density is assumed to be log-concave on $\R$, or more generally on $\R^{d}$ with $d>1$. The construction of the MLE is then free of any assumption on the support of the density. The study of the MLE on the set of log-concave densities has led to intensive work. We refer the reader to D\"umbgen and Rufibach~\cite{MR2546798}, Doss and Wellner~\cite{MR3485950} , Cule and Samworth~\cite{MR2645484}, Kim and Samworth~\cite{Kim2016} and Feng {\em et al}~\cite{Feng2021} as well as the references therein. Kim and Samworth~\cite{Kim2016} described the uniform rates of convergence of the MLE for the squared Hellinger loss over the class of log-concave densities in dimension $d\in \{1,2,3\}$ and they proved these rates to be minimax (up to a possible logarithmic factor). Besides, as for the monotonicity constraint in dimension one, the MLE also possesses for log-concave densities some adaptation properties: it converges at the parametric rate $1/\sqrt{n}$ (for the Hellinger loss and up to a possible logarithmic factor) when the logarithm of the target density is piecewise affine on a suitable convex subset of $\R^{d}$ with $d\in\{1,2,3\}$. This result was established by Kim {\em et al}~\cite{MR3845018} when $d=1$ and extended to the dimensions $d\in\{2,3\}$ by Feng {\em et al}~\cite{Feng2021}. In dimension $d\geqslant 4$, Kur {\em et al} \cite{Kur_2019optimality} showed that the MLE converges at a minimax rate (up to a logarithmic factor ) for the Hellinger loss. 

In the literature, there exist various generalisations of log-concavity. They usually take the following form. Given some function $\cL$ on $\R_{+}$, one considers the class of these densities $p$ on $\R^{d}$ for which  $\cL\circ p$ is concave. In the present paper, we shall restrict to the case $d=1$ and call these densities {\em $\cL$-concave}. Of particular interest is the case where the function $\cL$ is of the form $u\mapsto {\rm sign}(s) u^{s}$, with $s\in\R\setminus\{0\}$. An $\cL$-concave density is then called $s$-concave and log-concave ones are 0-concave by convention. While log-concave densities possess light tails that decrease at least exponentially fast as we move away from their modes, $s$-concave densities may exhibit heavier tails that decrease at a polynomial rate. The study of the existence, the consistency and the rates of convergence of the MLE over classes of $\cL$-concave densities can be found in Seregin and Wellner \cite{Ser-Well2010} and Doss and Wellner \cite{MR3485950}. In particular, Doss and Wellner \cite{MR3485950} showed that the MLE does not exist over the class of $s$-concave densities when $s<-1$ while for $s>-1$, it converges at a rate which is at least $n^{-2/5}$ for the Hellinger loss. The convergence established by Doss and Wellner is, however, not uniform over the class when $s\in (-1,0)$. For these values of $s$, we do not know whether or not the convergence of the MLE is uniform. It seems that the study of the MLE for estimating such densities has received less attention than for log-concave ones. In particular, we are not aware of any adaptation properties of the MLE over classes of $s$-concave densities when $s\ne 0$. Finally, we mention that some alternative approaches based on Rényi divergences have been proposed by Koenker and Mizera \cite{Koen-Miz2010} and Han and Wellner \cite{Han-Well2016} to solve the problem of estimating $s$-concave densities. The results established in these papers mainly address the calculability of these estimators, their connexions with the MLE and their asymptotic performance.


In the one dimensional case, our aim is to design a versatile estimation strategy that can be applied to a wide variety of density models, including some for which the MLE does not exist, and that automatically results in estimators with good estimation properties. In particular, these estimators should keep the nice minimax and adaptation properties of the MLE, when it exists, for estimating a density under a shape constraint. They should also remain stable with respect to a slight departure from the ideal situation where the data are truly i.i.d.\ and their density satisfies the required  shape constraint. In particular, the estimator should still perform well when the equidistribution assumption is slightly violated and the data set contains a small portion of outliers. It should also perform well when the shape of the density is slightly different from what was originally expected, that is, when the true density of the data does not satisfy the shape constraint but is close enough (with respect to the $\L_{1}$-loss) to a density that does satisfy it. In a nutshell, our aim is to build estimators that are {\em robust}. 

By robust, we mean here that the performance of the estimator should remain stable with respect to a small discrepancy (with respect to the Hellinger or the $\L_{1}$ distance) between the target density and the statistical model. It turns out that in some cases the MLE provides a robust solution when the support of the density is known or partly known. This is for example the case for estimating a nonincreasing or a convex density on $(a,+\infty)$ when $a\in\R$ is known or a concave density on $(a,b)$ when $a,b\in\R$ are both known.  The robustness of the MLE in such models can be viewed as a consequence of Baraud and Birg\'e~\cite{BarBir2018}[Section~6, Proposition~5] which states the following: if the MLE over a convex set of densities exists and takes positive values at the data,  it necessarily coincides with a $\rho$-estimator and therefore inherits its robustness properties. The sets that consist of nonincreasing, convex or concave densities on a given interval are convex and this result therefore applies. An illustration of this is that the Grenander estimator on the set $\cF$ of nonincreasing densities on $(0,+\infty)$ remains stable for estimating a density $p\et$ on $(0,+\infty)$ which does not belong to $\cF$ but lies close enough to it (with respect to the $\L_{1}$-distance). However, this apparent stability of the Grenander estimator should be interpreted with caution. The following example shows that a small error on the presumed support of $p\et$ may result in a dramatic misspecification of the model. Assume that $p\et$ is the uniform distribution on $(a,2a)$, where $a$ is a small positive number, and consider an arbitrary (left continuous) density $q$ in the set $\cF$ of nonincreasing densities on $(0,+\infty)$. Then,  
\[
\int_{0}^{+\infty}\ab{p\et(x)-q(x)}dx\ge \int_{0}^{a}q(x)dx+\int_{a}^{2a}\ab{\frac{1}{a}-q(x)}dx\ge a q(a)+\pa{\frac{1}{a}- q(a)}a=1.
\]
In particular, the $\L_{1}$-risk of the Grenander estimator on $\cF$ for estimating $p\et$ is at least 1 while it would no exceed $3.6/\sqrt{n}$ if $a$ were known, as proven in Birg\'e~\cite{MR1026298} [Theorem~1 and inequality~(3.3)]. This example shows that a small error on the presumed support of the target density may result in a large estimation error. Our approach allows one  to deal with statistical models that only rely on the presumed shape of the density and not on some information on its support. As a consequence, our estimators do not suffer from such weaknesses. 

We are not aware of many robust strategies for estimating a density under a shape constraint. For estimating concave and log-concave densities, Chan {\em et al}
~\cite{Chan2014EfficientDE} proposed a piecewise linear estimator on a data-driven partition of $\R$ into intervals. Their estimator is minimax optimal on the sets of concave and log-concave densities and it enjoys some robustness properties with respect to a departure (in $\L_{1}$-distance) of the true density from the model. Their approach is based on the estimation procedure described in Devroye and Lugosi~\cite{MR1843146} and uses the fact that the Yatracos class associated to the set of the densities that are piecewise linear on a partition of the line into a fixed  number of intervals is a Vapnik-Chervonenkis (VC) class. We refer the reader to van~der Vaart and Wellner~\cite{MR1385671} for an introduction to VC classes and their applications to statistics. Despite the desirable properties described above, the estimator proposed by Chan {\em et al} does not possess some of the nice ones that make the MLE so popular. For estimating a log-concave density, the MLE converges at a global rate of order $n^{-2/5}$ (for the $\L_{1}$ and Hellinger distances) but, as already mentioned, it also possesses some adaptation properties with respect to these densities the logarithms of which are piecewise linear. The estimator proposed by Chan {\em et al} does not possess such a property. Besides, their approach provides competitors to the MLE for some specific density models only.  Chan {\em et al}'s approach cannot deal with the estimation of a monotone density on a half-line for example and therefore cannot be used to provide a surrogate to the Grenander estimator. 

In dimension one, Baraud and Birg\'e~\cite{MR3565484}[Section~7] proposed to solve the problem of robust estimation of a density under a shape constraint by using  $\rho$-estimation. Their results hold for the squared Hellinger loss while ours is for the total variation one (TV-loss for short). The estimator we propose is more specifically designed for this loss and quite surprisingly the risk bounds we get for the TV-loss are slightly different from those obtained by Baraud and Birg\'e for the (squared) Hellinger one. We do not know if $\rho$-estimators would satisfy the same $\L_{1}$-risk bounds as those we establish here. 

Our procedure shares some similarities with that proposed by Devroye and Lugosi~\cite{MR1843146}. When the Yatracos class associated with the density model is VC, the risk bound we establish is similar to theirs except for the fact that we provide explicit numerical constants. However, unlike them, we also consider density models for which the Yatracos class is not VC, which is typically the case for models of densities that satisfy a shape constraint. Nevertheless, it is likely that with the same proof techniques, we could establish similar results for Devroye and Lugosi's estimators as those we prove for ours. 

The theory of $\ell$-estimation introduced in Baraud~\cite{BY-TEST} provides a generic way of building estimators that possess the robustness properties we are looking for. Even though the present paper is in the same vein, we modify Baraud's  procedure and establish, for the modified $\ell$-estimator, risk bounds with numerical constants that are essentially divided by a factor 2 as compared to his. {Another important difference with Baraud's result lies in the following fact. When the data are only independent with marginal densities $p_{1}\et,\ldots,p_{n}\et$, we measure the performance of our density estimator $\widehat p$ in terms of its $\L_{1}$-distance $\|p\et-\widehat p\|$ between $\widehat p$ and the average of the marginal densities $p\et=n^{-1}\sum_{i=1}^{n}p_{i}\et$. In contrast, Baraud considered, as a loss function, the average of the $\L_{1}$-distances of $\widehat p$ to the $p_{i}\et$, i.e.\ the quantity $n^{-1}\sum_{i=1}^{n}\norm{\widehat p-p_{i}\et}$. As a consequence, unlike Baraud, we can establish the convergence of  our estimator to $p\et$, as soon as its belongs to the model, even in the unfavourable situation where none of the marginals $p_{i}\et$ belongs to it.} 

The risk bounds we obtain hold for very general density models but our applications focus on the estimation of a density on the line that satisfies a shape constraint.  In a nutshell, we establish the following results which are to our knowledge new in the literature. 
\begin{itemize}
\item The procedure applies to a large variety of density models including some for which the MLE does not exist, e.g.\ the set of all monotone densities on an unknown half-line, the set of all unimodal densities on $\R$, the set of all convex densities on an interval or the set of $s$-concave densities with $s<-1$.
\item The global rates of convergence that we establish for our estimator are optimal in all the models we consider.
\item The estimator possesses some adaptation properties: it converges at the parametric rate $1/\sqrt{n}$ when the data are i.i.d.\ with a density that belongs to the model and satisfies some special properties. In particular, our estimator shares similar adaptation properties as those established for the MLE under a monotonicity or a log-concavity constraint. We also establish some adaptation properties on density models for which the MLE does not even exist.  
\item The estimator is robust with respect to model misspecification, contamination, the presence of outliers and is robust with respect to a departure from the equidistribution assumption we started from. 
\end{itemize}

Of particular interest is the class of $\cL$-concave densities. In this special case, we prove that our TV-estimator exists and enjoys some adaptation properties  as soon as the function $\cL$ is increasing. This result applies thus to all classes of $s$-concave densities without any restriction on $s$. If the function $\cL$ is additionally convex, the $\L_{1}$-risk of our TV-estimator can be bounded uniformly over the class of $\cL$-concave densities by a quantity that goes to 0 at rate $n^{-2/5}$. A similar bound can  also be established when $\cL$ is increasing and concave provided that the right and left derivatives of $\cL$ satisfy some suitable conditions. We show that these results apply, in particular, to all classes of $s$-concave densities with parameters $s>-1$.

The paper is organized as follows. The statistical framework is described in Section~\ref{sect-2} and the construction of the estimator as well as its properties are presented in Section~\ref{sec-ell-estimator}. The more specific properties of our estimator for estimating a mixture of densities that are monotone, convex or concave can be found in Sections~\ref{sect-kmono} and~\ref{sect-cvxcve} respectively while the cases of log-concave, $s$-concave and more generally $\cL$-concave densities are tackled in Section~\ref{section-s-shape}. These sections also contain some approximation results which may be of independent interest and are central to our approach. Some concluding remarks in Section~\ref{sect-conclusion} discuss the specificities and advantages of our approach as well as its limitations. The proofs are postponed to Section~\ref{sect-7}. 

\section{The statistical framework and main notations}\label{sect-2}
Let $\etc{X}$ be $n$ independent random variables and $P_{1}\et,\ldots,P_{n}\et$ their marginals on a measurable space $(\sX,\sA)$. Our aim is to estimate the $n$-tuple $\gP\et=(P_{1}\et,\ldots,P_{n}\et)$ from the observation of $\bsX=(X_{1},\ldots,X_{n})$ on the basis of a suitable {\em model} for $\gP\et$. More precisely, given a $\sigma$-finite measure $\mu$ on $(\sX,\sA)$ and a family $\overline \cM$ of densities with respect to $\mu$, we shall treat the $X_{i}$ as if they were i.i.d.\ with a density that belongs to $\overline \cM$, even though this might not be true, and estimate $\gP\et$ by a $n$-tuple of the form $(\widehat P,\ldots,\widehat P)$ where $\widehat P=\widehat P(\bsX)=\widehat p\cdot \mu$ is a random element of the set $\overline \sM=\{P=p\cdot \mu, p\in\overline\cM\}$. We refer to $\overline \sM$ and $\overline \cM$ as our probability and density models respectively. For the sake of simplicity, we abusively identify $\gP\et$ with the distribution $\bigotimes P_{i}\et$ of the observation $\bsX$. 

The density models we have in mind are nonparametric and consist of densities that satisfy a given shape constraint: monotonicity on a half line, convexity on an interval, log-concavity on the line, among other examples. 

In order to evaluate the accuracy of our estimator, we use the TV-loss $d$ on the set $\sP$ of all probability measures on $(\sX,\sA)$. We denote by $\norm{\cdot}$ the $\L_{1}$-norm on the set $\L_{1}(\sX,\sA,\mu)$ that consists of the equivalence classes of integrable functions on $(\sX,\sA,\mu)$. We recall that the TV-loss is a distance defined for $P,Q\in\sP$ by 
\begin{equation}\label{TV_def_eq}
\dTV{P}{Q}=\underset{A\in\sA}\sup\cro{P(A)-Q(A)}
\end{equation}
and if $P$ and $Q$ are absolutely continuous with respect to our dominating measure $\mu$, 
\[
\dTV{P}{Q}=\frac{1}{2}\norm{\frac{dP}{d\mu}-\frac{dQ}{d\mu}}.
\]
In general, for each $\gP \et$, denote by $P\et$ the uniform mixture of the marginals:
\[ 
P\et = \frac{1}{n} \sum_{i = 1}^n P_i\et. 
\]
In particular, when the data are i.i.d., their common distribution is $P\et\in\sP$.
%

Throughout this paper, we assume the following.
\begin{ass}\label{Ass00}
There exists a countable subset $\cM$ of $\overline \cM$ that is dense in $\overline \cM$ for the $\L_{1}$-norm.
\end{ass}
We recall that a subset of a separable metric space is separable. In particular, when the space $\L_{1}(\sX,\sA,\mu)$ is separable for the $\L_{1}$-norm, so is any subset $\overline \cM$ of densities on $(\sX,\sA,\mu)$ and Assumption~\ref{Ass00} is automatically satisfied. This is in particular the case when $(\sX,\sA)=(\R^{k},\sB(\R^{k}))$ with $k\ge 1$ and $\mu$ is the Lebesgue measure. If a family $\overline \cM$ of densities satisfies our Assumption~\ref{Ass00}, so does any subset $\overline \cD$ of $\overline \cM$. The set $\overline\cD$  may in turn be associated with a subset $\cD$ and a probability set $\sD=\{P=p\cdot\mu,\; p\in\cD\}$ that are both countable and respectively dense in $(\overline \cD,\norm{\cdot})$ and  $\overline \sD=\{P=p\cdot\mu,\; p\in\overline \cD\}$ for the total variation distance $d$. We may therefore write
\[
\inf_{P\in \sD}\dTV{P\et}{P}=\inf_{P\in \overline \sD}\dTV{P\et}{P}.
\]
We shall repeatedly apply this equality to sets $\overline \sD$ of interest without any further notice. As a consequence,
 replacing a density model $\overline \cD$ by a countable and dense subset $\cD$ changes nothing from the approximation point of view. Nevertheless, we prefer to work with $\cD$ rather than $\overline \cD$ in order to avoid some measurability issues that may result from the calculation of the supremum of an empirical process indexed by $\overline \cD$.  

Throughout the present paper, we use the same kind of notations as $\overline \cD,\cD,\sD,\overline \sD$ in order to distinguish between the density model, a countable and dense subset of it and their corresponding probability models. Following these notations $\sM=\{P=p\cdot\mu,\; p\in\cM\}$. An interval $I$ of $\R$ is said to be {\em nontrivial} if its interior $\mathring{I}$ is not empty or equivalently if its length is positive. 
When we say that {\em $p$ is a density on a (nontrivial) interval $I$}, we mean that $p$ is a density that vanishes outside $I$. The set of positive integers is denoted $\N\et$ and $|A|$ is the cardinality of a set $A$.  By convention, $\sum_{\varnothing}=0$. For an integrable function $f$ on $(\sX\on,\sA\on)$, $\E[f(\bsX)]$ is the integral of $f$ with respect to the probability measure $\bigotimes_{i=1}^{n}P_{i}\et=\gP\et$ while for $f$ on $(\sX,\sA)$ and $S\in\sP$, $\E_{S}[f(X)]$ is the integral of $f$ with respect to $S$. We use the same conventions for $\Var\left(f(\bsX)\right)$ and $\Var_{S}\left(f(X)\right)$. 
Finally, given a sequence of positive numbers $(v_{n})_{n\ge 1}$ tending to 0 as $n$ tends to infinity, we shall (abusively) say that an estimator converges at rate $v_{n}$ if its rate of convergence is $v_{n}$ or possibly faster than $v_{n}$.

\section{An $\ell$-type estimator for the TV-loss}\label{sec-ell-estimator}
Let $\overline \cM$  be a density model that satisfies our  Assumption~\ref{Ass00} for some $\cM\subset \overline \cM$. Given $P=p\cdot \mu$ and $Q=q\cdot \mu$ in $\sM$, we define
\begin{equation}\label{Non_symmetric_test_eq}
t_{(P,Q)}=\1_{q>p}-P(q>p)=P(p\ge q)-\1_{p\ge q}.
\end{equation}
Given the family $\sT=\{t_{(P,Q)},\, (P,Q)\in\sM^2\}$, we define for $P,Q\in\sM$ and  $\gx\in \sX^n$ 
\begin{equation}\label{def-TPQ}
\gT(\gx,P,Q)=\sum_{i=1}^n t_{(P,Q)}(x_i)=\sum_{i=1}^{n}\cro{\1_{q>p}(x_{i})-P(q>p)}
\end{equation}
and 
\[
\gT(\gx,P)=\underset{Q\in\sM}\sup\gT(\gx,P,Q).
\]
For  $\varepsilon>0$, we finally define our estimator  as any (measurable) element $\widehat{P}=\widehat p\cdot\mu$ that belongs to the set
\begin{equation}\label{def-EX}
\sE(\bsX)=\ac{P\in\sM,\, \gT(\bsX,P)\le \underset{P'\in\sM}\inf \gT(\bsX,P')+\varepsilon}.
\end{equation}
We call $\widehat P$ and $\widehat p$ a {\em TV-estimator} on $\sM$ and $\cM$ respectively. The parameter $\varepsilon$ is introduced in case a minimizer of $P\mapsto \gT(\bsX,P)$ does not exist on $\sM$. Any $\varepsilon$-minimizer would do provided that $\varepsilon$ is not too large.

The construction of estimators from an appropriate family of test statistics $t_{(P,Q)}$ is described in Baraud~\cite{BY-TEST} and our approach is in a similar vein. In particular, we use the following key property on the family $\sT$ (which can be compared to Assumption~1 in Baraud~\cite{BY-TEST}). 
\begin{lem}\label{bound_test_TV_lem}
For all probabilities $P,Q\in\sM$ and $S\in\sP$,
%
\begin{equation}
\dTV{P}{Q}- \dTV{S}{Q}\le \E_{S}\cro{t_{(P,Q)}(X)}\le \dTV{S}{P}.
\label{eq-SPQ}
\end{equation}
In particular, 
\begin{equation}\label{bound_test_TV_eq}
\dTV{P}{Q} - \dTV{P\et}{Q}\leq \frac{1}{n}\sum_{i=1}^n\E\cro{t_{(P,Q)}(X_i)}\leq \dTV{P\et}{P}\quad \text{where}\quad  P\et =\frac{1}{n}\sum_{i = 1}^n P_i\et.
\end{equation}
\end{lem}
However, our family $\sT$ does not satisfy the anti-symmetry assumption, namely $t_{(P,Q)}= -t_{(Q,P)}$, which is required for Baraud's construction. The risk bound that we establish below cannot therefore be deduced from Baraud~\cite{BY-TEST}. In fact, for the specific problem we want to solve here the anti-symmetry assumption can be relaxed which leads to an improvement on the numerical constants that are involved in the risk bounds. {Nevertheless, Baraud's approach and ours rely on similar heuristics. Denoting by $\overline P$ the best approximation of $P\et$ in $\sM$ with respect to the total variation distance, assuming that it exists for the sake of simplicity, we deduce from 
~\eref{bound_test_TV_eq} that the random variable $\gT(P)=\sup_{Q\in \sM}\E\cro{\gT(\bsX,P,Q)}$ satisfies 
\[
\dTV{P}{\overline P} - \dTV{P\et}{\overline P}\le \sup_{Q\in\sM}\cro{\dTV{P}{Q} - \dTV{P\et}{Q}}\leq \frac{\gT(P)}{n}\leq \dTV{P\et}{P}.
\]
These inequalities imply that 
\[
\ab{\frac{\gT(P)}{n}-\dTV{P}{\overline P}}\le \pa{\dTV{P\et}{P}-\dTV{P}{\overline P}}\vee \dTV{P\et}{\overline P}\le \dTV{P\et}{\overline P}\quad \text{for all $P\in\sM$}
\]
which means that the mappings $P\mapsto \gT(P)/n$ and $P\mapsto \dTV{P}{\overline P}$ are uniformly close over $\sM$ when $P\et$ is close enough to the model $\sM$. We therefore expect that the minimizer of the mapping $P\mapsto \dTV{P}{\overline P}$, which is $\overline P$, will be close to that of $P\mapsto \gT(P)/n$ over $\sM$. The procedure we propose here is to replace $\gT(P)$ by its empirical counterpart $\gT(\bsX,P)=\sup_{Q\in \sM}\gT(\bsX,P,Q)$.}

Our construction also shares some similarities with that proposed by Devroye and Lugosi~\cite{MR1843146}[Section 6.8 p.55]. However, a careful look at their selection criterion shows that it is slightly different from ours. They replace our function $\gT(\cdot,P,Q)$ given by~\eref{def-TPQ} by 
\[
\gT_{\text{DL}}(\cdot, P,Q):\gx\mapsto \ab{\sum_{i=1}^{n}\cro{\1_{q\ge p}(x_{i})-P(q\ge p)}}.
\]
Their approach leads to a set of estimators $\sE_{\text{DL}}(\bsX)$ defined in the same way as \eref{def-EX} for $\gT_{\text{DL}}$ in place of $\gT$ and $\varepsilon=1$. 
\begin{proof}[Proof of Lemma~\ref{bound_test_TV_lem}]
Let $P,Q\in\sM$. Using the definition \eqref{Non_symmetric_test_eq} of $t_{(P,Q)}$ and that of the TV-loss given by~\eref{TV_def_eq}, we obtain that for all $S\in\sP$,
\[
\E_{S}\cro{t_{(P,Q)}(X)}=S(q>p)-P(q>p)\leq\dTV{S}{P},
\]
which is exactly the second inequality in~\eref{eq-SPQ}. To establish the first one, we use the fact that $\dTV{P}{Q}=Q(q>p)-P(q>p)$. This leads to 
\begin{align*}
\E_{S} \cro{t_{(P,Q)}(X)}
&=S(q>p)-Q(q>p)+\cro{Q(q>p)-P(q>p)}\ge -\dTV{S}{Q}+\dTV{P}{Q} .
\end{align*} 
Finally, \eref{bound_test_TV_eq} results from the observation that
\[ \frac{1}{n} \sum_{i = 1}^n \mathbb{E} \left[ t_{(P,Q)} (X_i) \right] = \mathbb{E}_{P\et}\left[ t_{(P,Q)}(X) \right]. \]
\end{proof}

\subsection{Properties of the estimator}
As for Baraud and Birg\'e~\cite{MR3565484}, our main result is based on the key notion of \textit{extremal point} in a model. However, in our situation, this notion is different from theirs. While we require a VC condition on a class of sets, they require a VC condition on a class of functions. 
%
\begin{df}\label{df-extremal}
Let $\cF$ be a class of real-valued functions on a set $\sX$ with values in $\R$. We say that an element $\overline{f}\in\cF$ is extremal in $\cF$ (or is an extremal point of $\cF$) with degree not larger than $\frD\ge 1$ if the classes of subsets
\begin{align*}
    \sC^{>}(\cF,\overline{f}) &= \big\{\{x\in\sX\,\big|\,q(x)>\overline{f}(x)\},\; q\in\cF\setminus\{\overline{f}\}\big\}
\end{align*}
and 
\begin{align*}
    \sC^{<}(\cF,\overline{f}) &= \big\{\{x\in\sX\,\big|\,q(x) < \overline{f}(x)\},\; q\in\cF\setminus\{\overline{f}\}\big\}    
\end{align*}
are both VC  with dimension not larger than $\frD$. 
\end{df}
Additionally, we say that $\overline P$ is {\em an extremal point} of $\overline \sM$ with degree not larger than $\frD\ge 1$ if there exists $\overline{p}\in\overline \cM$ such that $\overline P=\overline{p}\cdot\mu$ and $\overline p$ is extremal in $\overline \cM$ with degree not larger than $\frD$. 
For each $\frD\ge 1$, we denote by $\overline \cO(\frD)$ the set of extremal points $\overline p$ in $\overline \cM$ with degree not larger than $\frD$, $\cO(\frD)$ a countable and dense subset of it, $\sO(\frD)$ the corresponding set of probability measures and $\overline \cO=\bigcup_{\frD\ge 1}\overline \cO(\frD)$ the set of all extremal points in $\overline \cM$. Finally, let $\cM$ be a countable and dense subset of $\overline \cM$ containing $\bigcup_{\frD\ge 1}\cO(\frD)$. 

\begin{thm}\label{shape-estimation-th}
Let $\overline \cM$ be a density model that satisfies our Assumption~\ref{Ass00} and assume that the set $\overline \cO$ is nonempty. Then any TV-estimator $\widehat P$ on $\sM$ satisfies the following properties. For all product distributions $\gP\et$ of the data and for all $\xi>0$, with a probability at least $1-e^{-\xi}$ 
\begin{equation} \label{inthm_ubd_dist_p_phat}
\dTV{P}{\widehat P} \leq 2\dTV{P \et}{P} + 20\sqrt{\frac{5\frD}{n}}+\sqrt{\frac{2(\xi +\log 2)}{n}}+\frac{\varepsilon}{n}\quad \text{for all $\frD \geq 1$ and $P \in  \overline \sO(\frD)$.}
\end{equation}
In particular,
\begin{align}
&\dTV{P \et}{\widehat P}\leq \inf_{\frD\ge 1}\cro{ 3\inf_{P\in \overline \sO(\frD)}\dTV{P \et}{P} + 20\sqrt{\frac{5\frD}{n}}} +\sqrt{\frac{2(\log 2+\xi)}{n}}+\frac{\varepsilon}{n},\label{Thm1_risk_bound_deviation_eq}
\end{align}
with the convention $\inf_{\varnothing}=+\infty$.  As a consequence of \eqref{inthm_ubd_dist_p_phat},
\begin{equation} \label{inthm_ubd_expt_dist_p_phat}
    \mathbb{E}\left[ \dTV{P}{\widehat P} \right] \leq 2\dTV{P \et}{P} + 48\sqrt{\frac{\frD}{n}} +\frac{\varepsilon}{n}\quad \text{for all $\frD \geq 1$ and $P \in  \overline \sO(\frD)$.}
\end{equation}
Moreover by \eqref{Thm1_risk_bound_deviation_eq},
%
\begin{equation}\label{Thm1_risk_bound_expecation_eq}
\mathbb{E}\cro{\dTV{P \et}{\widehat P}} \leq  \inf_{\frD\ge 1}\left\{ 3\inf_{P\in \overline \sO(\frD)}\dTV{P \et}{P} + 48\sqrt{\frac{\frD}{n}} \right\}+\frac{\varepsilon}{n}.
\end{equation}
\end{thm}
\begin{proof}
The proof is postponed to Subsection \ref{subsect-proof-ell-estimator}.
\end{proof}
Let us now comment on this result. 
In the favourable situation where the $X_{i}$ are i.i.d.\ with distribution $P\et$ in $\overline \sO$,  $\inf_{\frD\ge 1}\inf_{P\in \overline \sO(\frD)}\dTV{P \et}{P}=0$ and we deduce from~\eref{Thm1_risk_bound_expecation_eq} that the  estimator $\widehat P$ converges toward $P\et$ at the rate $1/\sqrt{n}$ for the total variation distance. More precisely, the risk of the estimator is not larger than $48\sqrt{\frD/n}+\varepsilon/n$ when $P\et$ belongs to $\overline \sO(\frD)$ for some $\frD\ge 1$. {Note that the result also holds when the data are independent only, provided that $P\et=n^{-1}\sum_{i=1}^{n}P_{i}\et$ is extremal. As we shall see, this situation may arise even when none of the marginals $P_{i}\et$ is extremal or even belongs to the model $\sM$. }
In the general case where the data are  independent only and for all $i\in\{1,\ldots,n\}$, their marginals are of the form
\begin{equation}\label{eq-marginales}
P_{i}\et=(1-\alpha_{i})\overline P+\alpha_{i} R_{i}=\overline P+\alpha_{i}\pa{R_{i}-\overline P}
\end{equation}
for some $\overline P\in \overline \sO(\frD)$ with $\frD\ge 1$, $\alpha_{1},\ldots,\alpha_{n}$ in $[0,1]$ and distributions $R_{1},\ldots,R_{n}$ in $\sP$, we deduce from~\eref{inthm_ubd_expt_dist_p_phat} that 
\begin{align*}
\mathbb{E}\cro{\dTV{\overline P}{\widehat P}}&\le 2\dTV{P \et}{\overline P} + 48\sqrt{\frac{\frD}{n}} +\frac{\varepsilon}{n}\le \frac{2}{n}\sum_{i=1}^{n}\alpha_{i}+ 48\sqrt{\frac{\frD}{n}} +\frac{\varepsilon}{n}.
\end{align*}
As compared to the previous situation where $P_{i}\et=\overline P\in \overline \sO(\frD)$, hence $\alpha_{i}=0$ for all $i$, we see that the risk bound we get only inflates by the additional term $2\overline \alpha=(2/n)\sum_{i=1}^{n}\alpha_{i}$ and it remains thus of the same order when $\overline \alpha$ is small enough as compared to $\sqrt{\frD/n}$. Note that this situation may occur even when $\alpha_{i}>0$ for all $i$, i.e.\ when possibly none of the marginals $P_{i}\et$ belong to $\overline \sO(\frD)$. In order to be more specific, we may consider the two following situations. In the first one, there exists some subset of the data which are i.i.d.\ with distribution $\overline P\in \overline \sO(\frD)$ while the other part, corresponding to what we shall call {\em outliers}, are independently drawn according to some arbitrary distributions. In this case, there exists a subset $S\subset \{1,\ldots,n\}$ such that $\alpha_{i}=1$ for $i\in S$ and $\alpha_{i}=0$ otherwise in \eref{eq-marginales}. Our procedure is stable with respect to the presence of such outliers as soon as $\overline \alpha=|S|/n$ remains small as compared to $\sqrt{\frD/n}$. In the other situation, which is called the {\em contamination} case, the data are i.i.d., a portion $\alpha\in (0,1]$ of them are drawn according to an arbitrary distribution $R$ while the other part follows the distribution $\overline P\in \overline \sO(\frD)$. Then \eref{eq-marginales} holds with $\alpha_{i}=\alpha$ and $R_{i}=R$ for all $i\in\{1,\ldots,n\}$. The risk bound we get remains stable under contamination as long as the level $\overline \alpha=\alpha$ of contamination remains small as compared to $\sqrt{\frD/n}$. 

A bound similar to \eref{Thm1_risk_bound_deviation_eq} has been established in Baraud~\cite{BY-TEST} for his $\ell$-estimators. 
His inequality~(48) can be reformulated in our context as  
\begin{align*}
\frac{1}{n} \sum_{i = 1}^n \dTV{P_i\et}{\widehat P}\le & \inf_{\frD\ge 1}\cro{ 6\inf_{P\in \overline \sO(\frD)} \cro{\frac{1}{n} \sum_{i = 1}^n \dTV{P_i \et}{P}} + 40\sqrt{\frac{5\frD}{n}}} +2\sqrt{\frac{2\xi}{n}}+\frac{2\varepsilon}{n}
-\inf_{P\in\overline \sM}\cro{\frac{1}{n} \sum_{i = 1}^n \dTV{P_i \et}{P}}.
\end{align*}
In comparison, \eqref{inthm_ubd_dist_p_phat} together with the inequalities
\[
\frac{1}{n}\sum_{i=1}^{n}d(P_{i}\et,\widehat P)\le  \frac{1}{n}\sum_{i=1}^{n}d(P_{i}\et,P)+d(P,\widehat P)\quad \text{and}\quad d(P\et,P)\le \frac{1}{n}\sum_{i=1}^{n}d(P_{i}\et,P)
\]
which hold for all $P\in \overline \cO$, imply that
\begin{align*}
\frac{1}{n} \sum_{i = 1}^n \dTV{P_i\et}{\widehat P}& \le \inf_{\frD\ge 1}\cro{ 3\inf_{P\in \overline \sO(\frD)}\cro{\frac{1}{n} \sum_{i = 1}^n \dTV{P_i \et}{P}}+ 20\sqrt{\frac{5\frD}{n}}} 
+\sqrt{\frac{2(\xi+\log 2)}{n}}+\frac{\varepsilon}{n}.
\end{align*}
If we omit the term $\inf_{P\in\overline \sM}\cro{(1/n)\sum_{i = 1}^n \dTV{P_i \et}{P}}$ that appears in his inequality and $\log 2$ that appears in ours, all the constants we get are divided by a factor 2 as compared to his.

When the Yatracos class $\{\{p>q\}, p,q\in\overline \cM\}$ is VC with dimension not larger than $\frD\ge 1$, all the elements of $\overline \cM$ are extremal with degree not larger than $\frD$ and \eref{Thm1_risk_bound_expecation_eq} becomes
\begin{align}
\E\cro{\dTV{P \et}{\widehat P}}\leq 3\inf_{P\in \overline \sM}\dTV{P \et}{P} +48\sqrt{\frac{\frD}{n}} +\frac{\varepsilon}{n}.\label{Thm1_risk_bound_deviation_eq2}
\end{align}
In the special case where the data are i.i.d.\ with distribution $P\et$, an inequality of the same flavour was established by Devroye and Lugosi~\cite{MR1843146}[Section~8.2] for their minimum distance estimate. Both inequalities involve a constant 3 in front of the approximation term $\inf_{P\in \overline\sM}\dTV{P \et}{P}$. In our inequality all the numerical constants are explicit. 

In the next sections, we take advantage of the stronger inequality~\eref{Thm1_risk_bound_deviation_eq} to consider density models $\overline \cM$ for which the Yatracos classes $\{\{p>q\}, p,q\in\overline \cM\}$ are not VC. 

\subsection{First examples based on geometric properties}
The way we have characterized extremal elements in our Definition~\ref{df-extremal}
is combinatorial in nature. We provide here a sufficient condition which is rather geometric.  To do so, we start with some notations. Given a positive integer $k$, an element $\gx=(x_{1},\ldots,x_{k})\in\sX^{k}$ and $f\in\cF$, we denote by $f_{\gx}$ and $\cF_{\gx}$ the images of $f$ and $\cF$ respectively by the mapping $f\mapsto (f(x_{1}),\ldots,f(x_{k}))$ with values in $\R^{k}$. The vectors $f_{\gx}$ and the subsets $\cF_{\gx}\subset \R^{k}$ can be viewed as finite-dimensional projections of $f$ and $\cF$ on $\R^{k}$. We recall that the border $\partial \cF_{\gx}$ of $\cF_{\gx}$ is the set of elements that belong to the closure of $\cF_{\gx}$ but not to its interior (with respect to topology induced by the Euclidean distance in $\R^{k}$). Given an element $\overline f\in\cF$, a sufficient condition for the projection $\overline f_{\gx}$ to belong to $\partial \cF_{\gx}$ is that there exists a hyperplane through $\overline f_{\gx}$ that supports $\cF_{\gx}$. More precisely, this means that we can find a vector $\gu(\overline f_{\gx})=(u_{1}(\overline f_{\gx}),\ldots,u_{k}(\overline f_{\gx}))\in \R^{k}\setminus\{\0\}$ which is orthogonal to an hyperplane that contains $\overline f_{\gx}$ and which satisfies 
\begin{equation}\label{inter-geo1}
\scal{\gu(\overline f_{\gx})}{f_{\gx}-\overline f_{\gx}}=\sum_{i=1}^{k}u_{i}(\overline f_{\gx})\pa{f(x_{i})-\overline f(x_{i})}\le 0\quad \text{for all $f\in\cF$.}
\end{equation}
When $\cF$ is convex, so is $\cF_{\gx}\subset \R^{k}$, and \eref{inter-geo1} is actually equivalent to the fact that $\overline f_{\gx}\in \partial \cF_{\gx}$. Finally, we say that a point $\gx\in\sX^{k}$, with $k\ge 2$, is {\em $\cF$-separated} if for all $i,j\in\{1,\ldots,k\}$ with $i\ne j$, there exists $f\in\cF$ such that $f(x_{i})\ne f(x_{j})$. The following result holds. 
\begin{prop}\label{Prop-geo}
Let $\cF$ be a class of functions on a set $\sX$, $\overline f$ an element of $\cF$ and $D$ a positive integer. If for all $\cF$-separated $\gx\in\sX^{D+1}$, \eref{inter-geo1} is satisfied for a vector $\gu(\overline f_{\gx})\in \R^{D+1}\setminus\{\0\}$ whose coordinates are not all of the same sign, the element $\overline f$ is extremal in $\cF$ with degree not larger than $D$.
\end{prop}
In the above condition, 0 is assumed to have the same sign as any other number. When $\cF$ is convex, Proposition~\ref{Prop-geo} tells us that $\overline f$ is extremal in $\cF$ with degree not larger than $\frD=D$ if any of its $(D+1)$-projections $\overline f_{\gx}$ belong to the border of $\cF_{\gx}$ and for each $\gx\in\R^{D+1}$ one can find a vector $\gu=\gu(\overline f_{\gx})$ whose coordinates are not all of the same sign such that the affine hyperplane $\{\overline f_{\gx}+\gv|\; \gv\in \R^{D+1}, \scal{\gv}{\gu}=0\}$ supports $\cF_{\gx}$. It is actually enough to check that the condition holds for these $\gx\in\R^{D+1}$ which are $\cF$-separated since for those which are not, i.e.\ there exist $i\ne j$ with $f(x_{i})=f(x_{j})$ for all $f\in\cF$, the set $\cF_{\gx}$ is contained in the hyperplane $\{\gy\in\R^{D+1}, y_{i}-y_{j}=0\}$ so that $\cF_{\gx}\subset \partial\cF_{\gx}$ and $\overline f_{\gx}$ automatically belongs to the border of $\cF_{\gx}$. 

The assumptions of Proposition~\ref{Prop-geo} implies, in particular,  Property (P1): 
any $(D+1)$-dimensional projection $\overline f_{\gx}$ of $\overline f$ lies at the border of the projection $\cF_{\gx}$ of $\cF$. Being extremal is, however, not equivalent to (P1) nor is it equivalent to being an extreme point when $\cF$ is convex. For example, if $\cF$ is the (convex) set of nonincreasing functions on $\sX=[0,2]$, the function $\overline f=(3/4)\1_{[0,1]}+(1/4)\1_{(1,2]}$ is extremal, as we shall see, but not extreme in $\cF$ since it can also be written as the convex combination $(1/2)\cro{\1_{[0,1]}+(1/2)\1_{[0,2]}}$ where both functions $\1_{[0,1]}$ and $(1/2)\1_{[0,2]}$  belong to $\cF$. To prove in this example that $\overline f$ is extremal with degree not larger than $\frD=2$, it suffices to argue as follows. Given three distinct points $x_{1}<x_{2}<x_{3}$, at least two consecutive ones, say $z_{1}<z_{2}$, satisfies $\overline f(z_{1})=\overline f(z_{2})$ so that $(f-\overline f)(z_{2})- (f-\overline f)(z_{1})\le 0$ for all $f\in\cF$. Inequality \eref{inter-geo1} is therefore satisfied with $\gu(\overline f_{\gx})=(-1,+1,0)$ or $\gu(\overline f_{\gx})=(0,-1,+1)$, depending on $z_{1},z_{2}\in\{x_{1},x_{2},x_{3}\}$. Proposition~\ref{Prop-geo} then applies and we conclude that $\overline f$ is extremal with degree not larger than 2.

Conversely, an extreme element of a convex set of functions, or more generally an element $\overline f\in\cF$ that only satisfies (P1), might not be extremal. If $\cF$ is the class of nonnegative functions on $\R$, the function $\overline f$ which is identically equal to 0 is extreme in the convex set $\cF$ and clearly satisfies (P1). However, $\overline f$ is not extremal in $\cF$ since the class $\sC^{>}(\cF,\overline{f})=\{\{f>0\},\; f\in\cF\}$, which is that of all the subsets of $\R$ (take $f=\1_{B}$ with $B\subset \R$), is not VC. In this example, we also note that for such a function $\overline f$ inequality \eref{inter-geo1} is satisfied for all $\gx\in\sX^{2}$ with $\gu(\overline f_{\gx})=(-1,-1)$ although $\overline f$ is not extremal with degree not larger than 1. This shows in passing that our assumption on the signs of the coordinates of $\gu(\overline f_{\gx})$ in Proposition~\ref{Prop-geo} is necessary to guarantee the fact $\overline f$ is extremal.


\begin{proof}
Let us set $k=D+1$. According to Definition~\ref{df-extremal}, if $\overline f$ were not extremal in $\cF$ with degree not larger than $\frD=D$, there would exist a subset $\{x_{1},\ldots,x_{k}\}$ of $\sX$ that can be shattered by at least one of the classes of sets $\sC^{>}(\cF,\overline{f})$ and $\sC^{<}(\cF,\overline{f}) $. We only consider the case where this class is $\sC^{>}(\cF,\overline{f})$ since the proof is the similar for $\sC^{<}(\cF,\overline{f})$.  By definition,  $\sC^{>}(\cF,\overline{f})$ shatters $\{x_{1},\ldots,x_{k}\}$ if for each subset $I\subset\{1,\ldots,k\}$, there exists $f_{I}\in\cF$ such that
\[
I=\ac{i\in \{1,\ldots,k\}, f_{I}(x_{i})-\overline f(x_{i})>0}.
\]
In particular, the point $\gx=(x_{1},\ldots,x_{k})$ is necessarily $\cF$-separated. Since \eref{inter-geo1} is satisfied for this choice of $\gx$ and the set $I\et=\{i\in\{1,\ldots,k\},\; u_{i}(\overline f_{\gx})>0\}$ is nonempty, we may write that
\[
\sum_{i\in I\et}\pa{f_{I\et}(x_{i})-\overline f(x_{i})}u_{i}(\overline f_{\gx})\le -\sum_{i\not\in I\et}\pa{f_{I\et}(x_{i})-\overline f(x_{i})}u_{i}(\overline f_{\gx})
\]
with the convention $\sum_{\varnothing}=0$. By definition of $I\et$ and $f_{I\et}$, the right-hand side of this inequality is nonpositive while the left-hand side is positive. This leads to a contradiction and shows that $\sC^{>}(\cF,\overline{f})$ cannot shatter any subset of $\sX$ with cardinality larger than $D$.
\end{proof}

Let us now take some generic example to illustrate Proposition~\ref{Prop-geo}. Let $r$ be a positive integer, $\sX$ an open subset of $\R^{l}$, $l\ge 1$, $g$ a function on $\sX$ with values in a (nontrivial) interval $I\subset \R$, and $\Phi_{r}$ the class of functions $\phi$ defined on $I$ which satisfy for all $\gv=(v_{1},\ldots,v_{r+1})\in I^{r+1}$ such $v_{1}<\ldots<v_{r+1}$,  
\begin{equation}\label{eq-condphi}
L_{\gv}(\phi)=\left|\begin{array}{lllll}
1 & v_{1}&\ldots & v_{1}^{r-1}& \phi(v_{1}) \\ 
\vdots & \vdots&\ldots & \vdots& \vdots \\ 
1 & v_{r+1}&\ldots & v_{r+1}^{r-1}& \phi(v_{r}) \\ 
\end{array}\right|\ge 0.
\end{equation}
For example $\Phi_{1}$ is the class of nondecreasing functions on $I$ while $\Phi_{2}$ is the class of convex ones. Let us now define $\cF=\{f=\phi\circ g,\; \phi\in\Phi_{r}\}$ and for $d\ge 1$, let $\cF_{d}$ be the subset of $\cF$ which consists of these functions $\overline f$ of the form $\overline \phi \circ g$ where $\overline \phi\in\Phi_{r}$ satisfies the following property: the restrictions of $\overline \phi$ to the elements of a partition of $I$ into $d$ nontrivial intervals coincide with a polynomial of degree at most $r-1$. We shall see now that a consequence of Proposition~\ref{Prop-geo} is that the elements of $\cF_{d}$ are all 
extremal in $\cF$ with degrees not larger than $\frD=rd\ge 1$. To prove this, it suffices to prove that the requirements of the proposition are satisfied, which we shall now do. Let $k=rd+1$, $\overline f=\overline \phi \circ g\in\cF_{d}$ and $\gx\in\sX^{k}$ be a $\cF$-separable point. In  particular, $g(x_{1}),\ldots,g(x_{k})$ are distinct and we may relabel the coordinates of $\gx$, if ever necessary, in such a way that $v_{1}=g(x_{1})<\ldots<v_{k}=g(x_{k})$. Since $\overline \phi$ is based on a partition of $I$ into $d$ intervals and $k=rd+1$, one of these intervals contains at least $r+1$ (consecutive) numbers $v_{i}$, i.e.\ there exists $j\in\{1,\ldots,k-r\}$ for which $\overline \phi$ is a polynomial of degree at most $r-1$ on the interval $[v_{j},v_{j+r}]\subset I$. Setting $\gv=(v_{j},\ldots,v_{j+r})\in I^{r+1}$, we deduce from~\eref{eq-condphi}, the linearity of $L_{\gv}$  and the fact that $L_{\gv}(\overline \phi)=0$ that for all $f=\phi\circ g\in\cF$
\begin{align*}
0\le L_{\gv}(\phi-\overline \phi)&=(-1)^{r}\left|\begin{array}{llll}
1 & v_{j+1}&\ldots & v_{j+1}^{r-1}\\ 
\vdots & \vdots&\ldots & \vdots \\ 
1 & v_{j+r}&\ldots & v_{j+r}^{r-1}\\ 
\end{array}\right|\pa{\phi-\overline \phi}(v_{j})+\ldots+
\left|\begin{array}{llll}
1 & v_{j}&\ldots & v_{j}^{r-1}\\ 
\vdots & \vdots&\ldots & \vdots \\ 
1 & v_{j+r-1}&\ldots & v_{j+r-1}^{r-1}\\ 
\end{array}\right|\pa{\phi-\overline \phi}(v_{j+r})\\
&=(-1)^{r}\left|\begin{array}{llll}
1 & v_{j+1}&\ldots & v_{j+1}^{r-1}\\ 
\vdots & \vdots&\ldots & \vdots \\ 
1 & v_{j+r}&\ldots & v_{j+r}^{r-1}\\ 
\end{array}\right|\pa{f-\overline f}(x_{j})+\ldots+\left|\begin{array}{llll}
1 & v_{j}&\ldots & v_{j}^{r-1}\\ 
\vdots & \vdots&\ldots & \vdots \\ 
1 & v_{j+r-1}&\ldots & v_{j+r-1}^{r-1}\\ 
\end{array}\right|\pa{f-\overline f}(x_{j+r}).
\end{align*}
Since $v_{j}<\ldots<v_{j+r}$, these Vandermond determinants are all positive, $-L_{\gv}(\phi-\overline \phi)$ can be written in the form \eref{inter-geo1} where the coordinates of $\gu(\overline f_{\gx})$ are all zero except for those with indices $j,\ldots,j+r$. Besides, the signs of these nonzero coordinates are not constant. The requirements of Proposition~\ref{Prop-geo} are satisfied and the elements of $\cF_{d}$ are therefore extremal with degree at most $rd$.

Let us now see some of the consequences of this result and some preliminary applications of our Theorem~\ref{shape-estimation-th}. When $\mu$ is the Lebesgue measure on $\sX$, $\L_{1}(\sX,\sB(\sX),\mu)$ is separable, Assumption~\ref{Ass00} is therefore automatically satisfied and our Theorem~\ref{shape-estimation-th} applies to the subset $\overline \cM$ of $\cF$ that consists of densities. The set $\overline \cO\subset \overline \cM$ that consists of those extremal ones therefore contains $\bigcup_{d\ge 1}\pa{\cF_{d}\cap \overline \cM}$. For example, in dimension 1, when $r=1$, $\sX=(0,+\infty)$, $I=(-\infty,0)$ and $g:x\mapsto -x$, the set $\overline \cM$ is that of all nonincreasing densities on $\sX$ and $\overline \cO$ contains those which are piecewise constant on a partition of $\sX$ into a finite number of (nontrivial) intervals. When $r=2$, $I=\sX$ is a (nontrivial) interval of $\R$ and $g$ the identity function on $I$, the set $\overline \cM$ consists of the convex densities on $\sX$ and $\overline \cO$ contains those which are piecewise affine on a partition of $\sX$ into a finite number of (nontrivial) intervals. Denoting by $\scal{\cdot}{\cdot}$ the Euclidean inner product of $\R^{l}$ and $\ab{\cdot}$ the corresponding norm, 
our example also includes in higher dimensions, models of densities on a compact domain of $\R^{l}$ which are of the form $\gy\mapsto \phi\pa{\scal{\gy}{\gt}}$, for a given $\gt$ in the unit sphere of $\R^{l}$, or of the form  $\gy\mapsto \phi\pa{\ab{A\gy}}$ for a given $(l\times l)$-matrix $A$, taking for $g$ the mappings $\gy\mapsto \scal{\gy}{\gt}$ and $\gy\mapsto A\gy$ respectively and functions $\phi$ that run among a set of monotone, convex or concave functions on the line. 

Nevertheless, these examples may appear too restrictive for practical purposes, especially in dimensions larger than 1. These restrictions are due to the fact that the geometric-type assumption required in Proposition~\ref{Prop-geo} is fulfilled for specific convex sets only although we typically wish to consider density models which are not convex. Fortunately, this assumption is sufficient but not necessary and our Definition~\ref{df-extremal} allows us to find extremal elements in sets that are possibly non-convex. The most interesting ones, nevertheless, arise in dimension $l=1$. In the remaining part of this paper we shall therefore restrict ourselves to the problem of estimating a density (with respect to the {\em Lebesgue measure}) on the line. Each of the forthcoming sections is devoted to  the problem of estimating densities or a mixture of densities that are monotone, concave, convex, $\log$-concave or more generally $s$-concave. None of these density models is convex.

\section{Estimating a piecewise monotone density}\label{sect-kmono}
We denote by $\cA(k)$ the class of nonempty subsets $A\subset \R$ with cardinality not larger than $k\ge 1$. The elements of $A$ provide a partition of $\R$ into $l=|A|+1\le k+1$ intervals $I_{1},\ldots,I_{l}$ the endpoints of which belong to $A$. We denote by $\gI(A)$ the set $\{\mathring{I}_{1},\ldots,\mathring{I}_{l}\}$ of their interiors. Although there exist several ways of partitioning $\R$ into intervals with endpoints in $A$, the set $\gI(A)$ is uniquely defined.  

\subsection{Piecewise monotone densities}

\begin{df}
Let $k\ge 2$. A function $g$ on $\R$ is said to be $k$-piecewise monotone if there exists $A\in\cA(k-1)$ such that the restriction of $g$ to each interval $I\in\gI(A)$ is monotone. In particular, there exist at most $k$ monotone functions $g_{I}$ on $I\in \gI(A)$ such that  
\[
g(x)=\sum_{I\in \gI(A)}g_{I}(x)\1_{I}(x)\quad \text{for all $x\in \R\setminus A$.}
\]
\end{df}
A $k$-piecewise monotone function $g$ associated to $A\in\cA(k-1)$ may not be monotone on each element of a partition based on $A$. We only require that $g$ be monotone on the interiors of these elements. The function  
\begin{displaymath}
g\,:\,\Bigg|
\begin{array}{rcl}
\R&\longrightarrow&\R_+\vspace{2pt}\\
x&\longmapsto &\frac{1}{\sqrt{|x|}}\1_{|x|>0}
\end{array}
\end{displaymath}
is 2-piecewise monotone, associated to $A=\{0\}\in\cA(1)$, but $g$ is neither monotone on $(-\infty, 0]$ nor on $[0,+\infty)$.

We denote by $\overline \cM_{k}$ the set of $k$-piecewise monotone densities. These sets obviously satisfy 
$\overline \cM_{j}\subset \overline \cM_{k}$ for $j\le k$ and form thus an increasing sequence with respect to set inclusion. The set $\overline \cM_{2}$ contains the unimodal densities on the line and in particular all the densities that are monotone on a half-line and vanish elsewhere. 

Of special interest are those densities in $\overline \cM_{k}$ which are also piecewise constant on a finite partition of $\R$ into intervals. More precisely, for $D\ge 1$ let $\overline \cO_{D,k}$ be the subset of $\overline \cM_{k}$ that consists of those densities that are constant on each element of a class $\gI(A)$ with $A\in\cA(D+1)$. The number $D$ is a majorant of the number of bounded intervals in the class $\gI(A)$, hence on the number of positive values that a density in $\overline \cO_{D,k}$ may take. The uniform distribution on a nontrivial interval has a density that belongs to $\overline \cO_{1,2}$, and also to $\overline \cO_{D,k}$ for all $D\ge 1$ and $k\ge 2$.  The sets  $\overline \cO_{D,k}$ with $D\ge 1$ and $k\ge 2$ are therefore nonempty. They satisfy the following property which is a consequence of Proposition~3 of Baraud and Birg\'e~\cite{MR3565484}.
\begin{prop}\label{prop-VC-monotone}
For all $D\ge 1$ and $k\ge 2$, all the elements of $\overline \cO_{D,k}$ are extremal in $\overline \cM_{k}$ with degree not larger than $3(k+D+1)$.
\end{prop}

For all $D\ge 1$ and $k\ge 2$, let $\cO_{D,k}$ be a countable and dense subset of $\overline \cO_{D,k}$ (for the $\L_{1}$-norm) and $\cM_{k}$ a countable and dense subset of $\overline \cM_{k}$ that contains $\bigcup_{D\ge 1}\cO_{D,k}$. It follows from Proposition~\ref{prop-VC-monotone} that the elements of $\cO_{D,k}$ are also extremal in $\cM_{k}$ with degree not larger than $3(k+D+1)$ for all $D\ge 1$. We immediately deduce from Theorem~\ref{shape-estimation-th} the following result. 

\begin{thm}\label{thm-monotone}
Let $k\ge 2$. For every product distribution $\gP\et$ of the data, a TV-estimator $\widehat P$ on $\sM_{k}$ satisfies 
\begin{equation}\label{eq-Thm2}
\mathbb{E}\cro{\dTV{P \et}{\widehat P}} \leq  \inf_{D\ge 1}\cro{3\inf_{P\in\overline \sO_{D,k}}\dTV{P \et}{P} + 83.2\sqrt{\frac{D+k+1}{n}}}+\frac{\varepsilon}{n}.
\end{equation}
\end{thm}

Our approach solves the problem of estimating a nonincreasing density on a half-line $I$ by taking $k=2$. For this specific problem, a bound of the same flavour was established by Baraud~\cite{BY-TEST} (see Proposition~6) for his $\ell$-estimator. When the data are i.i.d.\ with distribution $P\et$, our result shows better constants in the risk bound~\eref{eq-Thm2} as compared to Baraud's.
In particular, the constant 3 in front of the approximation term $\inf_{P\in\overline \sO_{D,k}}\dTV{P \et}{P}$ improves on his constant 5. Our estimator (as well as Baraud's) also improves on the Grenander estimator since our construction does not require the prior knowledge of the half-line $I$. 

For general values of $k\ge 2$, the problem of estimating a $k$-piecewise monotone density was also considered in Baraud and Birg\'e~\cite{MR3565484}. The authors used a $\rho$-estimator and the squared Hellinger loss in place of the total variation one to evaluate its risk--- see their Corollary~2. Their bound is similar to ours except for the fact that the quantity $\sqrt{(D+k+1)/n}$ appears there multiplied by a logarithmic factor. This logarithmic factor turns out to be necessary when one deals with the squared Hellinger loss while it disappears with  the total variation one. 

Up to the additional term $\eps/n$, the risk bound that we establish for the TV-estimator is given by the quantity
\[
A=\inf_{D\ge 1}\cro{3\inf_{P\in\overline \sO_{D,k}}\dTV{P \et}{P} + 83.2\sqrt{\frac{D+k+1}{n}}}.
\]
If we apply a location-scale transformation to the data $\etc{X}$, that is, if in place of the original data $X_{1},\ldots,X_{n}$ we observe the random variables $Y_{i}=\sigma X_{i}+m$ for $i\in\{1,\ldots,n\}$ with $(m,\sigma)\in\R\times (0,+\infty)$, the marginal distributions $P_{i}\et$ of the $X_{i}$ are then replaced by their images $Q_{i}\et$ by the mapping $x\mapsto  \sigma x+m$. This mapping leaves  the sets $\overline \sO_{D,k}$ and the total variation distance between two probabilities invariant. The quantity $A$ remains thus unchanged under such a transformation and the performance of the TV-estimator is therefore independent of the unit that is used to measure the data. 

Another consequence of Theorem~\ref{thm-monotone} is the following one. If the $X_{i}$ are not i.i.d.\ but only independent and their distributions are close enough to a probability $\overline P$ which admits a density $\overline p$ with respect to $\mu$, by using the triangle inequality we may bound $A$  by $3\dTV{P \et}{\overline P}+\B_{k,n}(\overline p)$ where $\B_{k,n}(p)$ is defined for a density $p$ on the line by 
\begin{align}
\B_{k,n}(p)&=\inf_{D\ge 1}\cro{3\inf_{Q\in\overline \sO_{D,k}}\dTV{P}{Q} + 83.2\sqrt{\frac{D+k+1}{n}}}\nonumber\\
&=\inf_{D\ge 1}\cro{\frac{3}{2}\inf_{q\in\overline \cO_{D,k}}\norm{p-q}+ 83.2\sqrt{\frac{D+k+1}{n}}}.\label{eq-defB}
\end{align}
This means that as long as $3\dTV{P \et}{\overline P}$ remains small as compared to $\B_{k,n}(\overline p)$, 
the bound on $\E[d(P \et,\widehat P)]$ would be almost the same as if the $X_{i}$ were truly i.i.d.\ with distribution $\overline P$. This property accounts for the robustness of our approach. 

In the remaining part of this section, our aim is to specify the order of magnitude of the quantity $\B_{k,n}(p)$ under different {\em a posteriori} assumptions on the density $p$. Note that the quantity $\B_{k,n}(p)$ also remains invariant under a location-scale transformation of the density $p$. 

\subsection{Estimation of bounded and compactly supported $k$-piecewise monotone densities}

For $k\geqslant 3$, let $\overline\cM_{k}^{\infty}$ be the subset of $\overline \cM_{k}$ that consists of the densities on $\R$ which coincide almost everywhere with a density of the form
\begin{equation}\label{def-mbark}
p=\sum_{i=1}^{k-2}w_{i}p_{i}\1_{(x_{i-1},x_{i})}
\end{equation}
where 
\begin{itemize}
\item[(i)] $(x_{i})_{i\in\{0,\ldots,k-2\}}$ is an increasing sequence of real numbers;
\item[(ii)] $w_{1},\ldots,w_{k-2}$ are nonnegative numbers such that $\sum_{i=1}^{k-2}w_{i}=1$;
\item[(iii)] for $i\in\{1,\ldots,k-2\}$, $p_{i}$ is a monotone density on the interval $I_{i}=(x_{i-1},x_{i})$ of length $L_{i}>0$ with variation  
\[
V_{i}=\sup_{x\in I_{i}}p_{i}(x)-\inf_{x\in I_{i}}p_{i}(x)<+\infty.
\]
\end{itemize}
A density $p$ in $\overline\cM_{k}^{\infty}$ is necessarily bounded and compactly supported. A monotone density on $\R_{+}$ which is bounded and compactly supported belongs to $\overline \cM_3^{\infty}$. A bounded unimodal density supported on a compact interval belongs to $\overline \cM_4^{\infty}$. 

For $p\in\overline\cM_{k}^{\infty}$, we set 
\begin{equation}\label{def-Rp}
\bs{R}_{k,0}(p)=\inf \cro{\sum_{i=1}^{k-2}\sqrt{w_{i}\log\pa{1+L_{i}V_{i}}}}^{2}
\end{equation}
where the infimum runs among all ways of writing $p$ in the form~\eref{def-mbark} a.e. Note that we allow some of the $w_{i}$ to be zero in which case the corresponding densities $p_{i}$ may be chosen arbitrarily and their choices do not contribute to the value of $\bs{R}_{k,0}(p)$. For $k\ge 3$ and $R>0$, let $\overline \cM_{k}^{\infty}(R)$ be the subset of $\overline\cM_{k}^{\infty}$ that consists of these densities $p$ for which $\bs{R}_{k,0}(p)< R$. When a density $p$ belongs to  $\overline\cM_{k}^{\infty}(R)$, we may therefore write $p$ in the form~\eref{def-mbark} a.e.\ with  $L_{i}$ and $V_{i}$ such that $\cro{\sum_{i=1}^{k-2}\sqrt{w_{i}\log\pa{1+L_{i}V_{i}}}}^{2}<R$. The sequences of sets $(\overline\cM_{k}^{\infty})_{k\ge 3}$ and $(\overline\cM_{k}^{\infty}(R))_{k\ge 3}$ are both increasing with respect to set inclusion.  
To see that a density $p\in \overline\cM_{l}^{\infty}(R)$ belongs to $p\in \overline\cM_{k}^{\infty}(R)$ when $l<k$, it suffices to introduce $k-l$ extra numbers $x_{l-1}<\ldots <x_{k-2}$ which are larger than $x_{l-2}$ and for $i\in\{l-2,\ldots, k-1\}$, consider the uniform density $p_{i}$ on $(x_{i},x_{i+1})$ with the weight $w_{i}=0$. 
%
%
%
It is not difficult to check that the set $\overline \cM_{k}^{\infty}(R)$ is invariant under a location-scale transformation.

\begin{thm}\label{thm-3}
Let $k\geqslant 3$ and $R>0$. If $p\in\overline \cM_{k}^{\infty}(R)$,  
\begin{equation}\label{eq-thm-3}
\B_{k,n}(p)\le 41.3\pa{\frac{R}{n}}^{1/3}+83.2\sqrt{\frac{2k}{n}}.
\end{equation}
\end{thm}

This result is to our knowledge new in the literature.  We deduce that the minimax risk for the $\L_{1}$-norm over $\overline \cM_{k}^{\infty}(R)$ is not larger than $(R/n)^{1/3}\vee (k/n)^{1/2}$ up to a  positive multiplicative constant. For large enough values of $n$, the bound is of order $(R/n)^{1/3}$ while for moderate ones and values of $R$ which are close enough to 0, which means that the densities in $\overline \cM_{k}^{\infty}(R)$ are close to a mixture of $k-2$ uniform distributions, the bound is of order $\sqrt{k/n}$.  

It is interesting to compare this result to that established in Baraud and Birg\'e~\cite{MR3565484}[page 3900] for their $\rho$-estimators. For the problem of estimating a bounded unimodal density $p$ supported on an interval of length $L$, which is an element of $\overline \cM_{4}$, Baraud and Birg\'e show that the Hellinger risk of the $\rho$-estimator is not larger than $(\sqrt{L\norm{p}_{\infty}}/n)^{1/3}\log n$ up to some numerical constant. With our TV-estimator, the bound we get is of order  
$(\log(1+L\norm{p}_{\infty})/n)^{1/3}$ and therefore only depends logarithmically on the quantity $L\norm{p}_{\infty}$ (which is not smaller than 1 since $p$ is a density).  

\begin{proof}
The proof of Theorem \ref{thm-3} is based on \eqref{eq-defB} and the following approximation result. The complete proof is postponed to subsection \ref{subsect-proof-kmono}.
\end{proof}

\begin{prop}\label{prop-birge}
Let $V\ge 0$ and $I$ be a bounded interval of length $L>0$. For all $D\ge 1$, there exists a partition $\cJ=\cJ(D,L,V)$ of $I$ into $D\ge 1$ nontrivial intervals with the following properties. 
For any monotone density $p$  on $I$ for which 
\[
V_{I}(p)=\sup_{x\in \mathring I}p(x)-\inf_{x\in \mathring I}p(x)\le V, 
\]
the function $\overline p=\overline p(\cJ)$ defined by 
\[
\overline p=\sum_{J\in\cJ}\overline p_{J}\1_{J}\quad \text{with}\quad \overline p_{J}=\frac{1}{\mu(J)}\int_{J}p\;d\mu, 
\]
is a monotone density on $I$ that satisfies
%
\begin{equation}\label{eq-lem-birge}
\int_{I}\ab{p-\overline p}d\mu\le \cro{(1+VL)^{1/D}-1}\wedge 2\le \frac{2\log\pa{1+VL}}{D}.
\end{equation}
\end{prop}
\begin{proof}
The proof is postponed to subsection \ref{subsect-proof-kmono}.
\end{proof}

Although the result is hidden in his calculations, Birg\'e~\cite{MR902242} has established a bound of the same flavour where the variation $V$ is replaced by a uniform bound on $p$. Unlike his, our bound \eref{eq-lem-birge} allows to recover the fact that when $V=0$, i.e. when $p$  is constant on $I$, the left-hand side equals 0 as expected. 
The combination of Theorem~\ref{thm-monotone} and Theorem~\ref{thm-3}  immediately leads to the following 
\begin{cor}\label{cor-monotone}
Let $k\ge 3$. If $X_1,\dots,X_n$ are i.i.d.\ with a density $p\in\overline \cM_{k}^{\infty}(R)$, then the TV-estimator $\widehat p$ on $\cM_{k}$ satisfies 
\begin{equation}\label{eq-cor-mon1}
\mathbb{E}\cro{\|p-\widehat{p}\|} \leq 82.6\pa{\frac{R}{n}}^{1/3}+166.4\sqrt{\frac{2k}{n}}+\frac{2\varepsilon}{n}.
\end{equation}
\end{cor}

\subsection{Estimation of other $k$-piecewise monotone densities}
Corollary~\ref{cor-monotone} provides an upper bound on the $\L_{1}$-risk of the TV-estimator  for estimating a density $p\in \overline \cM_{k}^{\infty}$. 
A natural question is how the estimator performs when the density $p$ is neither bounded nor supported on a compact interval. Since for such densities we may write
\begin{equation}\label{eq-fondmono}
\B_{k,n}(p)\le \inf_{\overline p\in \overline \cM_{k}^{\infty}}\cro{\frac{3}{2}\norm{p-\overline p}+\B_{k,n}(\overline p)},
\end{equation}
an upper bound on $\B_{k,n}(p)$ can be obtained by combining Theorem~\ref{thm-3} with an approximation result showing how general densities in $\overline \cM_{k}$ can be approximated by elements of $\overline \cM_{k}^{\infty}$. In this section, we therefore study the approximation properties of the set $\overline \cM_{k}^{\infty}$ with respect to possibly unbounded and non-compactly supported densities. We start with the case of a monotone density on a half-line and introduce the following definitions. 

\begin{df}\label{def-conjugate}
Given a nonincreasing density $p$ on $(a,+\infty)$ with $a\in\R$, we define $\widetilde p$ as the mapping on $(0,+\infty)$ given by    
\begin{equation}\label{def-ptilde}
\widetilde p(y)=\inf\ac{x>0,\; p(a+x)< y}\ge 0.
\end{equation}
%
We define the $x$-{\em tail function} $\tau_{x}(p,\cdot)$ associated to $p$ as 
\begin{displaymath}
\tau_x(p,\cdot)\,:\,\left|
\begin{array}{rcl}
[0,+\infty)&\longrightarrow&\R_{+}\vspace{2pt}\\
t&\longmapsto &\dps{\int_{t}^{+\infty}p(a+x)d\mu(x)}
\end{array}\right.
\end{displaymath}

the $y$-{\em tail function} $\tau_{y}(p,\cdot)$ as 
\begin{displaymath}
\tau_y(p,\cdot)\,:\,\left|
\begin{array}{rcl}
[0,+\infty)&\longrightarrow&\R_{+}\vspace{2pt}\\
t&\longmapsto &\dps{\int_{t}^{+\infty}\widetilde p(y)d\mu(y)}
\end{array}\right.
\end{displaymath}
and the {\em tail function} $\tau(p,\cdot)$ as 
\begin{equation}\label{eq-tail}
\tau(p,t)=\inf_{s>0}\cro{\tau_{x}(p,st)+\tau_{y}(p,p(a+s))}\quad \text{for all $t\ge 1$}.
\end{equation}
When $p$ is a nondecreasing density on $(-\infty,-a)$ with $a\in\R$, we define $\tau_{x}(p,\cdot)$, $\tau_{y}(p,\cdot)$ and $\tau(p,\cdot)$ as respectively the $x$-tail, $y$-tail and tail functions of the nonincreasing density $x\mapsto p(-x)$ .
\end{df}
Let us comment on these definitions. When $p$ is a continuous decreasing density from $(0,+\infty)$ onto $(0,+\infty)$, $\widetilde p$ is the reciprocal function $p^{-1}$. By reflecting the graph of $x\mapsto p(x)$ in the line $y=x$, we easily see that  $\widetilde p=p^{-1}$ is a nonincreasing density on $(0,+\infty)$. This property remains true in the general case as shown by the lemma below with the special value $B=0$. As a consequence, $\tau_{y}(p,\cdot)$ can be interpreted as the tail of the distribution function associated to the density $\widetilde p$ while $\tau_{x}(p,\cdot)$ is that of $p(a+\cdot)$. Figure~1 displays a graphical representations of the quantities $\tau_{x}(p,st)$ and $\tau_{y}(p,p(a+s))$ for $s>0$ and $t>1$.
\begin{lem}\label{lem-ty}
Let $p$ be a nonincreasing density on $(a,+\infty)$ with $a\in\R$ and $\widetilde p$ the mapping defined by \eref{def-ptilde}. For all $B\ge 0$, 
\begin{equation}\label{lem-ty-eq}
\int_{0}^{+\infty}\cro{p(a+x)-B}_{+}d\mu(x)=\int_{B}^{+\infty}\widetilde p(y)d\mu(y)=\tau_{y}(p,B).
\end{equation}
\end{lem}

By changing $p$ into $x\mapsto p(-x)$, we also obtain that when $p$ is nondecreasing on $(-\infty,-a)$, 
\[
\int_{-\infty}^{0}\cro{p(-a+x)-B}_{+}d\mu(x)=\int_{B}^{+\infty}\widetilde p(y)d\mu(y)=\tau_{y}(p,B)\quad \text{for all $B\ge 0$.}
\]
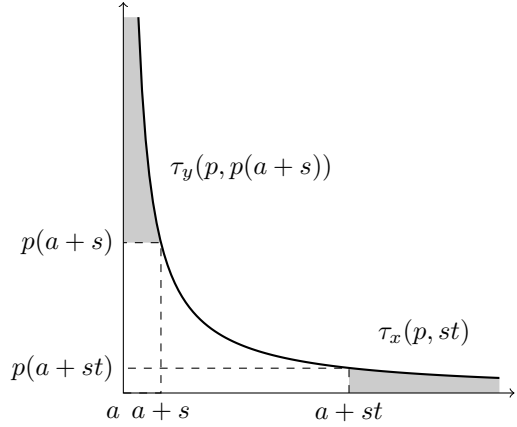
\begin{figure}
\begin{tikzpicture}
	\fill[black!20, domain=3:5, variable=\x]
 	 	plot ({\x}, {1/\x})
  	-- (5,0)
  	-- (3,0)
  	-- cycle;
	\draw[dashed, black] (3,0) node[below] {${\textcolor{black}{a+st}}$} -- (3,.33) -- (0,.33) node[left] {\textcolor{black}{$p(a+st)$}};
	\draw[black] (4,.5) node[above] {$\tau_x(p,st)$};
	
	\fill[black!20, domain=.2:.5, variable=\x]
 	 	plot ({\x}, {1/\x})
  	-- (0,2)
  	-- (0,5)
  	-- cycle;
	\draw[dashed,black] (0,0) node[below] {${\textcolor{black}{a \phantom{+}}}$} -- (.5,0) node[below] {${\textcolor{black}{a+s}}$} -- (.5,2) -- (0,2) node[left] {$\textcolor{black}{p(a+s)}$ };
	\draw[black] (.5,3) node[right] {$\tau_y(p,p(a+s))$};
	\draw[thick, black, domain=.2:5, samples=100] plot ({\x}, {1/\x});
	\draw[<->] (5.2,0) -- (0,0) -- (0,5.2);
\end{tikzpicture}
\caption{The quantities $\tau_{x}(p,st)$ and $\tau_{y}(p,p(a+s))$ (dark areas) for $s>0$ and $t>1$.}
\end{figure}

\begin{proof}
Let $y>0$. Since $p$ is a nonincreasing density on $(a,+\infty)$, it necessarily tends to 0 at $+\infty$. The set $I(y)=\ac{x>0,\; p(a+x)<y}$ is therefore a nonempty unbounded interval with endpoint $\widetilde p(y)<+\infty$ by definition of $\widetilde p(y)$. In particular, 
\[
(\widetilde p(y),+\infty)\subset I(y)\subset [\widetilde p(y),+\infty),
\]
and by taking the complements of those sets we obtain that for all $(x,y)\in(0,+\infty)\times (0,+\infty)$
\[
\1_{x<\widetilde p(y)}\le \1_{p(a+x)\ge y}\le \1_{x\le \widetilde p(y)}.
\]
Integrating these inequalities on $(0,+\infty)\times (B,+\infty)$ with respect to $\mu\otimes\mu$ and using Fubini's theorem, we obtain that 
\begin{align*}
&\int_{B}^{+\infty}\widetilde p(y)d\mu(y)\\
&=\int_{B}^{+\infty}\cro{\int_{0}^{+\infty}\1_{x<\widetilde p(y)}d\mu(x)}d\mu(y)\le \int_{B}^{+\infty}\cro{\int_{0}^{+\infty} \1_{p(a+x)\ge y}d\mu(x)}d\mu(y)\\
&=\int_{0}^{+\infty}\cro{\int_{B}^{+\infty} \1_{p(a+x)\ge y}d\mu(y)}d\mu(x)=\int_{a}^{+\infty}\cro{p(a+x)-B}_{+}d\mu(x)\\
&\le  \int_{B}^{+\infty}\cro{\int_{0}^{+\infty} \1_{x\le \widetilde p(y)}d\mu(x)}d\mu(y)=\int_{B}^{+\infty}\widetilde p(y)d\mu(y),
\end{align*}
which proves~\eref{lem-ty-eq}.
\end{proof}

It follows from \eqref{lem-ty-eq} that if $p$ is a nonincreasing density on $(a,+\infty)$, $\tau(p,t)$ also writes for all $t\ge 1$ as 
\[
\tau(p,t)=\inf_{s>0}\cro{\int_{st}^{+\infty}p(a+x)d\mu(x)+\int_{0}^{s}\cro{p(a+x)-p(a+s)}d\mu(x)}.
\]
It is not difficult to check that the mapping $\tau(p,\cdot)$ is nonincreasing on $[1,+\infty)$, tends to 0 at $+\infty$ and is invariant under a location-scale transformation, i.e.\  by changing $p$ into the density $\sigma^{-1}p[\sigma^{-1}(\cdot-m)]$ on $(\sigma a+m,+\infty)$ with $m\in\R$ and $\sigma>0$.

We consider the general situation where $p$ is an arbitrary element of $\overline \cM_{\ell}$ with $\ell\ge 2$. Changing the values of $p$  on a negligible set, if ever necessary, which will not change the way it can be approximated in $\L_{1}$-norm, we may assume without loss of generality that it can be written in the form: 
\begin{equation}\label{eq-formp00}
p=w_{1}p_{1}\1_{(-\infty,x_{1})}+\sum_{i=2}^{\ell-1}w_{i}p_{i}\1_{(x_{i-1},x_{i})}+w_{\ell}p_{\ell}\1_{(x_{\ell-1},+\infty)}
\end{equation}
where $x_{1}<\ldots<x_{\ell-1}$ is an increasing sequence of real numbers and $p_{1},\ldots,p_{\ell}$, $w_{1},\ldots,w_{\ell}$ are defined as follows. The functions $p_{1}$ and $p_{\ell}$ denote monotone densities on $(-\infty, x_{1})$ and $(x_{\ell-1},+\infty)$ respectively, $w_{1}=\int_{-\infty}^{x_{1}}pd\mu$, $w_{\ell}=\int_{x_{\ell}}^{+\infty}pd\mu$ and when $\ell>2$, the $p_{i}$ are monotone densities on $(x_{i-1},x_{i})$ and $w_{i}=\int_{x_{i-1}}^{x_{i}}pd\mu$ for all $i\in\{2,\ldots,\ell-1\}$. For $p$ written in the form~\eref{eq-formp00}, we set 
\begin{equation}\label{eq-def-tau-infty}
\tau_{\infty}(p,t)=\max_{i\in\{1,\ldots,\ell\}}\tau\pa{p_{i},t}\quad \text{for all $t\ge 1$.}
\end{equation}
The mapping $t\mapsto \tau(p,t)$ is nonincreasing on $\R_{+}$ and tends to 0 at $+\infty$. 
\begin{thm}\label{thm-monogene}
Let $\ell\ge 2$, $k\ge 2\ell$ and $R\ge \ell\log 2$. If $p$ is a density of the form~\eref{eq-formp00} a.e., 
\begin{equation}\label{thm-monogene-eq0}
\inf_{\overline p\in \overline \cM_{k}^{\infty}(R)}\norm{p-\overline p}\le 2\tau_{\infty}\pa{p,\exp\pa{\frac{R}{\ell}}-1}.
\end{equation}
\end{thm}
\begin{proof}
The proof  is postponed to Subsection \ref{subsect-proof-kmono}.
\end{proof}

\noindent
By combining Theorem \ref{thm-3} and Theorem \ref{thm-monogene} we obtain the following corollary.

\begin{cor}\label{cor-monogene}
Let $\ell\ge 2$ and $k\ge 2\ell$. If $p$ is a density of the form~\eref{eq-formp00} a.e.,
\begin{equation}\label{cor-monogene-eq0}
\B_{k,n}(p)\le 44.3\pa{\frac{\ell \log(1+r_{n})}{n}}^{1/3}+83.2\sqrt{\frac{2k}{n}}
\end{equation}
where 
\begin{equation}\label{def-rn}
r_{n}=\inf\ac{t\ge 1,\; \tau_{\infty}(p,t)\le \pa{\frac{\ell \log(1+t)}{n}}^{1/3}}
\end{equation}
and $\tau_{\infty}(p,\cdot)$ is defined by~\eref{eq-def-tau-infty}. Then, the TV-estimator $\widehat p$ on $\cM_{k}$ satisfies 
\begin{equation}\label{eq-cor-mon2}
\mathbb{E}\cro{\|p-\widehat{p}\|} \leq 88.6\pa{\frac{\ell \log(1+r_{n})}{n}}^{1/3}+166.4\sqrt{\frac{2k}{n}}+\frac{2\varepsilon}{n}.
\end{equation}
\end{cor}

\begin{proof}
Since $\tau_{\infty}(p,\cdot)$ tends to 0 at $+\infty$, the set 
\[
\cR=\ac{t\ge 1,\; \tau_{\infty}(p,t)\le \pa{\frac{\ell \log(1+t)}{n}}^{1/3}}
\]
is nonempty,  $r_{n}$ is well-defined and for all $t>r_{n}\ge 1$, 
\[
\tau_{\infty}(p,t)\le \pa{\frac{\ell \log(1+t)}{n}}^{1/3}.
\]
Using \eref{eq-fondmono}, Theorems~\ref{thm-3} and~\ref{thm-monogene} with $R=\ell\log(1+t)>\ell \log 2$ we obtain that
\begin{align*}
\B_{k,n}(p)&\le \frac{3}{2}\inf_{\overline p\in \overline \cM_{k}^{\infty}(R)}\norm{p-\overline p}+\sup_{\overline p\in \overline \cM_{k}^{\infty}(R)}\B_{k,n}(\overline p)\\
&\le 3\tau_{\infty}\pa{p,t}+41.3\pa{\frac{\ell \log(1+t)}{n}}^{1/3}+83.2\sqrt{\frac{2k}{n}}\\
&\le 44.3\pa{\frac{\ell \log(1+t)}{n}}^{1/3}+83.2\sqrt{\frac{2k}{n}},
\end{align*}
and the result follows from the fact that $t$ is arbitrary in $(r_{n},+\infty)$. 
\end{proof}

\begin{exa}
Let $n\ge 2$, $\alpha\ge 0$, $\beta\ge -1$, $\gamma\in (0,1)$ and $q$ be the mapping defined by 
\[
q(x)=\frac{2^{1-\gamma}}{x^{1-\gamma}}\1_{(0,2)}(x)+\frac{2^{1+\alpha}\pa{\log 2}^{1+\beta}}{x^{1+\alpha}\pa{\log x}^{1+\beta}}\1_{[2,+\infty)}(x).
\]
When $(\alpha,\beta)\in (0,+\infty)\times [-1,+\infty)$ and when $\alpha=0$ and $\beta>0$, $q$ is a positive, integrable, nonincreasing function on $(0,+\infty)$ and we may denote by $p$ the corresponding density, i.e.\ $p=cq$ for some $c>0$ depending on $\alpha,\beta$ and $\gamma$. The density $p$ may be written in the form~\eref{eq-formp00} with $\ell=2$, $w_{1}=0$, $w_{2}=1$, $x_{1}=0$ and $p_{2}=p$. Throughout this example, $C$ denotes a positive number depending on $\alpha,\beta$ and $\gamma$ that may vary from line to line. 

It follows from Definition~\ref{def-conjugate} with $a=0$ that when $\alpha>0$, for all $t\ge 2$  
\begin{align*}
\frac{\tau_{x}(p,t)}{c2^{1+\alpha}\pa{\log 2}^{1+\beta}}&=\int_{t}^{+\infty}\frac{dx}{x^{1+\alpha}\pa{\log x}^{1+\beta}}=\int_{\log t}^{+\infty}\frac{e^{-\alpha s}}{s^{1+\beta}}ds\le \frac{1}{\alpha t^{\alpha}\pa{\log t}^{1+\beta}},
\end{align*}
and when $\alpha=0$ and $\beta>0$
\begin{align*}
\frac{\tau_{x}(p,t)}{c2^{1+\alpha}\pa{\log 2}^{1+\beta}}&=\int_{\log t}^{+\infty}\frac{1}{s^{1+\beta}}ds= \frac{1}{\beta\pa{\log t}^{\beta}}.
\end{align*}
For $y>c$, $\widetilde p:y\mapsto 2(c/y)^{1/(1-\gamma)}$, hence
\[
\tau_{y}(p,t)=\int_{t}^{+\infty}\widetilde p(y)d\mu(y)=\frac{2(1-\gamma)c^{1/(1-\gamma)}}{\gamma t^{\gamma/(1-\gamma)}}\quad \text{for all $t\ge c$.}
\]
We deduce that for all $t\ge 1$ and $s\in [2/t,2]$, $p(s)\ge p(2)=c$ and 
\begin{align*}
C^{-1}\tau_{\infty}(p,t)&\le 
\begin{cases}
\cro{{(st)^{\alpha}\pa{\log(st)}^{1+\beta}}}^{-1}+s^{\gamma}& \text{when $\alpha>0$ and $\beta\ge -1$}\\
\pa{\log(st)}^{-\beta}+s^{\gamma}& \text{when $\alpha=0$ and $\beta>0$}.
\end{cases}
\end{align*}
Taking 
\[
s=
\begin{cases}
2\cro{\pa{t^{\alpha}\pa{\frac{\log(1+t)}{\log2}}^{1+\beta}}^{-\frac{1}{\alpha+\gamma}}\vee t^{-1}}& \text{when $\alpha>0$ and $\beta\ge -1$}\\
2\cro{\pa{\frac{\log(1+t)}{\log 2}}^{-\frac{\beta}{\gamma}}\vee t^{-1}}& \text{when $\alpha=0$ and $\beta>0$}
\end{cases}
\]
the value of which belongs to $[2/t,2]$, we obtain that for all $t\ge 1$ 
\begin{align*}
C^{-1}\tau_{\infty}(p,t)&\le 
\begin{cases}
\cro{t^{\alpha}\pa{\log(1+t)}^{1+\beta}}^{-\frac{\gamma}{\alpha+ \gamma}}\vee t^{-\gamma}& \text{when $\alpha>0$ and $\beta\ge -1$}\\
\pa{\log(1+t)}^{-\beta}\vee t^{-\gamma}& \text{when $\alpha=0$ and $\beta>0$}.
\end{cases}
\end{align*}
If $\alpha>0$ and $\beta\ge -1$, by taking 
\[
t=t_{n}=C'n^{\frac{1}{3}\pa{\frac{1}{\alpha}+\frac{1}{\gamma}}}\log^{-\kappa}n\quad \text{with}\quad \kappa=\frac{1}{3}\pa{\frac{1}{\alpha}+\frac{1}{ \gamma}}+\frac{1+\beta}{\alpha}
\]
for some constant $C'>0$ large enough, we obtain that $\tau_{\infty}(p,t_{n})\le \pa{2\log(1+t_{n})/n}^{1/3}$ and consequently, $r_{n}$ defined by~\eref{def-rn} satisfies $r_{n}\le t_{n}$. Applying Corollary~\ref{cor-monogene}, we conclude that  for all $k\ge 4$ 
\[
\B_{k,n}(p)\le C\cro{\pa{\frac{\log n}{n}}^{1/3}+\sqrt{\frac{2k}{n}}}.
\]

If $\alpha=0$ and $\beta>0$, we take $t=t_{n}$ such that $\log (1+t_{n})=C'n^{1/(1+3\beta)}$ for some constant $C'>0$ large enough, we obtain that  
\begin{align*}
\tau_{\infty}(p,t_{n})\le Cn^{\frac{-\beta}{1+3\beta}}\le \pa{\frac{2\log(1+t_{n})}{n}}^{1/3},
\end{align*}
hence $r_{n}\le t_{n}$ and we get that for all $k\ge 4$
\[
\B_{k,n}(p)\le C\cro{n^{\frac{-\beta}{1+3\beta}}+\sqrt{\frac{2k}{n}}}.
\]
We are not aware of any $\L_{1}$-risk bound for the Grenander estimator when the target density has an infinite support and heavy tails.
\end{exa}

\section{Convex-concave densities}\label{sect-cvxcve}

\subsection{Piecewise monotone convex-concave densities}

In this section, our aim is to estimate a density on the line which is piecewise monotone convex-concave in the sense defined below.

\begin{df} \label{def_cvx_cve}
A function $g$ is said to be convex-concave on an interval $I$ if it is either convex or concave on $I$. For $k\geq 2$, a function $g$ on $\mathbb{R}$ is said to be $k$-piecewise monotone convex-concave  if there exists $A \in \cA(k-1)$ such that the restriction of $g$ to the each interval $I\in \gI(A)$ is monotone and convex-concave. In particular, there exist at most $k$ functions $\{g_I,\; I \in \gI(A)\}$, where $g_{I}$ is monotone and convex-concave on $I$ such that 
\[
g(x) = \sum_{I \in \gI(A)} g_I(x)\1_{I}(x) \quad \text{for all $x\in\R\setminus A$}.
\]
\end{df}
We denote by $\overline{\cM}_k^1$ the set of $k$-piecewise monotone convex-concave densities. The Laplace density $x\mapsto (1/2)e^{-|x|}$ belongs to  $\overline{\cM}_2^1$, the uniform density on a bounded interval belongs to $\overline{\cM}_3^1$, all convex-concave densities on an interval belong to $\overline{\cM}_4^1$. A function $g\in \overline{\cM}_k^1$ associated to $A\in\cA(k-1)$ admits left and right derivatives at any point $x\in \R\setminus A$. These derivatives are  denoted by $g_l',g_r'$ respectively. More generally, when a function $f$ is continuous and convex-concave on a nontrivial bounded interval $[a,b]$, we define 
\[
f'_{r}(z)=\lim_{x\downarrow z}\frac{f(x)-f(z)}{x-z}\quad \text{for all $z\in [a,b)$}
\]
and 
\[
f'_{l}(z)=\lim_{x\uparrow z}\frac{f(x)-f(z)}{x-z}\quad \text{for all $z\in (a,b]$}.
\]
These quantities are finite for all $z\in (a,b)$ and belong to $[-\infty,+\infty]$ when $z\in \{a,b\}$. We say that $f$ admits a right derivative at $a$ and a left derivative at $b$ when $f'_{r}(a)$ and $f'_{l}(b)$ are finite respectively. 

The role played by piecewise constant functions in the previous section is here played by piecewise linear functions. For $D \geq 1$, let $\overline{\cO}_{D,k}^1$ be the subset of $\overline{\cM}_k^1$ that consists of those densities that are left-continuous and  affine on each interval of a class $\gI(A)$ with $A \in \cA(D+1)$.  
For example, the left-continuous version of the density of a uniform distribution on a nontrivial interval belongs to $\overline{\cO}_{1,3}^1$. 
The proposition below shows that the elements of $\overline{\cO}_{D,k}^1$ are extremal in $\overline{\cM}_k^1$.
\begin{prop} \label{prop_extr_cvx_cve}
Let $k\ge 2$, $D\ge 1$. If $p\in \overline \cM_{k}^{1}$ and $q\in \overline{\cO}_{D,k}^1$, 
the sets $\{x\in\R,\, p(x) -q(x) > 0\}$ and $\{x\in\R,\, p(x) - q(x) < 0\}$ are unions of at most $D + 2k-1$ intervals. In particular, the elements of $\overline{\cO}_{D,k}^1$ are extremal in $\overline{\cM}_k^1$ with degree not larger than $2(D + 2k -1)$.
\end{prop}

\begin{proof}
The proof is postponed to Subsection \ref{subsect-proof-cvxcve}.
\end{proof}

For all $D\geq 1$, $k\geq 2$, let $\cO_{D,k}^1$ be a countable and dense subset of $\overline{\mathcal{O}}_{D,k}^1$ (for the $\L_1$-norm) and $\cM_k^1$ a countable and dense subset of $\overline{\cM}_k^1$ that contains $\bigcup_{D \geq 1} \mathcal{O}_{D,k}^1$.
%
By proposition~\ref{prop-VC-monotone}, the elements of $\cO_{D,k}^1$ are also extremal in $\cM_k^1$ with degree no larger than $2(D+2k-1)$ for all $D \geq 1$. We deduce from Theorem \ref{shape-estimation-th} the following result.
%
\begin{thm}\label{thm_or_ineq_cvx_cve}
Let $k\ge 2$. For every product distribution $\gP\et$ of the data, a TV-estimator $\widehat P$ on $\sM_{k}^{1}$ satisfies 
\begin{equation} \label{or-ineq_cvx_cve_extr_pts}
    \E \left[ \dTV{P\et}{\hat{P}} \right] \leq \inf_{D \geq 1} \left\{ 3 \inf_{P \in \overline{\sO}_{D,k}^1} \dTV{P\et}{P} + 
    48 \sqrt{\frac{2(D+2k-1)}{n}} \right\} + \frac{\varepsilon}{n}.
\end{equation}
\end{thm}

In the remaining part of this section we assume that the $X_{i}$ are i.i.d.\ with a density $p\in \overline \cM_{k}^{1}$, in which case, the right-hand side of~\eref{or-ineq_cvx_cve_extr_pts} can be written as  $\B_{k,n}^{1}(p)+\varepsilon/n$ with 
\begin{equation}\label{def-b1kn}
\B_{k,n}^1(p) = \inf_{D \geq 1} \left[ \frac{3}{2} \inf_{q \in \overline{\cO}_{D,k}^1} \norm{p - q} + 48 \sqrt{\frac{2(D + 2k -1)}{n}} \right]. 
\end{equation}
As we did in Section~\ref{sect-kmono}, our aim is to bound the quantity $\B_{k,n}^1(p)$ under some suitable additional assumptions on the density $p$. 

\subsection{Approximation of a monotone convex-concave density by a piecewise linear function}
Let us now turn to the approximation of a monotone convex-concave density by a convex-concave piecewise linear function. The approximation result that we establish is actually true for a {\em sub-density} on an interval $[a,b]$, i.e.\ a nonnegative function on $[a,b]$ the integral of which is not larger than 1. In the remaining part of this chapter, we use the following convenient definition. 
\begin{df}\label{def-linear}
Let $D\ge 1$ and $f$ be a continuous function on a compact nontrivial interval $[a,b]$. We say that $\overline f$ is a $D$-linear interpolation of $f$ on $[a,b]$ if there exists a subdivision $a=x_{0}<\ldots<x_{D}=b$ such that $\overline f(x_{i})=f(x_{i})$ and $\overline f$ is affine on $[x_{i-1},x_{i}]$ for all $i\in\{1,\ldots,D\}$. 
\end{df}

This definition automatically determines the values of $\overline f$ on $[a,b]$ since $\overline f$ corresponds on $[x_{i-1},x_{i}]$ to the chord that connects  $(x_{i-1},f(x_{i-1}))$ to $(x_{i},f(x_{i}))$ for all $i\in\{1,\ldots,D\}$. The function $\overline f$ is therefore continuous and piecewise linear on a partition of $[a,b]$ into $D$ intervals and it inherits of some of the features of the function $f$. For example, if $f$ is nonnegative, increasing, decreasing, convex or concave, so is $\overline f$. If $f$ is convex (respectively concave), $\overline f\ge f$ (respectively $\overline f\le f$).

Given a continuous monotone convex-concave function $f$ with increment $\Delta=(f(b)-f(a))/(b-a)$ on a bounded nontrivial interval $[a,b]$, we define its linear index $\Gamma=\Gamma(f)$ as 
\[
\Gamma=1-\frac{1}{2}\pa{\frac{|p_{r}'(a)|\wedge |p_{l}'(b)|}{|\Delta|}+\frac{|\Delta|}{|p_{r}'(a)|\vee |p_{l}'(b)|}},
\]
with the conventions $0/0=1$ and $1/(+\infty)=0$. Since $f$ is convex-concave, monotone and continuous 
\[
|p_{r}'(a)|\wedge |p_{l}'(b)|\le |\Delta| \le |p_{r}'(a)|\vee |p_{l}'(b)| \quad \text{and}\quad 
\]
and 
\[
\Delta=0\implies |p_{r}'(a)|= |p_{l}'(b)|=0.
\]
With our conventions, $\Gamma$ is well-defined and belongs to $[0,1]$. When $f$ is affine, $\Delta=p_{r}'(a)=p_{l}'(b)$ and 
its linear index is 0.  In the opposite direction when $f$ is far from being affine, say when for some $c\in (a,b)$ and $v>0$
\[
f(x)=\frac{v}{b-c}(x-c)_{+}\quad \text{for all $x\in [a,b]$}
\]
its linear index $\Gamma=1-(b-c)/[2(b-a)]$ increases to 1 as $c$ approaches $b$. 

\begin{thm}\label{thm-approx-conv}
Let $p$ be a monotone, continuous, convex-concave density on a bounded interval $[a,b]$ of length $L>0$ with variation $V=|p(b)-p(a)|$ and linear index $\Gamma\in [0,1]$. For all $D\ge 1$, there exists a $2D$-linear interpolation $\overline p$ of $p$ such that 
\begin{equation}\label{eq-thm-approx-conv}
\int_{a}^{b}\ab{p-\overline p}d\mu\le \frac{4}{3}\cro{\pa{1+\sqrt{2LV\Gamma}}^{1/D}-1}^{2}.
\end{equation}
In particular, there exists a continuous convex-concave piecewise linear density $q$ based on a partition of $[a,b]$ into $2D$ intervals that satisfies
\begin{equation}\label{thm-approx-conv-01}
\int_{a}^{b}\ab{p-q}d\mu\le  5.14\frac{\log^{2}\pa{1+\sqrt{2LV\Gamma}}}{D^{2}}.
\end{equation}
\end{thm}

Since $\Gamma\in [0,1]$, \eref{thm-approx-conv-01} implies that 
\[
\int_{a}^{b}\ab{p-q}d\mu\le  5.14\frac{\log^{2}\pa{1+\sqrt{2LV}}}{D^{2}}.
\]
Nevertheless, when $p$ is affine on $[a,b]$, $\Gamma=0$ and we recover the fact that we may choose $\overline p$ and $q$ on $[a,b]$ such that $\int_{a}^{b}\ab{p-\overline p}d\mu=\int_{a}^{b}\ab{p-q}d\mu=0$.

Gu\'erin {\em et al} \cite{Guerin2006} already established that any bounded convex (or concave) function on a compact interval can be approximated by a piecewise affine function on a partition of this interval into $D$-pieces
with an $\L_{1}$-error not larger than $C/D^2$, for some number $C>0$ which is independent of $D$. The novelty in Theorem \ref{thm-approx-conv} lies in the fact that for probability densities the approximation error depends logarithmically on the product $LV$. 

\begin{proof}
The proof is based on several preliminary approximation results whose statements and proofs are postponed to Subsection \ref{subsect-proof-cvxcve}.
\end{proof}

\subsection{Estimation of $k$-piecewise monotone convex-concave\ bounded and compactly supported densities}
In this section, we consider a density $p\in\overline \cM_{k}^{1}$, with $k\ge 3$, that is of the form \eref{def-mbark} except for the fact that $(iii)$ is here replaced by 
\begin{itemize}
\item[$(iii')$] for $i\in\{1,\ldots,k-2\}$, $p_{i}$ is a monotone continuous convex-concave density on the interval $[x_{i-1},x_{i}]$ of length $L_{i}>0$, with variation $V_{i}=\ab{p_{i}(x_{i-1})-p(x_{i})}<+\infty$ and linear index $\Gamma_{i}\in [0,1]$. 
\end{itemize}
Throughout this section, when we refer to the form \eref{def-mbark}, it means with $(iii)$ replaces by $(iii')$. 
For such a density $p$, we define 
\begin{equation}\label{def-Rp1}
\bs{R}_{k,1}(p)=\inf\cro{\sum_{i=1}^{k-2}\pa{w_{i}\log^{2}\pa{1+\sqrt{2L_{i}V_{i}\Gamma_{i}}}}^{1/3}}^{3/2},
\end{equation}
where the infimum runs among all ways of writing $p$ in the form~\eref{def-mbark}. We denote by $\overline \cM_{k,1}^{\infty}$ the class of all densities of the form ~\eref{def-mbark} a.e.\ and for $R>0$, $\overline \cM_{k,1}^{\infty}(R)$ the subset of those which satisfy $\bs{R}_{k,1}(p)<R$.

Note that a concave (or convex) density on a (necessarily) compact interval belongs to $\overline \cM_{4}^{1}$.

The following result holds. 
\begin{thm} \label{thm_cvx_cve}
Let $k\ge 3$ and $R>0$. For all $p\in  \overline \cM_{k,1}^{\infty}(R)$
\begin{equation}\label{eq-cvx_cve}
\B_{k,n}^1(p) \le  7.71\cro{15.06\pa{\frac{R}{n}}^{2/5}+8.82\sqrt{\frac{4k-5}{n}}},
\end{equation}
where $\B_{k,n}^1(p)$ is defined by~\eref{def-b1kn}.
\end{thm}
%
Theorem \ref{thm_cvx_cve} together with Theorem \ref{thm_or_ineq_cvx_cve} imply  that the TV-estimator converges at rate $n^{- 2/5}$ in total variation distance, whenever the underlying density $p\et$ is bounded, compactly supported and belongs to the class $\overline{\cM}_{k,1}^{\infty}$ of $k$-piecewise monotone convex-concave densities. The rate $n^{- 2/5}$ matches the minimax lower bound established in Devroye and Lugosi~\cite{MR1843146}[Section~15.5]  for bounded convex densities. This rate is therefore optimal.

To the best of our knowledge, Theorem \ref{thm_cvx_cve} provides the sharpest known minimax upper bound in this setting, including the case of a monotone, convex or concave density on a compact interval. In comparison, Gao and Wellner \cite{MR2520591} proved that the MLE on the set of convex non-increasing densities on a given interval achieves the rate $n^{- 2/5}$ (for the Hellinger distance). Note that the construction of the MLE requires that the support of the target density be known while our TV-estimator assumes nothing.


Consider now the special case of a continuous, concave density on an interval $[a,b]$ (which is necessarily bounded).
A monotone continuous concave density $p$ on a bounded interval $[a,b]$ with length $L>0$ belongs to $\overline \cM_{3}^{1}$. It necessarily satisfies $L|p(a)-p(b)|/2\le 1$, hence $\bs{R}_{3,1}(p)\le  \log(1+\sqrt{2L|p(a)-p(b)|})\le \log 3$. If $p$ is not monotone but only continuous and concave on $[a,b]$, we may write $p$ as $w_{1}p_{1}+w_{2}p_{2}$ where $p_{1}$ and $p_{2}$ are monotone and concave densities on the intervals $[a,c]$ and $[c,a]$ respectively where $c$ is a maximizer of $p$ in $(a,b)$. The density $p$ belongs to $\overline \cM_{4}^{1}$ and by applying the previous inequality to the densities $p_{1}$ and $p_{2}$ successively and the inequality $z^{1/3}+(1-z)^{1/3}\le 2^{2/3}$ which holds for all $z\in [0,1]$, we obtain that
\begin{align*}
\bs{R}_{4,1}^{2/3}(p)&\le \pa{w_{1}\log^{2}(1+\sqrt{2(c-a)\pa{p_{1}(c)-p_{1}(a)}}}^{1/3}
+ \pa{w_{2}\log^{2}(1+\sqrt{2(b-c)\pa{p_{2}(c)-p_{2}(b)}}}^{1/3}\\
&\le \pa{w_{1}^{1/3}+w_{2}^{1/3}}(\log 3)^{2/3}\le (2\log 3)^{2/3}.
\end{align*}
We immediately deduce from Theorems~\ref{thm_or_ineq_cvx_cve} and~\ref{thm_cvx_cve} the following corollary.  
\begin{cor}\label{cor-cvx_cve}
If $X_{1},\ldots,X_{n}$ are i.i.d.\ with a density that is concave on an interval of $\R$, a TV-estimator $\widehat p$ on $\cM_{4}^{1}$ satisfies 
\[
\E\cro{\norm{p-\widehat p}}\le \frac{320}{n^{2/5}}+\frac{451}{\sqrt{n}}+\frac{2\varepsilon}{n}.
\]
In particular,
\[
\adjustlimits\inf_{\widetilde p}\sup_{p}\E\cro{\norm{p-\widetilde p}}\le \frac{320}{n^{2/5}}+\frac{451}{\sqrt{n}},
\]
where the supremum runs among all concave densities $p$ on an interval of $\R$ and the infimum over all density estimators $\widetilde p$ based on a $n$-sample with density $p$.
\end{cor}

\begin{proof}
Let $p\in\overline \cM_{k,1}^{\infty}(R)$. Without loss of generality we may assume that $p$ is of the form \eref{def-mbark} everywhere and choose a subdivision $(x_{i})_{i\in\{0,\ldots,k-2\}}$ in such a way that 
\[
\cro{\sum_{i=1}^{k-2}\pa{w_{i}\log^{2}\pa{1+\sqrt{2L_{i}V_{i}\Gamma_{i}}}}^{1/3}}^{3/2}\le R.
\]

Let $D\ge k-2$ and $D_{1},\ldots,D_{k-2}$ be some positive integers to be chosen later on that satisfy the constraint $\sum_{i=1}^{k-2}D_{i}\le D$. By Theorem~\ref{thm-approx-conv}, we may find for all $i\in\{1,\ldots,k-2\}$ a density $q_{i}$ that is continuous and supported on $[x_{i-1},x_{i}]$, piecewise linear on a partition of $[x_{i-1},x_{i}]$ into $2D_{i}$ intervals, that  satisfies 
\[
\int_{x_{i-1}}^{x_{i}}\ab{p_{i}-q_{i}}d\mu\le  5.14\frac{\log^{2}\pa{1+\sqrt{2L_{i}V_{i}\Gamma_{i}}}}{D_{i}^{2}}.
\]
The function $q=\sum_{i=1}^{k-2}w_{i}q_{i}\1_{(x_{i-1},x_{i}]}$ is a density, that is left-continuous, convex-concave on each interval $I\in\gI(\{x_{0},\ldots,x_{k-2}\})$ and affine on each interval of a partition $(x_{0},x_{k-2}]=\bigcup_{i=1}^{k-2}(x_{i-1},x_{i}]$ into $\sum_{i=1}^{k-2}2D_{i}\le 2D$ intervals.  It therefore belongs to $\overline \cO_{2D,k}$. Besides,
\begin{align*}
\norm{p-q}&\le \sum_{i=1}^{k-2}w_{i}\int_{x_{i-1}}^{x_{i}}\ab{p_{i}-q_{i}}d\mu\le 5.14\sum_{i=1}^{k-2}\frac{w_{i}\log^{2}\pa{1+\sqrt{2L_{i}V_{i}\Gamma_{i}}}}{D_{i}^{2}}
\end{align*}
and it follows from~\eref{def-b1kn} that
\begin{align*}
\B_{k,n}^1(p) &\le \frac{3}{2} \inf_{q \in \overline{\cO}_{2D,k}^1} \norm{p - q} + 68 \sqrt{\frac{2D + 2k -1}{n}}\\
&\le 7.71\sum_{i=1}^{k-2}\frac{w_{i}\log^{2}\pa{1+\sqrt{2L_{i}V_{i}\Gamma_{i}}}}{D_{i}^{2}}+68 \sqrt{\frac{2D + 2k -1}{n}}. 
\end{align*}
Let us set $c=68/7.71$, $s_{0}=[nR^{4}/(2c^{2})]^{1/5}$, 
\[
s_{i}=\cro{w_{i}\log^{2}\pa{1+\sqrt{2L_{i}V_{i}\Gamma_{i}}}}^{1/3}\quad \text{for all $i\in\{1,\ldots,k-2\}$}
\]
and choose $D=s_{0}+k-2$ and 
\[
D_{i}=\left\lceil\frac{s_{0}s_{i}}{\sum_{j=1}^{k-2}s_{j}}\right\rceil\ge \frac{s_{0}s_{i}}{\sum_{j=1}^{k-2}s_{j}}\vee 1\quad \text{for all $i\in\{1,\ldots,k-2\}$},
\]
so that $\sum_{i=1}^{k-2}D_{i}\le D$. Then, 
\begin{align*}
\B_{k,n}^1(p) &\le 7.71\cro{\sum_{i=1}^{k-2}\frac{s_{i}^{3}}{D_{i}^{2}}+c \sqrt{\frac{2D + 2k -1}{n}}}\le  7.71\cro{\frac{\pa{\sum_{j=1}^{k-2}s_{j}}^{3}}{s_{0}^{2}}+c \sqrt{\frac{2s_{0}+4k-5}{n}}}\\
&\le 7.71\cro{\frac{R^{2}}{s_{0}^{2}}+c \sqrt{\frac{2s_{0}}{n}}+c\frac{\sqrt{4k-5}}{n}}=7.71\cro{2(2c^{2})^{2/5}\frac{R^{2/5}}{n^{2/5}}+c\frac{\sqrt{4k-5}}{n}}
\end{align*}
which gives \eref{eq-cvx_cve}. 
\end{proof}

\section{The case of $s$-concave densities and their generalisations}\label{section-s-shape}
Given an increasing function $\cL$ on $(0,+\infty)$ and a density $p$ for which $\{p>0\}$ is an open subinterval of $\R$, we define 
\[
\cL p:x\mapsto 
\begin{cases}
\cL\circ p(x) & \text{when $x\in \{p>0\}$}\\
-\infty & \text{otherwise.}
\end{cases}
\]
A density $p$ is said to be {\em $\cL$-concave} if $\cL p$ is concave on $\{p>0\}$. 
We denote by $\overline\cM(\cL)$ the class of $\cL$-concave densities and for $D\ge 1$, $\overline\cO_{D}(\cL)$ the subset of $\overline \cM(\cL)$ which consists of these densities $\overline p$ for which $\cL\overline p$ is either affine or takes the value $-\infty$ on the elements of a partition $\gI(A)$, $A\in \cA(D)$, containing thus at most $D+1$ intervals. Note that the uniform distributions on an interval belong to $\overline\cO_{D}(\cL)$ for all $D\ge 2$, which shows in passing that the sets $\overline \cM(\cL)$ and $\overline\cO_{D}(\cL)$ are nonempty for every choices of $\cL$ and $D\ge 2$. As before, we denote by 
$\cO_{D}(\cL)$ a countable and dense subset of $\overline \cO_{D}(\cL)$ for $D\ge 1$ and $\cM(\cL)$ a countable and dense subset of $\overline \cM(\cL)$ which contains $\bigcup_{D\ge 1}\cO_{D}(\cL)$. The sets $\sO_{D}(\cL),\overline \sO_{D},\sM(\cL)$ and $\overline \sM(\cL)$ are the classes of probability distributions on $\R$ that are associated with the sets of densities $\cO_{D}(\cL),\overline \cO_{D}(\cL),\cM(\cL)$ and $\overline \cM(\cL)$ respectively.

The functions $\cL$ we have in mind are typically of the  form $\cL=\cL_{s}$ with $s\in\R$ and defined for $u>0$ by 
\[
\cL_{s}(u)=
\begin{cases}
-u^{s}&\text{when $s<0$}\\
\log u &\text{when $s=0$}\\
u^{s} &\text{when $s>0$.}
\end{cases}
\]
In the literature, a $\cL_{s}$-concave density is called {\em $s$-concave} and  log-concave in the special case $s=0$. For $s<s'$, the reader can check that $\overline \cM(\cL_{s'})\subset \overline \cM(\cL_{s})$, hence that the family $(\overline \cM(\cL_{s}))_{s\in\R}$ is nonincreasing with respect to set inclusion.

Seregin and Wellner \cite{Ser-Well2010} proved that the MLE on $\overline \cM(\cL_{s})$ exists for those $s>-1$ (at least for $n$ large enough). The following result shows that it is unstable at least for those $s\ge 1$.
 \begin{prop}
 Given $n\ge 5$ and $c\ge 5$, consider the distribution $P\et=(1-1/n)P_{0}+(1/n)\delta_{c}$ where $P_{0}$ denotes the uniform distribution on $(0,1)$. Let $\widehat L_{n}$ be the likelihood function on $\cM(\cL_{s})$, with $s\ge 1$, based on an $n$-sample with distribution $P\et$.  With a probability at least 63\%, for any density $p_{s}\in \cM(\cL_{s})$ such that $\widehat L_{n}(p_{s})>0$, the associated distribution $P_{s}=p_{s}\cdot\mu\in \sM(\cL_{s}) $ satisfies 
\[
d(P_{0}, P_{s})\ge \frac{1}{2}\pa{1-\frac{2}{c-1}}\ge \frac{1}{4} \quad \text{while}\quad P_{0}\in \sM(\cL_{s})\; \text{ and }\; d(P_{0},P\et)= \frac{1}{n}.
\]
 \end{prop}
%

\begin{proof}
Let $X_{(1)},\ldots,X_{(n)}$ the order statistics associated to the $n$-sample with distribution $P\et$. Since $c>1$, 
\[
\P\cro{X_{(1)}=c}=\P(X_{1}=\ldots=X_{n}=c)=n^{-n}\quad \text{and}\quad \P\cro{X_{(n)}\ne c}=\pa{1-\frac{1}{n}}^{n}<e^{-1}.
\]
We deduce that the probability of the set $\Omega=\{0\le X_{(1)}<X_{(n)}=c\}$ is at least $1-e^{-1}-n^{-n}>63\%$ for $n\ge 5$. 
Since $s\ge 1$, a density $p_{s}$ in $\cM(\cL_{s})$ is continuous and concave on $\{p_{s}>0\}$, $\{p_{s}>0\}$ is a bounded interval of the form $(a,b)$ with $a<b$ and $p_{s}$ can be extended continuously at $a$ and $b$. In particular, there exists $z\in [a,b]$ such that $p_{s}(z)=\sup_{x\in [a,b]}p_{s}(x)$. Besides, the condition $\widehat L_{n}(p_{s})>0$ implies that on $\Omega$, $a\le X_{(1)}\le 1\le X_{(n)}=c\le b$ and since the graph of the density $p_{s}$ on $[a,b]$ lies above its chord,  
\[
1= \int_{a}^{b}p_{s}(x)dx\ge \frac{p_{s}(z)(b-a)}{2}\ge \frac{p_{s}(z)(c-1)}{2}.
\]
Hence, $p_{s}(z)\le 2/(c-1)<1$ for $c>3$ and we conclude that
\[
d(P_{0},P_{s})\ge \frac{1}{2}\int_{0}^{1}\ab{1-p_{s}(x)}dx\ge \frac{1}{2}\int_{0}^{1}\pa{1-p_{s}(x)}dx\ge \frac{1}{2}\pa{1-\frac{2}{c-1}}.
\]
\end{proof}

\subsection{Extremal points and risk bounds over $\overline \cM(\cL)$}

The following result shows that the elements of $\overline \cO_{D}(\cL)$ are extremal.
\begin{prop}\label{prop-cL}
For all $D\ge 1$, the elements of $\overline \cO_{D}(\cL)$ are extremal in $\overline \cM(\cL)$ with degree not larger than $2(D+2)$.
\end{prop}
\begin{proof}
The proof is postponed to Section~\ref{subsect-proof-sshape}.
\end{proof}

 We immediately deduce from Theorem \ref{shape-estimation-th} the following risk bound for the TV-estimator $\widehat P=\widehat p\cdot \mu$ on $\overline \sM(\cL)$. 
\begin{thm}\label{thm-cL-concave}
For every product distribution $\gP\et$ of the data, a TV-estimator $\widehat P$ on $\sM(\cL)$ satisfies 
\begin{equation}\label{eq-Thm-cL-concave}
\mathbb{E}\cro{\dTV{P \et}{\widehat P}} \leq  \inf_{D\ge 1}\cro{3\inf_{P\in\overline \sO_{D}(\cL)}\dTV{P \et}{P} + 48\sqrt{2}\sqrt{\frac{D+2}{n}}}+\frac{\varepsilon}{n}.
\end{equation}
\end{thm}
As a consequence, if the data are i.i.d.\ with a density $p\et$ that belongs to $\overline \cO(\cL)=\bigcup_{D\ge 1}\overline \cO_{D}(\cL)$, the density $\widehat p\in\overline \cM(\cL)$ of the TV-estimator converges in $\L_{1}$  to $p\et$ at rate $1/\sqrt{n}$. When $\cL$ is continuous, hence one-to-one from $(0,+\infty)$ onto its range, this result applies to these densities $p\et$ of the form $\cL^{-1}(-a|x-c|+b)\1_{I}(x)$ where $a\in [0,+\infty)$, $b,c\in\R$ and $I$ is an open interval. This includes all uniform densities on $\R$ (even when $\cL$ is not continuous). In the special case where $\cL=\cL_{s}$ with $s\in (-1,0)$, densities of the form 
\[
p(x)=\pa{\frac{1}{a|x|+b}}^{1/|s|}\quad \text{or}\quad  p(x)=\pa{\frac{1}{ax+b}}^{1/|s|}\1_{x>0}\quad \text{with $a,b>0$}
\]
and their translated versions belong to $\overline \cO(\cL_{s})$. 

Of special interest is the case $s=0$, which corresponds to log-concave densities. As a result, we obtain that our TV-estimator on the set of log-concave densities estimates at rate $1/\sqrt{n}$ densities of the form $x\mapsto 2a\exp\cro{-a|x-b|}$ or $x\mapsto a\exp\cro{-a(x-b)}\1_{x>b}$ with $a>0$ and $b\in\R$.  Our estimator shares thus some adaptation property which is of the same flavour as those established for the MLE by Kim {\em et al}~\cite{MR3845018} (and Feng {\em et~al.} \cite{Feng2021}), and by Baraud and Birg\'e~\cite{MR3565484} for their $\rho$-estimator. As compared to theirs, our upper bound does not involve logarithmic factors. When considering the Hellinger loss, these logarithmic factors are sometimes necessary while they disappear for the total variation one. Nevertheless, the paper by Kim {\em et al}~\cite{MR3845018} additionally contains a risk bound for the MLE with respect to total variation distance (see their Theorem~1). For the exponential distribution (and their translated versions), they obtain a rate of convergence which is $(\log n)/\sqrt{n}$ for the MLE while our estimator would converge at the faster rate $1/\sqrt{n}$. We do not know whether or not this extra logarithmic factor for the MLE is technical or necessary. In the latter case, this would mean that the MLE is slightly sub-optimal for estimating exponential distributions among the set of log-concave densities. Nevertheless, the authors also exhibit examples of log-concave densities for which this extra logarithmic factor can be removed. For these cases, both the MLE and our TV estimator converge at the rate $1/\sqrt{n}$. 


In order to specify further the risk bound \eref{eq-Thm-cL-concave}, we need to investigate the approximation properties of $\overline \sO_{D}(\cL)$ with respect to $P\et$. In the remaining part of this section, we shall restrict ourselves to the case where $\cL$ is either convex or concave on $(0,+\infty)$ and the data  i.i.d.\ with a density $p\et\in\overline \cM(\cL)$. Within this framework, our aim is to study the approximation properties of the sets of densities $\overline \cO_{D}(\cL)$, $D\ge 1$, with respect to the target density $p\et$ for the $\L_{1}$-norm.

\subsection{Approximation properties of $\overline \cO_{D}(\cL)$ when $\cL$ is convex.}
In this section, we study the case where the function $\cL$ is convex which includes functions of the form $\cL=\cL_{s}$ with $s\ge 1$. 
When the function $\cL$ is not only increasing but also convex, it is continuous on $(0,+\infty)$ and its admits an inverse $\cL^{-1}$, defined on the range of $\cL$, which is increasing, continuous and concave. Since $p=\cL^{-1}(\cL p)$ on $\{p>0\}$, the density $p$ is necessarily concave on the interval $\{p>0\}$ and can therefore be well approximated by piecewise affine functions. However, such functions may not belong to the sets $\overline \cO_{D}(\cL)$. Nevertheless, we show that the approximation properties of the elements of $\overline \cO_{D}(\cL)$ with respect to $p$ are, up to a factor 2, as good as those we would get by using piecewise affine functions. 
\begin{prop}\label{prop-astGuillaume}
Let $\cL$ be a convex increasing function on $(0,+\infty)$ and $p$ an element of $\overline \cM(\cL)$. The set $\{p>0\}$ is a bounded open interval $(a,b)$ with $a<b$ and $p$ can be extended continuously on $[a,b]$. For all $D\ge 1$, there exist a $2D$-linear interpolation $\overline p$ of $p$ on $[a,b]$ as well as a density $s\in \overline{\cO}_{2D+1}(\cL)$ such that
\begin{equation}\label{eq-astGuillaume1}
\int_{\R} |p -s| d\mu \leq 2\int_{a}^{b} |p -\overline p| d\mu \leq \frac{2}{D^2}. 
\end{equation}
%
%
\end{prop}
In particular, this result applies to $\cL=\cL_{s}$ with $s\ge 1$. 
\begin{proof} 
The proof of the proposition is postponed to Section \ref{subsect-proof-sshape}.
\end{proof}

\subsection{Approximation properties of $\overline \cO_{D}(\cL)$ when $\cL$ is concave.}
In this section we investigate the situation where the function $\cL$ is concave. 
\begin{thm} \label{thm-approx-scv}
Assume that the function $\cL$ is increasing and concave on $(0,+\infty)$ and that there exist $\lambda,\kappa>1$ such that 
\begin{equation}\label{eq-condcL}
\lambda^{2}\cL_{l}'(\lambda y)\ge \kappa \cL_{r}'(y)\quad \text{for all $y>0$.}
\end{equation}
Then, set  
\begin{equation}\label{eq-cL-defR}
\gamma=\gamma(\lambda,\kappa)=\cro{\frac{2}{\log \kappa}\pa{\frac{\kappa}{\kappa-1}}^{1/2}+\frac{\lambda\sqrt{\lambda-1}
}{\sqrt{2 \kappa }}\pa{1+\frac{1}{2(\sqrt{\kappa}-1)}}}^{2}.
\end{equation}
For any $p \in \overline{\cM}(\cL)$ and any integer $D \ge 1$, there exists an extremal point $\overline{p} \in \overline{\cO}_{2D+3}(\cL)$ such that
\begin{equation} \label{eq-bd-approx-scv}
\int_{\mathbb{R}} |p(x) - \overline{p}(x)| dx \leq \frac{2 \gamma}{D^2}.
\end{equation}
\end{thm}

\begin{proof}
The proof is postponed to Section~\ref{subsect-proof-sshape}. 
\end{proof}

Let us comment on condition \eref{eq-condcL}. If we wish to interpret the fact that a density $p$ belongs to the set $\overline \cM(\cL)$ as a condition on its shape, it is natural to require that this shape remain unchanged under a location-scale transformation: the density $p_{\sigma}:x\mapsto \sigma p(\sigma x)$ with $\sigma>0$ should therefore belong to $\overline \cM(\cL)$ if $p$ does. 
Let $p\in \overline \cM(\cL)$ and assume for the sake of simplicity that both $\cL$ and $p$ are differentiable so that the derivatives of $\cL p$ and $\cL p_{\sigma}$  are respectively given  by 
\[
(\cL p)'(x)=p'(x)\cL'[p(x)]\quad \text{and}\quad (\cL p_{\sigma})'(x)= p'(\sigma x)\sigma^{2}\cL'[\sigma p(\sigma x)]\quad \text{for all $x\in\R$.}
\]
Under the assumption 
\begin{equation}\label{eq-compare}
\forall \sigma>0,\exists c>0\;\text{ such that }\; \sigma^{2}\cL'(\sigma y)=c\cL'(y)\;\text{ for all $y>0$,}
\end{equation}
%
the following equality holds 
\[
(\cL p_{\sigma})'(x)=c p'(\sigma x)\cL'[p(\sigma x)]=c(\cL p)'(\sigma x)\quad \text{for all $x\in\R$.}
\]
Since $(\cL p)'$ is nonincreasing because $\cL p$ is concave, so is $(\cL p_{\sigma})'$ and we conclude that under \eref{eq-compare} the densities $p_{\sigma}$ with $\sigma>0$ also belong  $\overline \cM(\cL)$. The set $\overline \cM(\cL)$ is stable under a scale transformation. 
Condition \eref{eq-compare} is satisfied for the functions $\cL_{s}$ with $s\in\R$ that are associated to the so-called $s$-concavity property. In our framework, we do not require that the set $\overline \cM(\cL)$ remain invariant under rescaling. From this point of view,  the fact that $p$ belongs to $\overline \cM(\cL)$ may not be interpreted in general as a genuine condition on the shape of $p$. Nevertheless, \eref{eq-condcL} and \eref{eq-compare}
share some similarities. On the one hand, our condition \eref{eq-condcL} only requires that the inequality hold in \eref{eq-compare} for given values of $\sigma=\lambda$ and $c=\kappa$, on the other hand we impose the constraint that these values be both larger than 1. 

The functions $\cL = \cL_s$ are differentiable on $(0,+\infty)$ and they satisfy the assumptions of Theorem \ref{thm-approx-scv} when $s\in (-1,1]$. More precisely, since $\cL_{s}'(y)=|s|y^{s-1}$ for $s\in (-1,1]\setminus\{0\}$ and $\cL_{0}'(y)=y^{-1}$ for all $y>0$, $\lambda^{2}\cL_{s}'(\lambda y)/\cL_{s}'(y)=\lambda^{s+1}$ for all $\lambda>1$ and we may then choose $\kappa=\lambda^{s+1}$. Consequently, for a given choice of $\lambda>1$, the value of $\gamma$ involved in \eref{eq-bd-approx-scv} can be chosen as 
\[
\gamma=\gamma(\lambda,\lambda^{s+1})=\cro{\frac{2}{(s+1)\log \lambda}\pa{\frac{\lambda^{s+1}}{\lambda^{s+1}-1}}^{1/2}+\frac{\lambda^{(1-s)/2}\sqrt{\lambda^{s+1}-1}
}{\sqrt{2}}\pa{1+\frac{1}{2(\sqrt{\lambda^{s+1}}-1)}}}^{2}.
\]
 For values of $s\in\{-1/2,0,1/2,1\}$ and the choice $\lambda=3.3$, which approximately gives the minimal value of $\lambda\mapsto \gamma(\lambda,\lambda^{s+1})$, we obtain the following upper bounds for $\gamma=\gamma(s)$
\begin{equation}\label{eq-gamma(s)}
\gamma(-1/2)<130,\ \gamma(0)<26.5,\ \gamma(1/2)=10.1\quad \text{and}\quad \gamma(1)<4.8.
\end{equation}

\subsection{Risk bounds of the TV-estimator over sets of $\cL$-concave densities.}
Combining the risk bound established for the TV-estimator $\widehat P=\widehat p\cdot \mu$ 
in Theorem~\ref{thm-cL-concave} with the approximation properties of the sets $\overline \cO_{D}(\cL)$ provided by Proposition~\ref{prop-astGuillaume} and Theorem~\ref{thm-approx-scv} , we derive the following result.     

\begin{cor} \label{cor-cL-cve-rates}
Let $X_1,\ldots,X_n$ be an $n$-sample with density $p\et \in \overline{\cM}(\cL)$ and $\widehat p\in \overline{\cM}(\cL)$ the density of the TV-estimator. The following results hold.
\begin{itemize}
\item[(i)] If $\cL$ is convex and increasing on $(0,+\infty)$,  
\begin{equation}\label{eq-cL-cve-rates}
\E\cro{\norm{p\et-\widehat p}}\le \frac{192}{n^{2/5}}+\frac{96\sqrt{10}}{\sqrt{n}}+\frac{2\varepsilon}{n}.
\end{equation}
In particular, this result applies to the family of $s$-concave densities $\overline \cM(\cL_{s})$ with $s\ge 1$. 
\item[(ii)] if $\cL$ is concave, increasing on $(0,+\infty)$ and satisfies \eref{eq-condcL} for some constants $\lambda,\kappa>1$, 
\begin{equation}\label{eq-cL-cve-rates2}
\E\cro{\norm{p\et-\widehat p}}\le  \frac{192 \gamma^{1/5}}{n^{2/5}}+\frac{96\sqrt{14}}{\sqrt{n}}+\frac{2\varepsilon}{n},
\end{equation}
where the constant $\gamma$ is given by \eref{eq-cL-defR}. In particular, this result applies to the family of $s$-concave densities $\overline \cM(\cL_{s})$ with $s\in (-1,1)$. In the special case $s=0$, $\overline \cM(\cL_{s})$ is the family of log-concave densities and one may take $\gamma=26.5$. 
\end{itemize}
\end{cor}

\begin{proof}
For $a,b,c>0$, define on the set of positive integers $D$, 
\[
F(D)=\frac{a}{D^{2}}+b\sqrt{\frac{D-1}{n}}+c.
\]
Then, for the choice 
\[
D\et=\PES{\pa{\frac{a}{b}}^{2/5}n^{1/5}}\quad \text{which satisfies}\quad 1\vee \cro{\pa{\frac{a}{b}}^{2/5}n^{1/5}}\le D\et\le \cro{\pa{\frac{a}{b}}^{2/5}n^{1/5}}+1
\]
we obtain that 
\begin{equation}\label{eq-FD}
\inf_{D\ge 1}F(D)\le F(D\et)\le 2a^{1/5}b^{4/5}n^{-2/5}+c.
\end{equation}

When $\cL$ is convex, it follows from \eref{eq-Thm-cL-concave}, \eref{eq-astGuillaume1} and the subadditivity of the square root that for all $D\ge 1$, 
\begin{align*}
\E\cro{\norm{p\et-\widehat p}}&=2\E\cro{\dTV{P \et}{\widehat P}}\leq  \inf_{D\ge 1}\cro{3\inf_{s\in\overline \cO_{2D+1}(\cL)}\norm{p\et-s} + 96\sqrt{2}\sqrt{\frac{2D+3}{n}}}+\frac{2\varepsilon}{n}\\
&\le \inf_{D\ge 1}\cro{\frac{6}{D^{2}} + 96\sqrt{2}\sqrt{\frac{2(D-1)+5}{n}}}+\frac{2\varepsilon}{n}\le \inf_{D\ge 1}\cro{\frac{6}{D^{2}} + (96\times 2)\sqrt{\frac{D-1}{n}}}+ \frac{96\sqrt{10}}{\sqrt{n}}+\frac{2\varepsilon}{n}.
\end{align*}
We obtain \eref{eq-cL-cve-rates} by applying~\eref{eq-FD} with $a=6$, $b=96\times 2$ and $c=96\sqrt{10}/\sqrt{n}+2\eps/n$. 

When $\cL$ is concave and fulfils the requirements of Theorem~\ref{thm-approx-scv}, we may argue similarly 
by applying \eref{eq-bd-approx-scv} in place of \eref{eq-astGuillaume1} and by using the fact that for all $D\ge 1$, 
\begin{align*}
\E\cro{\norm{p\et-\widehat p}}
&\le \inf_{D\ge 1}\cro{\frac{6 \gamma}{D^{2}} + 96\sqrt{2}\sqrt{\frac{2(D-1)+7}{n}}}+\frac{2\varepsilon}{n}\le \inf_{D\ge 1}\cro{\frac{6\gamma}{D^{2}} + 96\times 2\sqrt{\frac{D-1}{n}}}+ \frac{96\sqrt{14}}{\sqrt{n}}+\frac{2\varepsilon}{n}.
\end{align*}
\end{proof}

\section{Concluding remarks}\label{sect-conclusion}

\subsection{About the computational aspects of the TV-estimators}
The estimation procedure that we propose here requires very few assumptions on the statistical model and the density to be estimated. From this point of view, it provides a competitor to $\rho$-estimators. As for the latter, it is unlikely that an algorithm for calculating TV-estimators can be designed at this level of generality. A more reasonable approach is to see how the estimator can be calculated in a specific situation that is, for a given statistical model, as we would do for the MLE. It is likely that the complexity of the algorithm will be related not only to the size of the model the statistician considers but also to the assumptions he or she wants to make on the target density. For illustration, let us consider the case of the Grenander estimator. In this situation, the statistician has a nice expression of the optimizer of the likelihood function because he or she has beforehand restricted the optimization to the set of densities that are supported on a given half-line (restriction to a specific model) and assumed that all the data do belong to this half-line (assumption on the target density). The MLE would not exist otherwise. Under the same assumptions, it is worth noticing that the calculation of the $\rho$-estimator becomes tractable as well (we merely recovers the MLE). The drawback of these somewhat restrictive assumptions, which make the calculability of the MLE possible and easy (as well as that of the $\rho$-estimator), lie in the fact that a very small error on the support of the density may result in a large distance between the true density and the model and, in turn, to a poor performance of these estimators. This is what we have shown in our Introduction. By considering the whole set of nonincreasing densities on a half-line, the TV- (and $\rho$-) estimator would not suffer from these weaknesses but may be more difficult to compute. This example shows that it seems difficult to disentangle the computational issues from the properties we wish the resulting estimator to achieve. We believe that these computational aspects, of both the TV- and $\rho$-estimators, are an area of research on their own and that solving these computational issues may require different sets of assumptions from those which are needed to establish their mathematical properties. 

\subsection{About the limitation of the approach for estimating a density under a shape constraint}
Our approach offers the advantage that it solves in a uniform framework various problems of estimation of a density under a shape constraint. It provides an estimator that is well-defined, even for sets of densities on which the MLE does not exist, and it leads to an estimator that enjoys some desirable adaptation properties. When  the data are i.i.d.\ with a distribution $P_{0}$ whose density has the expected shape as well as some additional features, the risk bound that we establish for our TV-estimator is of order $1/\sqrt{n}$. Interestingly, the estimator is robust to a departure from this ideal situation. When the data are only independent
but the average $P\et=n^{-1}\sum_{i=1}^{n}P_{i}\et$ of their marginals lies close enough to $P_{0}$, say in a ball (with respect to the total variation distance) centered at $P_{0}$ with radius $r>0$, the risk bound of the TV-estimator does not deteriorate by more than the additional term $3r$. This stability of the risk is the key ingredient that allows us to establish not only the robustness of our estimator but also, by means of approximation theory, some uniform $\L_{1}$-risk bounds over classes of densities of interest. However, this versatile approach suffers from some restrictions: it mainly applies to densities on the line. This is due to the fact that the level sets of unimodal, convex, concave or log-concave functions on the line are quite simple. They take the form of an interval or a union of two intervals. The situation dramatically changes in higher dimensions and the notion of extremal points, that we have introduced here, becomes much less interesting. In dimension 2, the class of all uniform densities on polygons, for example, does not possess any extremal points. It is unclear to us whether or not this limitation to the line is due to our approach or is rather inherent to the problem of robustness (or stability) which we want to solve.
\section{Proofs}\label{sect-7}

\subsection{Some useful lemmas for approximating densities and convex functions with the $\L_{1}$-norm}\label{subsect-7.1}

The following result shows that normalizing to one a nonnegative approximating function of a density cannot worsen its  $\L_{1}$-error by more than a factor 2.
\begin{lem} \label{lem_l1_proj}
Let $p$ be a density on a measured space $(E,\cE,\nu)$ and $f$ a nonnegative integrable function on $(E,\cE,\nu)$ which is not $\nu$-a.e.\ equal to 0 on $E$. 
Then
\[ 
\int_E \left|p- \frac{f}{\int_E f d\nu}\right| d\nu \leq 2 \cro{1\wedge \int_E |p - f| d\nu}. 
\]
\end{lem}
This inequality cannot be improved in general since equality holds when $f=p\1_{I}$ and $I$ is a measurable subset of $E$ on which $p$ is not $\nu$-a.e.\ equal to 0. 

\begin{proof}
The fact that 
\[
\int_E \left|p- \frac{f}{\int_E f d\nu}\right|d\nu\le 2
\]
comes from the triangle inequality. Let us now prove that 
\[
\int_E \left|p- \frac{f}{\int_E f d\nu}\right| d\nu \leq 2 \int_E |p - f| d\nu.
\]
We first assume that $c = \int_{E} f d\nu\in (0,1]$. Since $f/c$ and $p$ are two densities, 
\begin{align*}
\int_E \left|\frac{f}{c} - p \right| d\nu &= 2 \int_I \left[ p - \frac{f}{c} \right] \1_{\{c p \geq f\}} d\nu \leq 2 \int_E [p - f] \1_{\{c p \geq f\}} d\nu\leq 2 \int_E |p - f| d\nu. 
\end{align*}
This proves the lemma when $c \in (0,1]$. let us now assume that $c > 1$.
The previous case of the lemma applies to the density $f/c$ and the nonnegative function $p/c$ the integral of which is not larger than 1. This  yields
\begin{align*}
\int_E \left| \frac{f}{c} - p \right|d\nu 
&\leq 2 \int_E \left|\frac{f}{c} - \frac{p}{c} \right| d\nu=\frac{2}{c}\int_I |f - p| d\nu \leq 2 \int_E|f - p| d\nu.
\end{align*}
\end{proof}

The following lemma applies to approximation functions that are unimodal.   
\begin{lem} \label{lem-scale_dens}
Let $f$ be an integrable function which is nondecreasing on $(-\infty,c)$ and nonincreasing on $(c,+\infty)$ for some $c \in \mathbb{R}$, and such that $\kappa=\int_{-\infty}^{+\infty} f(x) dx \in (0,1]$.
Then, for any probability density $p$ on the line,
\[ \int_{-\infty}^{+\infty} \left| p(x) - f\left(c + \kappa(x-c) \right) \right| dx \le 2 \int_{-\infty}^{+\infty} |p(x) - f(x)| dx. \] 
\end{lem}
The constant 2 is optimal as shown by the following example. Given $\eps\in (0,1)$, let $f=\eps \1_{[-1/2,1/2]}$ and  $p=\1_{[1/\eps,1+1/\eps]}$ so that $\kappa=\eps$ and one may take $c=0$. The function $x\mapsto f(\kappa x)$ is then the uniform density on $[-1/(2\eps),1/(2\eps)]$ and its support, as that of $f$, is disjoint from
that of $p$. We conclude that
\[
\int_{-\infty}^{+\infty} \left| p(x) - f\left(\kappa x\right) \right| dx=2\quad \text{and}\quad  \int_{-\infty}^{+\infty} |p(x) - f(x)| dx=1+\eps.
\]
\begin{proof}
Since the Lebesgue measure is translation invariant, we may assume without loss of generality that $c = 0$. Let $x \in \mathbb{R}$. If $x < 0$, $\kappa x> x$ and $f(\kappa x)\ge f(x)$ since $f$ is nondecreasing on $(-\infty, 0)$ and if $x>0$, $\kappa x< x$ and $f(\kappa x)\ge f(x)$ since $f$ is nonincreasing on $(0,+\infty)$. We deduce that for all $x\in\R$, $f(\kappa x)\ge f(x)$ and consequently, 
\begin{align*}
\int_{- \infty}^{+\infty} \left| f \left( \kappa x \right) - f(x)  \right| dx &=
\int_{- \infty}^{+\infty} f \left( \kappa x \right) dx -  \int_{-\infty}^{+\infty} f(x)  dx = \frac{1}{\kappa}\int_{- \infty}^{+\infty} f \left( x \right) dx - \int_{-\infty}^{+\infty} f(x)  dx \\
&= 1- \int_{-\infty}^{+\infty} f(x)  dx=\int_{-\infty}^{+\infty} p(x) dx - \int_{-\infty}^{+\infty} f(x)  dx \leq \int_{-\infty}^{+\infty} |p(x) - f(x)| dx.
\end{align*}
Using the triangle inequality, we conclude that 
\begin{align*}
\int_{- \infty}^{+\infty} \left| f \left( \kappa x \right) - p(x)  \right| dx &\le \int_{- \infty}^{+\infty} \left| f \left( \kappa x  \right) - f(x)  \right| dx + \int_{-\infty}^{+\infty} |p(x) - f(x)| dx \le 2 \int_{-\infty}^{+\infty} |p(x) - f(x)| dx.
\end{align*}
\end{proof}

The following result shows how well a convex (or concave) function can be approximated by its chord, that is, its $1$-linear interpolation on a compact interval. 
\begin{prop}\label{prop-approx-convexe}
Let $f$ be a convex-concave continuous function on a bounded nontrivial interval $[a,b]$ and $\ell_{f}$ be the linear function
\begin{equation}\label{def-lf}
\ell_{f}:x\mapsto f(a)+\frac{f(b)-f(a)}{b-a}\pa{x-a}.
\end{equation}
The following results hold.
\begin{itemize}
\item[(i)]\label{cas-conv1} If $f$ admits a right derivative $f'_{r}(a)$ at $a$ and a left derivative $f'_{l}(b)$ at $b$,
\begin{equation}\label{eq-approx-convexe}
\int_{a}^{b}\ab{f-\ell_{f}}d\mu\le \frac{(b-a)^{2}}{8}\ab{f'_{r}(a)-f'_{l}(b)}.
\end{equation}
\item[(ii)]\label{cas-conv2} If $f$ is strictly monotone on $[a,b]$ with $f'_{r}(a)\ne 0$ and $f'_{l}(b)\ne 0$
\begin{equation}\label{eq-approx-convexe-01}
\int_{a}^{b}\ab{f-\ell_{f}}d\mu\le \frac{(f(b)-f(a))^{2}}{8}\ab{\frac{1}{f'_{r}(a)}-\frac{1}{f'_{l}(b)}},
\end{equation}
with the convention $1/(\pm \infty)=0$.
\end{itemize}
\end{prop}
\begin{proof}
Let us first assume that $f$ is concave on $[0,1]$,  admits a right derivative at 0, a left derivative at 1 and satisfies $f(0)=0$ and $f(1)$=1. Let us show that 
\begin{equation}\label{eq-approx-fc00}
\int_{0}^{1}\ab{f-\ell_{f}}d\mu\le \frac{1}{2}\frac{\pa{f_{r}'(0)-1}\pa{1-f'_{l}(1)}}{f'_{r}(0)-f'_{l}(1)}
,
\end{equation}
with the convention $0/0=1$. If $f$ is linear on $[0,1]$, then $f_{r}'(0)=f'_{l}(1)=1$ and the inequality is satisfied with our convention. Otherwise, $f_{l}'(1)< 1=f(1)-f(0)<f_{r}'(0)$ and since $f$ lies above its chord and under its tangents at 0 and 1, 
\[
\ell_{f}(x)\le f(x)\le\min\ac{f'_{r}(0)x,1+f_{l}'(1)(x-1)} \quad \text{for all $x\in [0,1]$}.
\]
Since $\ell_{f}(x)=x$, we deduce that for all $c\in [0,1]$
\begin{align*}
\int_{0}^{1}\ab{f-\ell_{f}}d\mu&\le \int_{0}^{c}(f'_{r}(0)-1)xd\mu+\int_{c}^{1}(1-f_{l}'(1))(1-x)d\mu\\
&=\frac{1}{2}\cro{(f'_{r}(0)-1)c^{2}+(1-f_{l}'(1))(1-c)^{2}}.
\end{align*}
The result follows by minimizing with respect to $c$, i.e.\ by taking $c=(1-f_{l}'(1))/(f_{r}'(0)-f_{l}'(1))\in (0,1)$. 
In particular, we deduce from \eref{eq-approx-fc01} that 
\begin{align*}
\int_{0}^{1}\ab{f-\ell_{f}}d\mu\le \frac{c(1-c)}{2}\pa{f'_{r}(0)-f'_{l}(1)}
\end{align*}
and since $c(1-c)\le 1/4$, we obtain that 
\begin{equation}\label{eq-approx-fc01}
\int_{0}^{1}\ab{f-\ell_{f}}d\mu\le \frac{1}{8}\pa{f'_{r}(0)-f'_{l}(1)}.
\end{equation}
Note that the inequality also holds when $f'_{r}(0)=f'_{l}(1)=1$. 

When $f$ is increasing on $[0,1]$ and satisfies $0<1-f_{l}'(1)<1$, i.e.\ $f_{l}'(1)\ne 0$, we also deduce from \eref{eq-approx-fc01} and the convexity of $z\mapsto 1/z$ on $(0,+\infty)$ that 
\begin{align*}
\int_{0}^{1}\ab{f-\ell_{f}}d\mu&\le \frac{1}{4}\frac{1}{\frac{1}{2}\cro{\pa{\frac{1}{1-f'_{l}(1)}-1}+\pa{1+\frac{1}{f'_{r}(0)-1}}}}\le \frac{1}{4}\cro{\frac{1}{2}\pa{\frac{1}{\frac{1}{1-f'_{l}(1)}-1}+\frac{1}{1+\frac{1}{f'_{r}(0)-1}}}}\\
&= \frac{1}{8}\cro{\frac{1-f'_{l}(1)}{f'_{l}(1)}+\frac{f'_{r}(0)-1}{f'_{r}(0)}}
\end{align*}
which leads to 
\begin{equation}\label{eq-approx-fc02}
\int_{0}^{1}\ab{f-\ell_{f}}d\mu\le \frac{1}{8}\pa{\frac{1}{f'_{l}(1)}-\frac{1}{f'_{r}(0)}}.
\end{equation}
Note that the inequality is still satisfied when $f'_{r}(0)=+\infty$ with the convention $1/(+\infty)=0$. 

Let us now turn to the proofs of~\eref{eq-approx-convexe} and~\eref{eq-approx-convexe-01}. Note that ~\eref{eq-approx-convexe} is clearly true when $f$ is constant on $[a,b]$ and we may therefore assume that $f(a)\ne f(b)$. We obtain \eref{eq-approx-convexe} and \eref{eq-approx-convexe-01} by applying \eref{eq-approx-fc01} and \eref{eq-approx-fc02} respectively  to the function 
\[
g(x)=\frac{f(a+x(b-a))-f(a)}{f(b)-f(a)}
\]
when $f$ is concave and satisfies $f(b)>f(a)$ or when $f$ is convex and satisfies $f(a)>f(b)$. In the other cases, one may use the function 
\[
g(x)=\frac{f(b-x(b-a))-f(b)}{f(a)-f(b)}.
\]
\end{proof}

The following result is an extension of  Proposition~\ref{prop-approx-convexe}. It shows how well a convex (or concave) function can be approximated by a suitable $D$-linear interpolation on a compact interval. 
\begin{prop}\label{prop-guer1}
Let $D$ be some positive integer and $f$ a convex-concave function on a non-trivial interval $[a,b]$ such that 
\[
\Delta=\frac{f(b)-f(a)}{b-a}\ne 0.
\]
\begin{itemize}
\item[(i)] If $f$ admits a right derivative at $a$ and a left derivative at $b$, there exists a $D$-linear interpolation $\ell_{1}$ of $f$ such that 
\begin{align}
\int_{a}^{b}\ab{f-\ell_{1}}d\mu&\le \frac{(b-a)^{2}}{2D^{2}}\frac{\ab{f_{r}'(a)-\Delta}\ab{\Delta-f'_{l}(b)}}{\ab{f'_{r}(a)-f'_{l}(b)}}\le \frac{(b-a)^{2}}{8D^{2}}\ab{f'_{r}(a)-f'_{l}(b)},\label{prop-guer1a}
\end{align}
with the convention $0/0=0$ in the right-hand side of \eref{prop-guer1a}.
\item[(ii)] If $f$ is strictly monotone on $[a,b]$ with $f_{r}'(a)\ne 0$ and $f_{r}'(b)\ne 0$, there exists a $D$-linear interpolation $\ell_{2}$ of $f$
\begin{equation}
\int_{a}^{b}\ab{f-\ell_{2}}d\mu\le \frac{\pa{f(a)-f(b)}^{2}}{8D^{2}}\ab{\frac{1}{f'_{r}(a)}-\frac{1}{f'_{l}(b)}}\label{prop-guer1b}
\end{equation}
with the convention $1/(\pm \infty)=0$.

\item[(iii)] If $f$ is monotone and convex or concave on $[a,b]$, there exists a $D$-linear interpolation $\ell_{3}$ of $f$ such that 
\begin{equation}
\int_{a}^{b}\ab{f-\ell_{3}}d\mu\le \frac{(b-a)\ab{f(a)-f(b)}}{2D^{2}}.\label{prop-guer1c}
\end{equation}
\end{itemize}
\end{prop}

\begin{proof}
If $f$ is a concave and monotone function on $[0,1]$ satisfying $f(0)=0$ and $f(1)=1$, Gu\'erin {\em et al}~\cite{Guerin2006} proved that there exists a $D$-linear interpolation $\ell$ of $f$ on $[0,1]$ such that 
\[
\int_{0}^{1}\ab{f-\ell}d\mu\le \frac{1}{2D^{2}}.
\]
If $f$ admits a right derivative at 0 and a left derivative at 1, they proved that there exists a $D$-linear interpolation $\ell$ of $f$ on $[0,1]$ such that 
\begin{equation}\label{eq-guer1}
\int_{0}^{1}\ab{f-\ell}d\mu\le \frac{1}{2D^{2}}\frac{\pa{f'_{r}(0)-1}\pa{1-f_{l}'(1)}}{f'_{r}(0)-f'_{l}(1)}.
\end{equation}
The results established in Proposition~\ref{prop-guer1} are deduced from~\eref{eq-guer1} by arguing as in the proof of Proposition~\ref{prop-approx-convexe}.
\end{proof}

\subsection{Proofs of Section \ref{sec-ell-estimator}}\label{subsect-proof-ell-estimator}

\begin{proof}[Proof of Theorem \ref{shape-estimation-th}]
Let $\frD\ge 1$ such that $\overline \cO(\frD)$ is not empty. Such an integer $\frD$ exists since $\overline \cO$ is nonempty. Let $\overline P=\overline p\cdot \mu$ be an arbitrary point in $\sO(\frD)$ with $\overline p\in\cO(\frD)$.  For $P,Q\in\sM$ and $\zeta\ge 0$, we set 
\begin{align*}
\gZ_+(\bsX,P)&=\sup_{Q\in\sM}\cro{\gT(\bsX,P,Q)-\E\cro{\gT(\bsX,P, Q)}}-\zeta\\
\gZ_{-}(\bsX,P)&=\sup_{Q\in\sM}\cro{\E\cro{\gT(\bsX,Q,P)}-\gT(\bsX,Q,P)}-\zeta
\end{align*}
and 
\[
\gZ(\bsX,P)=\gZ_+(\bsX,P)\vee\gZ_{-}(\bsX,P).
\]
Applying the first inequality of \eqref{bound_test_TV_eq} with $P=Q$ and $Q=\overline{P}$, we infer that for all $Q\in\sM$, 
\begin{align*}
n\dTV{\overline{P}}{Q} 
&\leq n\dTV{P \et}{\overline{P}}+\E\cro{\gT(\bsX,Q,\overline P)}\\
&=n\dTV{P \et}{\overline{P}}+\E\cro{\gT(\bsX,Q,\overline P)}-\gT(\bsX,Q,\overline P)+\gT(\bsX,Q,\overline P)\\
&\leq n\dTV{P \et}{\overline{P}}+\gZ(\bsX,\overline P)+\gT(\bsX,Q,\overline P)+\zeta\\
&\leq n\dTV{P \et}{\overline{P}}+\gZ(\bsX,\overline P)+\gT(\bsX,Q)+\zeta.
\end{align*}
In particular, the inequality applies to $Q=\widehat P\in \sE(\bsX)$ and using the fact that 
\[
\gT(\bsX,\widehat P)\leq \inf_{P \in\sM}\gT(\bsX,P)+\varepsilon \leq \gT(\bsX,\overline P)+\varepsilon,
\]
we deduce that 
\begin{align}
n\dTV{\overline{P}}{\widehat P}&\leq n\dTV{P \et}{\overline{P}}+\gZ(\bsX,\overline P)+\gT(\bsX,\overline P)+\zeta+\varepsilon.\label{eq-fond00}
\end{align}
Using now the second inequality of \eqref{bound_test_TV_eq} with $P=\overline P$, we  obtain that
\begin{align*}
\gT(\bsX,\overline P)&=\sup_{Q\in\sM}\gT(\bsX,\overline P,Q)\\
&\leq \sup_{Q\in\sM}\cro{\gT(\bsX,\overline P,Q)-\E\cro{\gT(\bsX,\overline P,Q)}-\zeta}+\sup_{Q\in\sM}\E\cro{\gT(\bsX,\overline P,Q)}+\zeta\\
&\leq \gZ(\bsX,\overline P)+n\dTV{P \et}{\overline{P}}+\zeta,
\end{align*}
which with~\eref{eq-fond00} lead to 
\begin{equation}\label{eq-fond01}
n\dTV{\overline{P}}{\widehat P}\leq 2n\dTV{P \et}{\overline{P}} + 2\zeta+\varepsilon+2\gZ(\bsX,\overline P).
\end{equation}
Let us now bound from above $\gZ(\bsX,\overline P)$. We set for $P\in\sM$
\begin{align*}
\gw(P)&=\E\cro{\sup_{Q\in\sM}\cro{\gT(\bsX,P,Q)-\E\cro{\gT(\bsX,P, Q)}}}
\vee \E\cro{\sup_{Q\in\sM}\cro{\E\cro{\gT(\bsX,Q, P)}-\gT(\bsX,Q,P)}}.
\end{align*}
The functions $t_{(\overline P,Q)}$ satisfy $|t_{(\overline P,Q)}(x)-t_{(\overline P,Q)}(x')|\le 1$ for all $Q\in\sM$ and $x,x'\in E$, hence 
\[
\ab{\gZ_{+}((x_{1},\ldots,x_{i},\ldots,x_{n}),\overline P)-\gZ_{+}((x_{1},\ldots,x_{i}',\ldots,x_{n}),\overline P)}\leq 1
\]
for all $\gx\in\gE$, $x_{i}'\in E$ and $i\in\{1,\ldots,n\}$. Arguing as in the proof of Lemma 2 of Baraud~\cite{BY-TEST} (with $\xi +\log 2$ in place of $\xi$), we deduce that with a probability at least $1-(1/2)e^{-\xi}$, 
\begin{align}
\gZ_{+}(\bsX,\overline P)&\le \E\cro{\gZ_{+}(\bsX,\overline P)}+\sqrt{\frac{n(\xi +\log 2)}{2}}\\
&=\E\cro{\sup_{Q\in\sM}\cro{\gT(\bsX,\overline P,Q)-\E\cro{\gT(\bsX,\overline P, Q)}}}+\sqrt{\frac{n(\xi +\log 2)}{2}}-\zeta\nonumber\\
&\le \gw(\overline P)+\sqrt{\frac{n(\xi +\log 2)}{2}}-\zeta.\label{eq-thm1-007}
\end{align}
Arguing similarly, with a probability at least $1-(1/2)e^{-\xi}$, 
\begin{align}
\gZ_{-}(\bsX,\overline P)&\le \gw(\overline P)+\sqrt{\frac{n(\xi +\log 2)}{2}}-\zeta.\label{eq-thm1-007b}
\end{align}
Putting \eref{eq-thm1-007} and \eref{eq-thm1-007b} together and choosing $\zeta= \gw(\overline P)+\sqrt{n(\xi +\log 2)/2}$, we obtain that with a probability at least $1-e^{-\xi}$,
\[
\gZ(\bsX,\overline P)=\gZ_+(\bsX,P)\vee\gZ_{-}(\bsX,P)\le \gw(\overline P)+\sqrt{\frac{n(\xi +\log 2)}{2}}-\zeta\le 0
\]
which with \eref{eq-fond01} lead to the bound
%
\begin{equation}\label{eq-thm1-final}
\dTV{\overline{P}}{\widehat P}\leq 2\dTV{P \et}{\overline{P}} + \frac{2\gw(\overline P)}{n}+\sqrt{\frac{2(\xi +\log 2)}{n}}+\frac{\varepsilon}{n}.
\end{equation}
It remains now to control $\gw(\overline P)$. Since $\overline p\in\cO(\frD)\subset \overline \cO(\frD)$, it is extremal in $\overline \cM\supset \cM$ with degree not larger than $\frD$, the classes $\ac{\{q<\overline p\},\; q\in\cM\setminus\{\overline p\}}$ and $\ac{\{q>\overline p\},\; q\in\cM\setminus\{\overline p\}}$ are both VC with dimensions not larger than $\frD$. We may therefore apply Proposition 3.1 in Baraud~\cite{Bar2016} with $\sigma=1$ and get 
\begin{align}
\E\cro{\sup_{q\in\cM\setminus\{\overline p\}}\ab{ \sum_{i=1}^{n}\pa{\1_{\overline p>q}(X_{i})-P_i\et(\overline p>q)}}}
&\leq 10\sqrt{5n\frD},\label{eq-Besup00}
\end{align}
and 
\begin{align}
\E\cro{\sup_{q\in\cM\setminus\{\overline p\}}\ab{ \sum_{i=1}^{n}\pa{\1_{\overline p<q}(X_{i})-P_i\et(\overline p<q)}}}
&\leq 10\sqrt{5n\frD}.\label{eq-Besup01}
\end{align}
These inequalities entail $\gw(\overline P)\leq 10\sqrt{5n\frD}$, and we infer from~\eref{eq-thm1-final} that 
\begin{equation} \label{eq_ubd_dist_pbar_phat}
  \dTV{\overline{P}}{\widehat P} \leq 2\dTV{P \et}{\overline{P}} + 20\sqrt{\frac{5\frD}{n}}+\sqrt{\frac{2(\xi +\log 2)}{n}}+\frac{\varepsilon}{n}.  
\end{equation}
Since $\overline P$ is arbitrary in the set $\sO(\frD)$ which is dense on $\overline \sO(\frD)$, we infer that equation \eqref{eq_ubd_dist_pbar_phat} holds for all $\overline P \in \overline \sO(\frD)$, which yields \eqref{inthm_ubd_dist_p_phat}.
Hence, by the triangle inequality,
\begin{align*}
  \dTV{P \et}{\widehat P}  &\leq \inf_{P\in\overline \sO(\frD)} \left\{ \dTV{P \et}{P} +  \dTV{P}{\widehat P} \right\} \\
  &\leq 3\inf_{P\in\overline \sO(\frD)}\dTV{P \et}{P} + 20\sqrt{\frac{5\frD}{n}}+\sqrt{\frac{2(\xi +\log 2)}{n}}+\frac{\varepsilon}{n}.
\end{align*}
With our convention that $\inf_{\varnothing}=+\infty$, the inequality is also true when  $\overline \sO(\frD)=\varnothing$, hence for all values of $\frD$, which  leads to \eref{Thm1_risk_bound_deviation_eq}. Inequality \eref{Thm1_risk_bound_expecation_eq} follows by integrating this deviation bound with respect to $\xi$. 
\end{proof}

\subsection{Proofs of Section \ref{sect-kmono}}\label{subsect-proof-kmono}

\begin{proof}[Proof of Proposition \ref{prop-birge}]
We restrict ourselves to the case where $p$ is nonincreasing on $I$, the proof in the other case is similar. Let $q$ be the function that coincides with $p$ on $\mathring I=(a,b)$ and satisfies $q(a)=\sup_{x\in (a,b)}p(x)$ and $q(b)=\inf_{x\in(a,b)}p(x)$. Clearly, $p=q$ a.e.\ and satisfies $V_{I}(p)=q(a)-q(b)=V_{I}(q)$. Without loss of generality, we may therefore assume that $I=[a,b]$ and that $V_{I}(p)=p(a)-p(b)$. 

Since $p$ is nonincreasing in $I$, for all intervals $J\subset I$ with endpoints $u<v$, 
\begin{equation}\label{eq-lem-lulu}
\int_{J}\ab{p-\overline p_{J}}d\mu\le \frac{(v-u)(p(u)-p(v))}{2}.
\end{equation}
In particular, when $D=1$ it suffices to take $\cJ=\{I\}$ and the result follows from \eref{eq-lem-lulu} with $u=a$ and $v=b$ and the trivial inequality $\int_{I}|p-\overline p_{J}|d\mu\le 2$. It remains to prove the result for  $D\ge 2$ and since \eref{eq-lem-birge} is trivially true for $V=0$ we may also assume that $V>0$. 

Let us set $\eta=(1+VL)^{1/D}-1>0$, $x_{0}=a$ and for all $j\in\{1,\ldots,D\}$, 
\[
x_{j}=x_{j-1}+L\frac{(1+\eta)^{j}}{\sum_{k=1}^{D}(1+\eta)^{k}}=x_{0}+L\frac{(1+\eta)^{j}-1}{(1+\eta)^{D}-1}=x_{0}+\frac{(1+\eta)^{j}-1}{V}.
\]
Then, we obtain an increasing sequence of points $a=x_{0}<x_{1}<\ldots<x_{D}=b=a+L$ and a partition $\cJ$ of $I$ into $D$ intervals based on $\{x_{0},\ldots,x_{D}\}$. Using \eref{eq-lem-lulu} and the facts that $x_{j+1}-x_{j}=(1+\eta)(x_{j}-x_{j-1})>x_{j}-x_{j-1}$ for $j\in\{1,\ldots,D-1\}$, we obtain 
\begin{align*}
\int_{I}\ab{p-\overline p}d\mu&=\sum_{J\in\cJ}\int_{J}\ab{p-\overline p_{J}}d\mu\le \frac{1}{2}\sum_{j=1}^{D}(x_{j}-x_{j-1})[p(x_{j-1})-p(x_{j})]\\
&=\frac{1}{2}\cro{p(x_{0})(x_{1}-x_{0})+\sum_{j=1}^{D-1}p(x_{j})\cro{(x_{j+1}-x_{j})-(x_{j}-x_{j-1})}}
-\frac{p(x_{D})(x_{D}-x_{D-1})}{2}\\
&\le \frac{1}{2}\cro{V_{I}(p)(x_{1}-x_{0})+\eta\sum_{j=1}^{D-1}p(x_{j})(x_{j}-x_{j-1})}
+\frac{p(x_{D})\cro{(x_{1}-x_{0})-(x_{D}-x_{D-1})}}{2}\\
&\le  \frac{1}{2}\cro{V(x_{1}-x_{0})+\eta\sum_{j=1}^{D-1}\int_{x_{j-1}}^{x_{j}}pd\mu}\le \frac{1}{2}\cro{V(x_{1}-x_{0})+\eta}=\eta.
\end{align*}
Together with the trivial bound $\int_{I}\ab{p-\overline p}d\mu\le 2$, this last inequality leads to \eref{eq-lem-birge}. The second inequality derives from the fact that $(e^{x}-1)\wedge 2\le 2x/\log 3\le 2x$ for all $x\ge 0$.
\end{proof}

\begin{proof}[Proof of Theorem~\ref{thm-3}]
Let $D,D_{1},\ldots,D_{k-2}$ be positive integers and $p$ a density in $\overline \cM_{k}^{\infty}(R)$. We may therefore write $p$ in the form~\eref{def-mbark} with 
\[
\cro{\sum_{i=1}^{k-2}\sqrt{w_{i}\log\pa{1+L_{i}V_{i}}}}^{2}\le R.
\]
Applying Proposition~\ref{prop-birge} to the density $p_{i}$, with $I=I_{i}=(x_{i-1},x_{i})$, $L=L_{i}=(x_{i}-x_{i-1})$ and $V=V_{i}$ for each $i\in\{1,\ldots,k-2\}$, we build a monotone density $\overline p_{i}$ on $I_{i}$ which is piecewise constant on a partition of $I_{i}$ into $D_{i}\ge 1$ nontrivial intervals and that satisfies 
\[
\int_{I_{i}}\ab{p_{i}-\overline p_{i}}d\mu\le \frac{2S_{i}}{D_{i}}\quad \text{with}\quad S_{i}=\log\pa{1+V_{i}L_{i}}.
\]
Let us now take $D_{i}=\left\lceil D\sqrt{w_{i}S_{i}}/(\sum_{i=1}^{k-2}\sqrt{w_{i}S_{i}})\right\rceil\vee 1$ for all $i\in\{1,\ldots,k-2\}$. Since 
\[
\frac{D\sqrt{w_{i}S_{i}}}{\sum_{i=1}^{k-2}\sqrt{w_{i}S_{i}}}\vee 1\le D_{i}\le \frac{D\sqrt{w_{i}S_{i}}}{\sum_{i=1}^{k-2}\sqrt{w_{i}S_{i}}}+1,
\]
the density $\overline p=\sum_{i=1}^{k-2}w_{i}\overline p_{i}$ satisfies 
\begin{align*}
\norm{p-\overline p}&\le \sum_{i=1}^{k-2}w_{i}\int_{I_{i}}\ab{p_{i}-\overline p_{i}}d\mu\le  \sum_{i=1}^{k-2}\frac{2w_{i}S_{i}}{D_{i}}\le \frac{2}{D}\pa{\sum_{i=1}^{k-2}\sqrt{w_{i}S_{i}}}^{2}\le \frac{2R}{D}.
\end{align*}
Besides, the density $\overline p$ is $k$-piecewise monotone, supported on $[x_{1},x_{k-2}]$ and piecewise constant on a partition of $\R$ consisting of at most $\sum_{i=1}^{k-2}D_{i}\le D+k-2$ bounded intervals. It therefore belongs to $\overline \cO_{D+k-2,k}$. Finally, let us choose 
\[
D=\left\lceil\pa{\frac{9R^{2}n}{83.2^2}}^{1/3}\right\rceil\le \pa{\frac{9R^{2}n}{83.2^2}}^{1/3}+1. 
\]
Using the sub-additivity property of the square root, we deduce from~\eref{eq-defB} that 
\begin{align*}
\B_{k,n}(p)&\le \frac{3R}{2D}+83.2\sqrt{\frac{D-1+2k}{n}}\\
&\le \frac{3^{1/3}\times 83.2^{2/3}}{2}\pa{\frac{R}{n}}^{1/3}+\frac{83.2}{\sqrt{n}}\sqrt{\pa{\frac{9nR^{2}}{83.2^{2}}}^{1/3}+2k}\\
&\le \cro{\frac{3^{1/3}\times 83.2^{2/3}}{2}+ 3^{1/3}\times83.2^{2/3}}\pa{\frac{R}{n}}^{1/3}+83.2\sqrt{\frac{2k}{n}},
\end{align*}
which is~\eref{eq-thm-3}.  
\end{proof}

\begin{proof}[Proof of Theorem \ref{thm-monogene}]
Let us start with the following lemma where we show that the mapping $\tau(p,\cdot)$ controls the $\L_{1}$-approximation error of a monotone density $p$ by the elements of the class $\overline \cM_{3}^{\infty}(R)$. 
\begin{lem}\label{lem-tronc}
Let $p$ be a density on $\R$, $B$ some positive number and $I$ a subset of $\R$ on which the density $p$ is not a.e.\ equal to 0. The density $p_{|I}^{\wedge B}=(p\wedge B)\1_{I}/\int_{I}(p\wedge B)d\mu$ satisfies, 

\begin{equation}\label{lem-tronc-eq1}
\norm{p-p_{| I}^{\wedge B}}\le 2 \cro{\int_{I}\pa{p-B}_{+}d\mu+\int_{I^{c}}pd\mu}.
\end{equation}
If $p$ is a monotone density on a half-line
%
\begin{equation}\label{lem-tronc-eq2}
\inf_{\overline p\in\overline \cM_{3}^{\infty}(R)}\norm{p-\overline p}\le 2\tau\pa{p,\exp(R)-1}\quad \text{for all $R\ge \log 2$.}
\end{equation}
Besides, if $p$ is a nonincreasing density on $(a,a+l)$ (respectively a nondecreasing density on $(a-l,a)$) with $a\in\R$ and $l\in(0,+\infty]$, we may restrict the infimum to the nonincreasing densities on $(a,a+l)$ (respectively the nondecreasing densities on $(a-l,a)$) that belong to $\overline \cM_{3}^{\infty}(R)$.
\end{lem}

\begin{proof}
Since $p$ is not equal to 0 a.e.\ on $I$, $\int_{I}(p\wedge B)d\mu>0$ and $p_{| I}^{\wedge B}$ is therefore a well-defined density on $I$. By Lemma~\ref{lem_l1_proj}, 
\begin{align*}
\int_{\R}\ab{p-p_{| I}^{\wedge B}}d\mu&\le 2\int_{\R}\ab{p-(p\wedge B)\1_{I}}d\mu= 2 \int_{I^{c}}pd\mu+2 \int_{I}\pa{p-B}_{+}d\mu.
\end{align*}
Changing $p$ into $x\mapsto p(-x)$ if necessary and possibly changing the value of $p$ at the endpoint of the half-line, we may assume without loss of generality that $p$ is a nonincreasing density on a half-line of the form $(a,+\infty)$ with $a\in\R$. Let us now set 
\[
l=\sup\{z>0,\; p(a+z)>0\}\in (0,+\infty]\quad \text{and}\quad t=\exp(R)-1\ge 1.
\]

We first consider the case where $l=+\infty$.  Given $s>0$, let us take $B=p(a+s)>0$ and $I=(a,a+st)$. Since $p$ is nonincreasing on $(a,+\infty)$, $p(x)\ge p(a+s)=B$ for all $x\in (a,a+s)\subset I$ and consequently
\[
\int_{I}(p\wedge B)d\mu\ge \int_{a}^{a+s}(p\wedge B)d\mu=sB.
\]
The density $\overline p_{s}=p_{| I}^{\wedge B}$  belongs to $\overline \cM_{3}^{\infty}$, is supported on an interval of length not larger than $st$ and its variation on $I$ is not larger than 
\[
\frac{B}{\int_{I}(p\wedge B)d\mu}-p(a+st)<\frac{B}{\int_{I}(p\wedge B)d\mu}.
\]
Hence, it follows from~\eref{def-Rp} that
\begin{align*}
\bs{R}_{k,0}(\overline p_{s})< \log\pa{1+\frac{tsB}{\int_{I}(p\wedge B)d\mu}}\le  \log(1+t)=R
\end{align*}
and consequently, $\overline p_{s}\in \overline \cM_{3}^{\infty}(R)$. Applying~\eref{lem-tronc-eq1} and Lemma~\ref{lem-ty}, we obtain that for all $s>0$
\begin{equation}\label{eq-infs2}
\inf_{\overline p\in \overline \cM_{3}^{\infty}(R)}\norm{p-\overline p}\le \norm{p-\overline p_{s}}\le 2\cro{\tau_{x}(p,st)+\tau_{y}\pa{p,p(a+s)}}
\end{equation}
and we derive~\eref{lem-tronc-eq2} from~\eref{eq-tail}. Since for all $s>0$, $\overline p_{s}$ is a density on $(a,+\infty)$, we may restrict the infimum to these densities in $\overline \cM_{3}^{\infty}(R)$ that satisfy this property. 

Let us now turn to the case where $l<+\infty$ and define $s_{0}=l/t\le l$. Given $s\in (0,s_{0})$, we take $B=p(a+s)>0$ and $I=(a,a+st)$. Since $st<l$, $p(a+st)>0$ and by arguing as before, we obtain that 
\[
\inf_{\overline p\in \overline \cM_{3}^{\infty}(R)}\norm{p-\overline p}\le 2\cro{\tau_{x}(p,st)+\tau_{y}\pa{p,p(a+s)}}\quad \text{for all $s\in (0,s_{0})$.}
\]
It follows from the monotonicity of $\tau_{y}(p,\cdot)$ that for all $0<s<s_{0}\le s'$, 
\[
\tau_{y}\pa{p,p(a+s)}\le \tau_{y}\pa{p,p(a+s_{0})}\le \tau_{y}\pa{p,p(a+s')},
\]
and since the mapping $u\mapsto \tau_{x}(p,u)$ is continuous and nonincreasing,  for all $s'\ge s_{0}$ 
\begin{align*}
&\inf_{s\in (0,s_{0})}\cro{\tau_{x}(p,st)+\tau_{y}\pa{p,p(a+s)}}\\
& \le \inf_{s\in (0,s_{0})}\tau_{x}(p,st)+\tau_{y}\pa{p,p(a+s')}=\tau_{x}(p,s_{0}t)+\tau_{y}\pa{p,p(a+s')}\\
&=0+\tau_{y}\pa{p,p(a+s')}=\tau_{x}(p,s't)+\tau_{y}\pa{p,p(a+s')}.
\end{align*}
Consequently, \eref{eq-infs2} remains satisfied for all $s>0$. Since it is actually enough to restrict the infimum to those $s\in (0,s_{0})$ and since for such values of $s$ the density $\overline p_{s}=p_{| I}^{\wedge B}$ vanishes outside $(a,a+s)\subset (a,a+l)$, we  may restrict the infimum in~\eref{lem-tronc-eq2} to those densities in $\overline \cM_{3}^{\infty}(R)$ that vanish outside $(a,a+l)$.
\end{proof}

Let us set $\eta=\tau_{\infty}\pa{p,\exp\pa{R/\ell}-1}$. Since $R/\ell\ge \log 2$, by applying Lemma~\ref{lem-tronc} to the densities $p_{i}$, we may find for all $i\in\{1,\ldots,\ell\}$ a density $\overline p_{i}\in\overline \cM_{3}^{\infty}(R/\ell)$ such that $\|p_{i}-\overline p_{i}\|\le 2\tau\pa{p_{i},\exp(R/\ell)-1}$.  In particular, the density $\overline p=\sum_{i=1}^{\ell}w_{i}\overline p_{i}$ satisfies 
\begin{equation}\label{thm-monogene-eq1}
\norm{p-\overline p}\le \sum_{i=1}^{\ell}w_{i}\norm{p_{i}-\overline p_{i}}\le \max_{i\in\{1,\ldots,\ell\}}\norm{p_{i}-\overline p_{i}}\le 2\eta.
\end{equation}

When $\ell>2$ and $i\in\{2,\ldots,\ell-1\}$, it follows from the definition of $\overline \cM_{3}^{\infty}(R/\ell)$ and Lemma~\ref{lem-tronc} that we may choose $\overline p_{i}$ in such a way that it vanishes outside an interval $I_{i}=(x_{i,0},x_{i,1})\subset (x_{i-1},x_{i})$ with 
\[
\log\cro{1+\pa{x_{i,0}-x_{i,1}}\pa{\sup_{x\in I_{i}}\overline p_{i}(x)-\inf_{x\in I_{i}}\overline p_{i}(x)}}< \frac{R}{\ell}
\]
and $x_{i,0}=x_{i-1}$ when $p_{i}$ is nonincreasing and $x_{i,1}=x_{i}$ when $p_{i}$ is nondecreasing. The mapping $\overline p_{1}$ is a nondecreasing density on an interval of the form $I_{1}=(x_{1,0},x_{1,1})$ with $x_{1,0}<x_{1,1}=x_{1}$ and 
\[
\log\cro{1+\pa{x_{1,1}-x_{1,0}}\pa{\sup_{x\in I_{1}}\overline p_{1}(x)-\inf_{x\in I_{1}}\overline p_{1}(x)}}< \frac{R}{\ell}.
\]
Similarly, $\overline p_{\ell}$ is a nonincreasing density on an interval of the form $I_{\ell}=(x_{\ell,0},x_{\ell,1})$ with $x_{\ell,0}=x_{\ell-1}<x_{\ell,1}$ and 
\[
\log\cro{1+\pa{x_{\ell,1}-x_{\ell,0}}\pa{\sup_{x\in I_{\ell}}\overline p_{\ell}(x)-\inf_{x\in I_{\ell}}\overline p_{\ell}(x)}}< \frac{R}{\ell}.
\]
The density $\overline p=\sum_{i=1}^{\ell}w_{i}\overline p_{i}$ can also be written as 
\[
w_{1}\overline p_{1}\1_{(x_{1,0},x_{1})}+\sum_{i=2}^{\ell-1}w_{i}\cro{\overline p_{i}\1_{I_{i}}+0\1_{(x_{i-1},x_{i})\setminus I_{i}}}+w_{\ell}\overline p_{\ell}\1_{(x_{\ell-1},x_{\ell,1})}
\]
and may therefore be written in the form~\eref{def-mbark} when $k\ge 2\ell$. Moreover, by the Cauchy-Schwarz inequality 
\begin{align*}
\bs{R}_{k,0}(\overline p)< \cro{\sum_{i=1}^{\ell}\sqrt{w_{i}\pa{\frac{R}{\ell}}}+0}^{2}=\frac{R}{\ell}\cro{\sum_{i=1}^{\ell}\sqrt{w_{i}}}^{2}\le R
\end{align*}
and consequently, $\overline p\in\overline \cM_{k}^{\infty}(R)$. We deduce from~\eref{thm-monogene-eq1} that
\begin{align*}
\inf_{q\in \overline \cM_{k}^{\infty}(R)}\norm{p-q}&\le \norm{p-\overline p}\le 2\eta
\end{align*}
which is \eref{thm-monogene-eq0}.
\end{proof}

\subsection{Proofs of Section \ref{sect-cvxcve}}\label{subsect-proof-cvxcve}
\begin{proof}[Proof of Proposition \ref{prop_extr_cvx_cve}]
The proof relies on the following lemma the proof of which is a direct consequence of convexity and is therefore omitted.
\begin{lem}\label{lem-5}
Let $g$ be a convex and continuous function on a nontrivial interval $J$. The set $\{x\in J,\; g(x)<0\}$ is an interval (possibly empty). The set $\{x\in J,\; g(x)>0\}$ has one of the following forms: $\varnothing, J$, $J\cap (-\infty, c_{0})$, $J\cap (c_{1},+\infty)$, $[J\cap (-\infty, c_{0})]\cup [J\cap (c_{1},+\infty)]$ with $c_{0},c_{1}\in J$ and $c_{0}<c_{1}$. In particular $\{x\in J,\; g(x)>0\}$ is the union of at most two intervals. When $g$ is continuous and concave on $J$, the same conclusion holds with $\{x\in J,\; g(x)>0\}$ in place of $\{x\in J,\; g(x)<0\}$ and vice-versa.
\end{lem}

Since $p$ belongs to $\overline \cM_{k}^{1}$ and $q$ belongs to $\overline{\cO}_{D,k}^1$, there exists $A=\{a_{1},\ldots,a_{l}\}$ with $l\in\{1,\ldots,k-1\}$ such that $p$ is convex-concave on each element of $\gI(A)$ and there exists a subset $B\subset \R$ with cardinality not larger than $D+2$ such that $q$ is left-continuous on $\R$ and affine on each element of $\gI(B)$. Since $p$ and $q$ are densities, $p$ is necessarily convex on the two unbounded intervals of $\gI(A)$ and $q$ vanishes on the two unbounded intervals of $\gI(B)$. 

We define $J_{1}=(-\infty,a_{1}]$, $J_{l+1}=(a_{l},+\infty)$ and when $l\ge 2$, $J_{i}=(a_{i-1},a_{i}]$ for all $i\in\{2,\ldots,l\}$. Besides, we set $I_{i}=\mathring{J}_{i}$ for all $i\in\{1,\ldots,l+1\}$. We shall repeatedly use Lemma~\ref{lem-5} with $g=p-q$ throughout the proof. Given $\epsilon \in\{\pm 1\}$, we set 
\begin{equation}\label{decomp-Ceps}
C_{\epsilon}=\ac{x\in\R,\; \epsilon g(x)>0}=\bigcup_{i=1}^{l+1}\ac{x\in J_{i},\; \epsilon g(x)>0}.
\end{equation}
Our aim is to show that $C_{\epsilon}$ is the union of at most $D+2k-1$ intervals. The second part of the proposition is a consequence of Lemma~1 of Baraud and Birg\'e~\cite{MR3565484}. 

If $m_{1}=|B\cap I_{1}|=0$ then $q=0$ on $J_{1}$ (since it is left-continuous), $g=p-0$ is continuous, monotone (nondecreasing) and convex on $I_{1}$, $\{x\in J_{1},\; g(x)<0\}=\varnothing$ and $\{x\in J_{1},\; g(x)>0\}$ is an interval. 

If $m_{1}=|B\cap I_{1}|\ge 1$, we may partition $J_{1}$ into $s=m_{1}+1\ge 2$ consecutive intervals $K_{1},\ldots,K_{s}$ that we may choose to be of the form $(a,b]$, $a<b$, $a\in\R\cup\{-\infty\}$, $b\in\R$. Since $p$ is continuous on $I_{1}$ and $q$ is left-continuous, $g$ is continuous on $K_{1},\ldots,K_{s-1}$ and on $\mathring{K}_{s}$. On $K_{1}$, $\{x\in K_{1},\; g(x)<0\}=\varnothing$ and the set $\Lambda_{1,]}^{+}=\{x\in K_{1},\; g(x)> 0\}$ is an interval which is either empty or contains the right endpoint of $K_{1}$. When $s\ge 2$, we may apply Lemma~\ref{lem-5} to $g$ and the intervals $K_{i}$ with $i\in\{2,\ldots,s-1\}$. We obtain that  $\Lambda_{i}^{-}=\{x\in K_{i},\; g(x)<0\}$ is an interval and $\{x\in K_{i},\; g(x)>0\}$ is of the form $\Lambda_{i,(}^{+}\cup\Lambda_{i,]}^{+}$ where $\Lambda_{i,(}^{+},\Lambda_{i,]}^{+}$ are two (possibly empty) intervals and when they are not, $\inf\Lambda_{i,(}^{+}=\inf K_{i}$ and the right endpoint of $K_{i}$ belongs to $\Lambda_{i,]}^{+}$. The set  $\{x\in K_{s},\; g(x)<0\}$ can also be written as  $\Lambda_{s}^{-}\cup\Lambda_{s,]}^{-}$ where $\Lambda_{s}^{-},\Lambda_{s,]}^{-}$ are two possibly empty intervals and when  $\Lambda_{s,]}^{-}$ is not empty, it reduces to $\{a_{1}\}$. The set $\{x\in K_{s},\; g(x)>0\}$ writes $\Lambda_{s,(}^{+}\cup\Lambda_{s}^{+}$ where $\Lambda_{s,(}^{+},\Lambda_{s}^{+}$ are two possibly empty intervals and when they are not $\inf \Lambda_{s,(}^{+}=\inf K_{s}$ and $\sup \Lambda_{s}^{+}=\sup K_{s}$. We conclude that 
\begin{align*}
\ac{x\in J_{1},\; g(x)<0}=\cro{\bigcup_{i=2}^{s-1}\Lambda_{i}^{-}}\cup\cro{\Lambda_{s}^{-}\cup\Lambda_{s,]}^{-}}
\end{align*}
and 
\begin{align*}
\ac{x\in J_{1},\; g(x)>0}&=\Lambda_{1,]}^{+}\cup\cro{\bigcup_{i=2}^{s-1}\pa{\Lambda_{i,(}^{+}\cup\Lambda_{i,]}^{+}}}\cup\cro{\Lambda_{s,(}^{+}\cup\Lambda_{s}^{+}}
\end{align*}
are both the unions of at most $s=m_{1}+1$ intervals. 

By arguing similarly, we obtain that on the interval $J_{l+1}$: the sets $\{x\in J_{l+1},\; \epsilon g(x)>0\}$ with $\epsilon \in\{\pm 1\}$ are the unions of at most $m_{l+1}+1$ intervals where $m_{l+1}=|B\cap J_{l+1}|$.

When $l\ge 2$, let us now consider an interval of the form $J_{i}=(a_{i-1},a_{i}]$ with $i\in\{2,\ldots,l\}$ and set $m_{i}=|B\cap I_{i}|$. If $m_{i}=0$, $g$ is continuous and convex-concave on $I_{i}$ and by arguing as for $K_{s}$, we obtain that $\{x\in J_{i},\; \epsilon g(x)>0\}$ is the union of at most 2 intervals for every $\epsilon \in\{\pm 1\}$. If $m_{i}\ge 1$, we may partition $I_{i}$ with $m_{i}+1$ intervals of the form $(a,b]$ with $a<b$, $a,b\in\R$. On each of these intervals, $g$ is continuous and convex-concave and by applying Lemma~\ref{lem-5} and arguing as previously, we obtain that $\{x\in J_{i},\; \epsilon g(x)>0\}$ is a union of at most $m_{i}+2$ intervals. 

Using~\eref{decomp-Ceps} and the facts that $\sum_{i=1}^{l+1}m_{i}\le |B|\le D+1$ and  $l\le k-1$, we conclude that the sets $C_{\varepsilon}$ are unions of at most 
\begin{align*}
m_{1}+1+m_{l+1}+1+\sum_{i=2}^{l}\pa{m_{i}+2}\le |B|+2l\le  D+2k-1
\end{align*}
intervals.
\end{proof}

\begin{proof}[Proof of Theorem \ref{thm-approx-conv}] The proof is based on Proposition \ref{prop-approx-convexe} as well as two other approximation results, namely Proposition \ref{prop-approx-con1} and Proposition \ref{prop-approx-con2}  below. We shall additionally use a elementary lemma to simplify some calculations. The proofs the two propositions and the lemma follow that of Theorem \ref{thm-approx-conv}. 

\begin{prop}\label{prop-approx-con1}
Let $p$ be a monotone, continuous and convex-concave sub-density on a bounded interval $[a,b]$ of length $L>0$ with a right derivative $p'_{r}(a)$ at $a$ and a left derivative $p'_{l}(b)$ at $b$. For all $D\ge 1$, there exists a $D$-linear interpolation $\overline p$ of $p$ such that
\begin{equation}\label{eq-approx-monoconv01}
\int_{a}^{b}\ab{p-\overline p}d\mu\le \frac{4}{3}\cro{\pa{1+L\sqrt{\ab{p'_{l}(b)-p'_{r}(a)}}}^{1/D}-1}^{2}.
\end{equation}
\end{prop}

\begin{prop}\label{prop-approx-con2}
Let $p$ be a strictly monotone, continuous, convex-concave sub-density  with variation $V=|p(a)-p(b)|$ on a nontrivial bounded interval $[a,b]$. Assume furthermore that $p'_{r}(a)$ and $p'_{l}(b)$ are nonzero. Then, for all $D\ge 1$ there exists a $D$-linear interpolation $\overline p$ of $p$ on $[a,b]$ such that
\begin{equation}\label{eq-cor-monoconv01}
\int_{a}^{b}\ab{p-\overline p}d\mu\le \frac{4}{3}\cro{\pa{1+V\sqrt{\ab{\frac{1}{p'_{r}(a)}-\frac{1}{p'_{l}(b)}}}}^{1/D}-1}^{2}.
\end{equation}
\end{prop}

\begin{lem} \label{lem_fun_cve}
If $F$ is a nondecreasing, differentiable, concave function on $\R_{+}$, the mapping 
\[
\phi: u \mapsto \left[ F\pa{\sqrt{u}} - F(0) \right]^2
\]
is concave on $\mathbb{R}_+$. In particular for all $D\ge 1$, 
\[
F_{D}:u\mapsto \cro{\pa{1+\sqrt{u}}^{1/D}-1}^{2}
\]
is concave on $\R_{+}$.
\end{lem}

Let's turn to the proof of Theorem \ref{thm-approx-conv}. We first introduce some mappings of interest we will use in the proofs of Theorem~\ref{thm-approx-conv} and Proposition~\ref{prop-approx-con1}.
Let $\cV$ be the linear space of continuous functions $f$ on $[a,b]$ that admit a right derivative $f'_{r}(a)$ at $a$ and a left derivative $f'_{l}(b)$ at $b$. Given $m\in\R$, we define $\cT_{1}$ and $\cT_{2}$ as the mappings defined on $\cV$ by 
\begin{equation}\label{def-ctj}
\cT_{1}:f\mapsto [x\mapsto f(a+b-x)],\quad  \cT_{2}:f\mapsto [x\mapsto m-f(x)].
\end{equation}
Although $\cT_{2}$ depends on the choice of $m$, we drop this dependency in the notation for the sake of convenience. The mappings $\cT_{j}$ are one-to-one from $\cV$ onto itself, isometric with respect to the $\L_{1}$-norm on $\cV$ and they satisfy for all $f\in\cV$
\begin{align}
\ab{f(a)-f(b)}&=\ab{\cT_{j}(f)(a)-\cT_{j}(f)(b)}\label{prop1-ctj}\\
\ab{f'_{r}(a)-f'_{l}(b)}&=\ab{(\cT_{j}(f))'_{r}(a)-(\cT_{j}(f))'_{l}(b)}\label{prop2-ctj}
\end{align}
and $\cT_{j}^{-1}=\cT_{j}$ for all $j\in\{1,2\}$. \\
Let us now turn to the proof of Theorem~\ref{thm-approx-conv} and assume first that $p$ is nondecreasing, continuous and convex on $[a,b]$ so that $p_{r}'(a)\ge 0$. If $p$ is constant, the result is clear by taking $\overline p=p$. We may therefore assume that $p(a)<p(b)$ and choose a point $c\in (a,b)$ such that $p(c)>p(a)$. In particular, $w_{1}=\int_{a}^{c}pd\mu>0$ and 
\[
0<\frac{p(c)-p(a)}{c-a}\le p'_{l}(c)\le \frac{p(b)-p(c)}{b-c}<+\infty.
\]
The restriction $p_{1}$ of $p$ on the interval $[a,c]$ is nondecreasing, continuous and  convex and so is the density $p_{1}/w_{1}$. We may therefore apply Proposition~\ref{prop-approx-con1} to $p_{1}/w_{1}$ and find a $D$-linear interpolation $\overline p_{1}$ of $p$ on $[a,c]$ that satisfies
\begin{align}
\int_{a}^{c}\ab{p-\overline p_{1}}d\mu&\le \frac{4w_{1}}{3}\cro{\pa{1+(c-a)\sqrt{\frac{p'_{l}(c)-p'_{r}(a)}{w_{1}}}}^{1/D}-1}^{2}.\label{eq-approx-conv001}
\end{align}
The restriction $p_{2}$ of $p$ to $[c,b]$ is increasing, continuous and convex with nonzero right and left derivatives at $c$ and $b$ respectively. We may therefore apply Proposition~\ref{prop-approx-con2} to the density $p_{2}/w_{2}$ with $w_{2}=\int_{c}^{b}pd\mu=1-w_{1}>0$ and find a $D$-linear interpolation $\overline p_{2}$ of $p$ on $[c,b]$ that satisfies 
\begin{align}
\int_{c}^{b}\ab{p-\overline p_{2}}d\mu\le \frac{4w_{2}}{3}\cro{\pa{1+\frac{p(b)-p(c)}{\sqrt{w_{2}}}\sqrt{\frac{1}{p'_{r}(c)}-\frac{1}{p'_{l}(b)}}}^{1/D}-1}^{2}.\label{eq-approx-conv002}
\end{align}

We may choose $c\in (a,b)$ such that 
\[
p'_{l}(c)\le \Delta\quad \text{and}\quad p'_{r}(c)\ge \Delta\quad \text{with}\quad \Delta=\frac{p(b)-p(a)}{b-a}.
\]
Then 
\begin{align}
A&=(c-a)^{2}\pa{p'_{l}(c)-p'_{r}(a)}+(p(b)-p(c))^{2}\pa{\frac{1}{p'_{r}(c)}-\frac{1}{p'_{l}(b)}}\nonumber \\
&\le (b-a)^{2}\pa{\Delta-p'_{r}(a)}+(p(b)-p(a))^{2}\pa{\frac{1}{\Delta}-\frac{1}{p'_{l}(b)}}\nonumber\\
&=2 (b-a)(p(b)-p(a))\cro{1-\frac{1}{2}\pa{\frac{p'_{r}(a)}{\Delta}+\frac{\Delta}{p'_{l}(b)}}}=2LV\Gamma.\label{def-A-con}
\end{align}

The function $\overline p=\overline p_{1}\1_{[a,c)}+\overline p_{2}\1_{[c,b]}$  is a $2D$-linear interpolation of $p$ on $[a,b]$. Since by Lemma~\ref{lem_fun_cve} the function $F_{D}$ is concave and increasing, we deduce from~\eref{def-A-con}, ~\eref{eq-approx-conv001} and \eref{eq-approx-conv002} that
\begin{align*}
\frac{3}{4}\norm{p-\overline p}&\le w_{1}\int_{a}^{c}\ab{p-\overline p_{1}}d\mu+w_{2}\int_{c}^{b}\ab{p-\overline p_{2}}d\mu\\
&\le w_{1}F_{D}\pa{\frac{(c-a)^{2}\pa{p'_{l}(c)-p_{r}(a)}}{w_{1}}}\\
&\quad +w_{2}F_{D}\pa{\frac{(p(b)-p(c))^{2}}{w_{2}}\pa{\frac{1}{p'_{r}(c)}-\frac{1}{p'_{l}(b)}}}\\
&\le F_{D}(A)\le F_{D}(2LV\Gamma), 
\end{align*}
which is~\eref{eq-thm-approx-conv}. 

The density $q=\overline p/\int_{a}^{b}\overline pd\mu$ is continuous convex-concave on $[a,b]$ and piecewise linear on a partition of $[a,b]$ into $2D$ intervals and it follows from Lemma~\ref{lem_l1_proj} that 
\begin{equation}\label{eq-approx-conv003}
\int_{a}^{b}\ab{p-q}d\mu\le 2\ac{1\wedge \cro{\frac{4}{3}\pa{\pa{1+\sqrt{2LV\Gamma}}^{1/D}-1}^{2}}}.
\end{equation}
The mapping $z\mapsto (4/3)(e^{z}-1)^{2}$ is not larger than 1 if and only if  $z\le z_{0}=\log(1+\sqrt{3/4})$ and for all $z\in [0,z_{0}]$, $e^{z}-1\le (e^{z_{0}}-1)z/z_{0}$. Consequently, for all $z\ge 0$
\begin{align*}
2\ac{1\wedge \cro{\frac{4}{3}\pa{e^{z}-1}^{2}}}&= \frac{8}{3}\pa{e^{z_{0}\wedge z}-1}^{2}\le \frac{8}{3}\pa{\frac{e^{z_{0}}-1}{z_{0}}z}^{2}\le 5.14z^{2}.
\end{align*}
Applying this inequality with $z=D^{-1}\log\pa{1+\sqrt{2LV\Gamma}}$, we deduce~\eref{thm-approx-conv-01} from ~\eref{eq-approx-conv003}. 

Theorem~\ref{thm-approx-conv} is therefore proven for a nondecreasing, continuous convex density $p$. In order to prove the result in the other cases, we use the transformations $\cT_{1}$ and $\cT_{2}$ introduced above and defined by ~\eqref{def-ctj}. These transformations are isometric with respect to the $\L_{1}$-norm and they preserve the variation of a monotone function. Note that they also preserve the linear index, i.e.\ 
for all continuous convex-concave function $f$ on $[a,b]$ and $j\in\{1,2\}$, $\Gamma(f)=\Gamma\pa{\cT_{j}(f)}$.
Applying $\cT_{1}$ to nonincreasing continuous convex densities we establish \eref{thm-approx-conv-01}  and we extend it to all monotone concave continuous densities  by applying $\cT_{2}$.
\end{proof}

\begin{proof}[Proof of Proposition \ref{prop-approx-con1}]
Let us first consider a function $g$ that is monotone, continuous and convex on $[a,b]$ and that satisfies $g(a)=g_{r}'(a)=0$ and $0<\int_{a}^{b}gd\mu\le 1$. Then $g$ is nonnegative and nondecreasing on $[a,b]$. Since $\int_{a}^{b}gd\mu>0$, $g$ is not identically equal to 0 on $[a,b]$ and consequently, 
\[
g_{l}'(b)\ge \frac{g(b)-g(a)}{b-a}>0.
\]

Let $q> 1$ and 
\[
x_{0}=a\quad \text{and} \quad x_{i}=x_{i-1}+L\frac{q^{-i}}{\sum_{j=1}^{D}q^{-j}}\quad \text{for all $i\in\{1,\ldots,D\}$},
\]
so that $x_{D}=b$ and $\Delta_{i+1}=x_{i+1}-x_{i}=q^{-1}(x_{i}-x_{i-1})=q^{-1}\Delta_{i}$ for all $i\in\{1,\ldots,D-1\}$ and 
\[
\Delta_{D}=L\frac{q^{-D}(1-q^{-1})}{q^{-1}-q^{-(D+1)}}=L\frac{q-1}{q^{D}-1}.
\]

Let $g'$ be any nondecreasing function on $[a,b]$ satisfying $0=g_{r}'(a)=g'(a)$, $g_{l}'(b)=g'(b)$ and $g_{l}'(x)\le g'(x)\le g_{r}'(x)$ for all $x\in (a,b)$. Since $g$ is convex, we may write 
\begin{equation}\label{eq-approx-convexe2}
g(x)\ge g(x_{i})+g'(x_{i})(x-x_{i})\quad \text{for all $i\in\{1,\ldots,D\}$ and $x\in [a,b]$}.
\end{equation}
In particular,  for all $i\in\{1,\ldots,D-1\}$
\begin{align*}
\int_{x_{i}}^{x_{i}+1}\pa{g(x_{i})+g'(x_{i})(x-x_{i})}d\mu(x)&=\cro{g(x_{i})+\frac{g'(x_{i})\Delta_{i+1}}{2}}\Delta_{i+1}\le \int_{x_{i}}^{x_{i}+1}gd\mu, 
\end{align*}
hence,  
\begin{equation}\label{eq-approx-convexe3}
g'(x_{i})\Delta_{i+1}^{2}\le 2\int_{x_{i}}^{x_{i}+1}(g-g(x_{i}))d\mu.
\end{equation}
Moreover, applying~\eref{eq-approx-convexe2} with $x_{i}=x_{D}=b$ and using the facts that $g$ is nonnegative and nonincreasing, we get
\begin{align*}
1\ge \int_{a}^{b}gd\mu&\ge \int_{a}^{b}\pa{g(b)+g'(b)(x-b)}_{+}d\mu(x)=\int_{b-g(b)/g'(b)}^{b}\pa{g(b)+g'(b)(x-b)}d\mu(x)=\frac{g^{2}(b)}{2g'(b)},
\end{align*}
consequently
\begin{equation}\label{eq-approx-convexe4}
g(b)\le \sqrt{2g'(b)}.
\end{equation}

Let $g_{i}$ be the restriction of $g$ on the interval $[x_{i-1},x_{i}]$  and $\overline g$ the function on $[a,b]$ that coincides on $[x_{i-1},x_{i}]$ with $\ell_{g_{i}}$ defined by \eref{def-lf} with $f=g_{i}$  for all $i\in\{1,\ldots,D\}$. The function $\overline g$ is a $D$-linear interpolation of $g$ on $[a,b]$ that satisfies 
\begin{align*}
\int_{a}^{b}\ab{g-\overline g}d\mu&=\sum_{i=1}^{D}\int_{x_{i-1}}^{x_{i}}\pa{\overline g-g}d\mu\le \frac{1}{8}\sum_{i=1}^{D} \Delta_{i}^{2}\pa{g_{l}'(x_{i})-g_{r}'(x_{i-1})}\le \frac{1}{8}\sum_{i=1}^{D} \Delta_{i}^{2}\pa{g'(x_{i})-g'(x_{i-1})}\\
&=\frac{1}{8}\cro{-\Delta_{1}^{2}g'(a)+\Delta_{D}^{2}g'(b)+\sum_{i=1}^{D-1}g'(x_{i})\pa{\Delta_{i}^{2}-\Delta_{i+1}^{2}}}.
\end{align*}
Since $g'(a)=0$ and $\Delta_{i}=q\Delta_{i+1}$ for all $i\in\{1,\ldots,D-1\}$ we deduce that 
\begin{equation}\label{eq-approx-convexe5}
\int_{a}^{b}\ab{g-\overline g}d\mu\le \frac{1}{8}\cro{\Delta_{D}^{2}g'(b)+(q^{2}-1)\sum_{i=1}^{D-1}g'(x_{i})\Delta_{i+1}^{2}}.
\end{equation}
Using \eref{eq-approx-convexe3}, \eref{eq-approx-convexe4} and the fact that $g$ is nondecreasing, 
\begin{align*}
\frac{1}{2}\sum_{i=1}^{D-1}g'(x_{i})\Delta_{i+1}^{2}&\le \sum_{i=1}^{D-1}\int_{x_{i}}^{x_{i+1}}(g-g(x_{i}))d\mu\\
&=\int_{x_{1}}^{x_{D}}gd\mu-\sum_{i=1}^{D-1}\Delta_{i+1}g(x_{i})=\int_{x_{1}}^{x_{D}}gd\mu-q^{-1}\sum_{i=1}^{D-1}\Delta_{i}g(x_{i})\\
&\le \int_{x_{1}}^{x_{D}}gd\mu-q^{-1}\sum_{i=1}^{D-1}\int_{x_{i-1}}^{x_{i}}gd\mu= \int_{x_{1}}^{x_{D}}gd\mu-q^{-1}\int_{x_{0}}^{x_{D-1}}gd\mu\\
&=(1-q^{-1})\int_{x_{1}}^{x_{D}}gd\mu+q^{-1}\int_{x_{D-1}}^{x_{D}}gd\mu-q^{-1}\int_{x_{0}}^{x_{1}}gd\mu\\
&\le (1-q^{-1})\times 1+q^{-1}\Delta_{D}g(b)\le  1-q^{-1}+q^{-1}\sqrt{2\Delta_{D}^{2}g'(b)}.
\end{align*}
It follows from \eref{eq-approx-convexe5} that 
\begin{align*}
\int_{a}^{b}\ab{g-\overline g}d\mu&\le \frac{1}{8}\cro{\Delta_{D}^{2}g'(b)+2(q^{2}-1)\pa{ 1-q^{-1}+q^{-1}\sqrt{2\Delta_{D}^{2}g'(b)}}}\\
&\le \frac{(q-1)^{2}}{8}\cro{\frac{L^{2}g'(b)}{(q^{D}-1)^{2}}+2\pa{1+\frac{1}{q}}\pa{1+\frac{L\sqrt{2g'(b)}}{q^{D}-1}}}.
\end{align*}
Finally, choosing $q$ such that
\[
q^{D}-1=L\sqrt{g'(b)}\quad \text{i.e.}\quad q=\pa{1+L\sqrt{g'(b)}}^{1/D}> 1,
\]
we get that
\begin{align}
\int_{a}^{b}\ab{g-\overline g}d\mu&\le \frac{1+4(1+\sqrt{2})}{8}\cro{\pa{1+L\sqrt{g'(b)}}^{1/D}-1}^{2}\nonumber\\
&\le \frac{4}{3}\cro{\pa{1+L\sqrt{g'(b)}}^{1/D}-1}^{2}.\label{eq-approx-convexe6}
\end{align}

Let us now prove Proposition~\ref{prop-approx-con1} in the case where $p$ is convex and nondecreasing. Then $p'_{r}(a)\ge 0$ and we may set $\ell:x\mapsto p(a)+p_{r}'(a)(x-a)$ and $g: x\mapsto p(x)-\ell(x)$. The function $g$ is nonnegative, nondecreasing, continuous and convex on $[a,b]$ and it satisfies $g(a)=g'_{r}(a)=0$. If $\int_{a}^{b}gd\mu=0$, then $p=\ell$ and we may choose $\overline p=p$. Otherwise $0<\int_{a}^{b}gd\mu\le 1$, since $\ell$ is nonnegative on $[a,b]$, and we may apply our previous result to $g$. This leads to a $D$-linear interpolation $\overline g$ of $g$ from which we may define $\overline p=\overline g-\ell(x)$ which is a $D$-linear interpolation of $p$ on $[a,b]$. Inequality~\eref{eq-approx-monoconv01} follows from~\eref{eq-approx-convexe6} and the facts that $\ab{p-\overline p}=\ab{g-\overline g}$ and $g'_{l}(b)=p'_{l}(b)-p'_{r}(a)$. \\

In order to prove Proposition~\ref{prop-approx-con1} in the other cases, we use transformations defined by \eqref{def-ctj}. We can note that if $\ell$ is a $D$-linear interpolation of $f$, $\cT_{j}(\ell)$ is $D$-linear interpolation of $\cT_{j}(f)$ (based on the same subdivision).

If $p$ is convex and nonincreasing, we apply the transformation $\cT_{1}$. Then $g=\cT_{1}(p)$ is a convex, nondecreasing sub-density on $[a,b]$ and our previous result applies. We may find a $D$-linear interpolation $\overline g$ on $[a,b]$ such that 
\begin{equation}\label{eq-approx-convexe7}
\int_{a}^{b}\ab{g-\overline g}d\mu\le \frac{4}{3}\cro{\pa{1+L\sqrt{\ab{g_{l}'(b)-g_{r}'(a)}}}^{1/D}-1}^{2},
\end{equation}
and the function $\overline p=\cT_{1}(\overline g)$ is a $D$-linear interpolation of $p=\cT_{1}(g)$ that satisfies~\eref{eq-approx-monoconv01}. The result is therefore proven for all convex continuous monotone sub-density on $[a,b]$. 

If $p$ is a concave continuous sub-density, we apply the transformation $\cT_{2}$ with $m=2S/(b-a)$ and $S=\int_{a}^{b}pd\mu\le 1$. Since $p$ is concave and monotone, $(p(a)\vee p(b))(b-a)/2\le S$, hence $p(x)\le m$ for all $x\in [a,b]$, and $g=\cT_{2}(p)=m-p$ is a monotone convex sub-density since 
\[
\int_{a}^{b}\pa{m-p}d\mu=\frac{2S}{b-a}\times (b-a)-S=S\le 1.
\]
Applying the previous result, we may find a $D$-interpolation $\overline g$ of $g$ that satisfies~\eref{eq-approx-convexe7} and $\overline p=\cT_{2}(\overline g)$ satisfies~\eref{eq-approx-monoconv01}. This completes the proof of Proposition~\ref{prop-approx-con1}.
\end{proof}

\begin{proof}[Proof of Proposition \ref{prop-approx-con2}]
We start with the following lemma. 
\begin{lem} \label{lem_sym_cvx_dec}
If $f$ and $g$ are two continuous increasing or decreasing functions from on interval $I$ onto an interval $J$,  
\begin{equation} \label{eq_sym_dist_L1}
  \int_I |f - g| d\mu = \int_J  |f^{-1} - g^{-1}| d\mu.  
\end{equation}
Moreover, if $\inf I=0$ and $f$ is nonnegative, continuous and decreasing
\[
\int_{J}f^{-1}d\mu\le \int_{I}fd\mu.
\]
\end{lem}
\begin{proof}
Let us first assume that $f,g$ are both increasing. For all $(t,y)\in I\times J$
\[
\1_{g(t)\le y\le f(t)}+\1_{f(t)\le y\le g(t)}=\1_{f^{-1}(y)\le t\le g^{-1}(y)}+\1_{g^{-1}(y)\le t\le f^{-1}(y)}.
\]
By integrating this equality on $I\times J$ and using Fubini's theorem, we obtain that 
\begin{align*}
\int_{I}\ab{f(t)-g(t)}d\mu(t)&=\int_{I}\cro{\int_{J}\cro{\1_{g(t)\le y\le f(t)}+\1_{f(t)\le y\le g(t)}}d\mu(y)}d\mu(t)\\
&=\int_{J}\cro{\int_{I}\cro{\1_{f^{-1}(y)\le t\le g^{-1}(y)}+\1_{g^{-1}(y)\le t\le f^{-1}(y)}}d\mu(t)}d\mu(y)\\
&=\int_{J}\ab{f^{-1}(y)-g^{-1}(y)}d\mu(y).
\end{align*}
When $f$ and $g$ are both decreasing, we may apply the above equality to $\overline f=f(-\cdot)$ and $\overline g=g(-\cdot)$ from $K=\{-x,\; x\in I\}$ to $J$ then $\overline f^{-1}=-f^{-1},\overline g^{-1}=-g^{-1}$ and the result follows from the facts 
\[
\int_{I}\ab{f-g}d\mu=\int_{K}\ab{\overline f-\overline g}d\mu \text{ and } \int_{J}\ab{\overline f^{-1}-\overline g^{-1}}d\mu=\int_{J}\ab{f^{-1}-g^{-1}}d\mu.
\]

Let us now turn to the  proof of the second part of the lemma. Since $\inf I=0$ and $f$ is decreasing and continuous from $I$ onto $J$, for all $t\in J$
\[
\mu\pa{\ac{x\in I,\; f(x)\ge t}}=\mu\pa{\ac{x\in I,\; x\le f^{-1}(t)}}=f^{-1}(t).
\]
Besides, $f$ being nonnegative, the interval $J\subset \R_{+}$ and it follows from Fubini's theorem that
\begin{align*}
\int_{J}f^{-1}(t)d\mu(t)&=\int_{J}\mu\pa{\ac{x\in I,\; f(x)\ge t}}d\mu(t)\\
&=\int_{J}\cro{\int_{I}\1_{t\le f(x)}d\mu(x)}d\mu(t)=\int_{I}\cro{\int_{J}\1_{t\le f(x)}d\mu(t)}d\mu(x)\\
&\le \int_{I}\cro{\int_{0}^{+\infty}\1_{t\le f(x)}d\mu(t)}d\mu(x)=\int_{I}
f(x)d\mu(x).
\end{align*}
\end{proof}

Let us now turn to the proof of~\eref{eq-cor-monoconv01}. We first claim that it is sufficient to establish ~\eref{eq-cor-monoconv01} when $p$ is decreasing. If $p$ were increasing, we could apply the result to $q=\cT_{1}(p)$ defined by~\eref{def-ctj}, which is then decreasing, and find a $D$-linear interpolation of $q$ on $[a,b]$ that satisfies 
\[
\int_{a}^{b}\ab{q-\overline q}d\mu\le \frac{4}{3}\cro{\pa{1+V\sqrt{\ab{\frac{1}{q'_{r}(a)}-\frac{1}{q'_{l}(b)}}}}^{1/D}-1}^{2}.
\]
We then conclude by using the facts that $\cT_{1}$ is an $\L_{1}$-isometry, the function $\overline p=\cT_{1}(\overline q)$ which is a $D$-linear interpolation of $p$, and $q'_{r}(a)=-p_{l}'(b)$ and  $q'_{l}(b)=-p_{r}'(a)$. 

We may therefore assume that $p$ is decreasing and change $p$ into $p(\cdot-a)$, which amounts to translating the sub-density $p$, we may also assume without loss of generality that $a=0$. By Lemma~\ref{lem_sym_cvx_dec}, $s=p^{-1}$ is then a sub-density on $[p(b),p(0)]$ which is furthermore decreasing, continuous and convex-concave. 
Since the right and left derivatives of $p$ at $0$ and $b$ respectively are not zero, $s$ admits a right derivative at $p(b)$ and a left derivative at $p(0)$ given by $1/p'_{l}(b)$ and $1/p'_{r}(0)$ (with our convention $1/(+\infty)=0$. We may therefore apply Proposition~\ref{prop-approx-con1} and find a $D$-linear interpolation $\overline s$ of $s$ on $[p(b),p(0)]$ that satisfies 
\[
\int_{p(b)}^{p(0)}\ab{s-\overline s}d\mu\le \frac{4}{3}\cro{\pa{1+V\sqrt{\ab{\frac{1}{p'_{r}(0)}-\frac{1}{p'_{l}(b)}}}}^{1/D}-1}^{2}.
\]
Since $s$ is continuous and decreasing from $[p(b),p(0)]$ onto $[0,b]$, so is $\overline s$, and we may set $\overline p=\overline s^{-1}:[0,b]\to [p(b),p(0)]$. The function $\overline p$ is a $D$-linear interpolation of $p$ on $[0,b]$ and we conclude by using the equality $\int_{p(b)}^{p(0)}\ab{s-\overline s}d\mu=\int_{0}^{b}\ab{p-\overline p}d\mu$ which is a consequence of Lemma~\ref{lem_sym_cvx_dec}.
\end{proof}

\begin{proof}[Proof of Lemma~\ref{lem_fun_cve}]
For all $u>0$, 
\[
 \phi'(u) =  \frac{F\pa{\sqrt{u}} - F(0)}{\sqrt{u}}F'(\sqrt{u})
\]
is the product of $u\mapsto [F\pa{\sqrt{u}} - F(0)]/\sqrt{u}$ and $u\mapsto F'(\sqrt{u})$ which are both nonnegative and nonincreasing since $F$ is nondecreasing and concave. The function $\phi'$ is therefore nonincreasing and $\phi$ concave. 
\end{proof}

\subsection{Proofs of Section \ref{section-s-shape}}\label{subsect-proof-sshape}

\begin{proof}[Proof of Proposition~\ref{prop-cL}]
Let $p\in \overline \cM(\cL)$, $J=\{p>0\}$ and $\overline p\in\overline \cO_{D}(\cL)$. There exists a partition $\cI$ of $\R$ into at most $D+1$ intervals such that the restriction of $\cL\overline p$ to an element $I\in\cI$ is either affine or equal to $-\infty$. We denote by $\phi$ and $\overline \phi$ the functions $\cL p$ and $\cL \overline p$ respectively which map $\R$ into $\R\cup\{-\infty\}$. Since $\cL$ is increasing 
\[
\{p>\overline p\}=\{\phi>\overline \phi\}\quad \text{and}\quad \{p<\overline p\}=\{\phi<\overline \phi\}.
\]
To prove that $\overline p$ is extremal in $\overline \cM(\cL)$ with degree not larger that $2(D+2)$, it suffices to prove that $\{\phi>\overline \phi\}$ and $\{\phi<\overline \phi\}$ are both unions of at most $D+2$ intervals and then to use Lemma~1 of Baraud and Birg\'e~\cite{MR3565484}. This is what we shall do now. 

Let us start with the set $\{\phi>\overline \phi\}$.  Either $\overline \phi$ is finite on $I\in \cI$ and $I\cap J\cap \{\phi>\overline \phi\}$ is an interval since $\phi-\overline \phi$ is a concave function on the interval $I\cap J$ when it is not empty. Otherwise, $\overline \phi$ takes the value $-\infty$ on $I$ and $I\cap J\cap \{\phi>\overline \phi\}=I\cap J$ remains an interval.  The set $I\cap J^{c}\cap \{\phi>\overline \phi\}$ is empty since $\phi$ takes the value $-\infty$ on $J^{c}$. Hence 
\[
I\cap \ac{\phi>\overline \phi}=\pa{I\cap J\cap \ac{\phi>\overline \phi}}\cup \pa{I\cap J^{c}\cap \ac{\phi>\overline \phi}}
\]
is an interval and 
\[
\ac{p>\overline p}=\ac{\phi>\overline \phi}=\bigcup_{I\in \cI}\pa{I\cap \ac{\phi>\overline \phi}}
\]
the union of at most $D+1$ intervals. 

Let us now turn to the set $\{\phi<\overline \phi\}$ and define $K=(a,b)$ as the open interval $\{\phi>-\infty\}\cap\{\overline \phi>-\infty\}$ with $a\in\R\cup\{-\infty\}$, $b\in\R\cup\{+\infty\}$, $a\le b$. If $K=\varnothing$, i.e. $a=b$, $\{\phi<\overline \phi\}=\{\overline \phi>-\infty\}=\{\overline p>0\}$
is an interval. Otherwise $K$ is not empty, $\psi=\phi-\overline \phi$ is a continuous function on $K$ which is either concave on the whole interval $K$ or only piecewise concave on each element of $I\cap K$ with $I\in\gI(A\cap K)$ when $A\cap K\ne \varnothing$. In any case, 
\[
\ac{x\in K,\; \phi(x)<\overline \phi(x)}=\ac{x\in K,\; \phi(x)-\overline \phi(x)<0}
\]
is a union of at most $|A\cap K|$ intervals plus, possibly, an additional one with endpoint $a$ and another one with endpoint $b$. If $b< +\infty$, either $\overline \phi(b)=-\infty$ or $\phi(b)=-\infty$ and the set $\{x\ge b,\; \phi(x)<\overline \phi(x)\}$ is either empty or is an interval that contains $b$. We may argue similarly to establish that $\{x\le a,\; \phi(x)<\overline \phi(x)\}$ is either empty or an interval that contains $a$. This means that the set 
\[
\ac{\phi<\overline \phi}=\pa{K\cap \ac{\phi<\overline \phi}}\cup \pa{K^{c}\cap  \ac{\phi<\overline \phi}}
\]
is a union of at most $|A\cap K|+2\le D+2$ intervals.\\
\end{proof}

\begin{proof}[Proof of Proposition~\ref{prop-astGuillaume}]
Since $\cL$ is increasing and convex on $(0,+\infty)$, it grows to $+\infty$ at $+\infty$ and it can be extended continuously at $0$. Using the same notation $\cL$ for this extension, we obtain that $\cL: [0,+\infty) \to [u,+\infty)$ for some $u\in\R$. The inverse $\cL^{-1}: [u,+\infty) \to [0,+\infty)$ is continuously increasing and concave. Since $\cL p$ is by definition  concave on $\{p > 0\}=(a,b)$, the density $p = \cL^{-1}(\cL p)$ is also concave on this interval and $(a,b)$ is necessarily bounded. Furthermore, $p$ can be continuously extended on the compact interval $[a,b]$. This extension $\widetilde p$ is also concave on $[a,b]$ and thus admits a maximum at some point $c\in [a,b]$. The function $\widetilde p$ is nondecreasing on $[a,c]$ and nonincreasing on $[c,b]$. 
Applying $(iii)$ in Proposition~\ref{prop-guer1} we know that there exist a $D$-linear interpolation $\overline p_{1}$ of $\widetilde p$ on $[a,c]$ and a $D$-linear interpolation $\overline p_{2}$ of $\widetilde p$ on $[c,b]$ such that 
\begin{align*}
\int_{a}^{c}\ab{\widetilde p-\overline p_{1}}d\mu\le \frac{(c-a)(\widetilde p(c)-\widetilde p(a))}{2 D^{2}}\quad \text{and}\quad \int_{c}^{b}\ab{\widetilde p-\overline p_{2}}d\mu\le \frac{(b-c)(\widetilde p(c)-\widetilde p(b))}{2 D^{2}}.
\end{align*}
Since $\overline p_{1}(c)=\overline p_{2}(c)=\widetilde p(c)$, the function $\overline p=\overline p_{1}\1_{[a,c]}+\overline p_{2}\1_{(c,b]}$ is a $(2D)$-linear interpolation of $\widetilde p$ on $[a,b]$ which satisfies 
\begin{align*}
\int_{a}^{b}\ab{p-\overline p}d\mu=\int_{a}^{b}\ab{\widetilde p-\overline p}d\mu\le \frac{1}{D^{2}}\cro{\frac{(c-a)(\widetilde p(c)-\widetilde p(a))}{2}+\frac{(b-c)(\widetilde p(c)-\widetilde p(b))}{2}}.
\end{align*}
The quantities $[(c-a)(\widetilde p(c)-\widetilde p(a))]/2$ and $[(b-c)(\widetilde p(c)-\widetilde p(b))]/2$ are the areas of two disjoint triangles that lie under the graph of the density $\widetilde p$ on $[a,b]$ and their sum is therefore not larger than 1. This leads to the second inequality in \eref{eq-astGuillaume1}. Note that since $p$ is positive on $(a,b)$, so is its $(2D)$-linear interpolation $\overline p$: $\overline p(x)>0$ for all $x\in (a,b)$. 

Since $\cL$ extends continuously at $0$, the function $\cL\widetilde{p}$ is well-defined on $[a,b]$ and corresponds to the continuous extension of $\cL p$ on this interval. In particular $\cL \widetilde{p}$ is concave on the compact interval $[a,b]$. 
Let $a=x_{0}<\ldots<x_{2D}=b$ be the subdivision associated with the $2D$-linear interpolation $\overline p$ of $\widetilde p$ on $[a,b]$. The function $\overline q$ defined by 
\begin{align*}
\overline q(x)=\sum_{i=1}^{2D}\pa{\frac{x-x_{i-1}}{x_{i}-x_{i-1}}\cL\widetilde p(x_{i})+\frac{x_{i}-x}{x_{i}-x_{i-1}}\cL\widetilde p(x_{i-1})}\1_{[x_{i-1},x_{i})}(x)\quad \text{for all $x\in [a,b)$}
\end{align*}
and $\overline q(b)=\cL\widetilde p(b)$, is a $2D$-linear interpolation of $\cL\widetilde p$ based on the same subdivision as $\overline p$. In particular, since $\cL\widetilde p$ is concave, $\overline q$ is also concave and it satisfies  $\overline q\le \cL\widetilde p$ on $[a,b]$. Since $\cL^{-1}$ is increasing and concave, the function $\check{p}=\cL^{-1}\circ \overline q$ satisfies $0\le \check{p}\le  \cL^{-1}(\cL\widetilde p)=\widetilde p$ on $[a,b]$ and for all $x\in [x_{i-1},x_{i})$ with $i\in\{1,\ldots,2D\}$
\begin{align*}
\check{p}(x)=\cL^{-1}(\overline q(x))&= \cL^{-1}\pa{\frac{x-x_{i-1}}{x_{i}-x_{i-1}}\cL\widetilde p(x_{i})+\frac{x_{i}-x}{x_{i}-x_{i-1}}\cL\widetilde p(x_{i-1})}\\
&\ge \frac{x-x_{i-1}}{x_{i}-x_{i-1}}\cL^{-1}(\cL\widetilde p(x_{i}))+\frac{x_{i}-x}{x_{i}-x_{i-1}}\cL^{-1}(\cL\widetilde p(x_{i-1}))\\
&=\frac{x-x_{i-1}}{x_{i}-x_{i-1}}\widetilde p(x_{i})+\frac{x_{i}-x}{x_{i}-x_{i-1}}\widetilde p(x_{i-1})=\overline p(x)
\end{align*}
and $\check{p}(b)=\cL^{-1}(\cL\widetilde p(b))=\widetilde p(b)=\overline p(b)$. We deduce that $0<\overline p\le \check{p}\le \widetilde p$ on $(a,b)$ and by using the second inequality in~\eref{eq-astGuillaume1} we obtain that
\begin{equation}\label{eq-astGuillaume}
\int_{a}^{b}\ab{p-\check{p}}d\mu=\int_{a}^{b}\pa{\widetilde p-\check{p}}d\mu\le \int_{a}^{b}\pa{\widetilde p-\overline p}d\mu=\int_{a}^{b}\ab{p-\overline p}d\mu\le \frac{1}{D^{2}}.
\end{equation}
In a nutshell, the function $\check{p}$ possesses the following properties: for all $x\in (a,b)$, $0<\check{p}(x)\le p(x)$, hence $\kappa=\int_{a}^{b}\check{p}d\mu\in (0,1]$, and  
$\cL\circ \check{p}=\overline q$ is concave and piecewise affine on each element of a partition of $[a,b]$ into $2D$ intervals.  
The function $\check{p}\1_{(a,b)}$ satisfies thus $\{\check{p}\1_{(a,b)}>0\}=(a,b)$ as well as all the requirements for belonging to $\overline \cO_{2D+1}(\cL)$, except from the fact that $\check{p}\1_{(a,b)}$ might not be a density. Let us define
\[
s(x)=(\check{p}\1_{(a,b)})(\kappa x+(1-\kappa) c)\quad \text{for $x\in\R$.}
\]
Then $s$ is a density, $\{s>0\}$ is the open interval $(c+\kappa^{-1}(a-c), c+\kappa^{-1}(b-c))$ and for all $x\in \{s>0\}$ the function 
\[
\cL s(x)=(\cL\circ \check{p})(\kappa x+(1-\kappa) c)=\overline q\pa{\kappa x+(1-\kappa) c}
\]
is concave and piecewise affine on each element of a partition of $\{s>0\}$ into the $2D$-intervals. By definition, $\cL s$ takes the value $-\infty$ on $\{s=0\}$ which is a union of two intervals. Consequently, $s$ belongs to $\overline \cO_{2D+1}(\cL)$ and we conclude by using Lemma~\ref{lem-scale_dens}. 
\end{proof}

\begin{proof}[Proof of Theorem~\ref{thm-approx-scv}]
Since the function $\cL$ is concave and increasing on $(0,+\infty)$, it is continuous on this interval and its inverse $\cL^{-1}$ is convex, increasing and continuous on the range $(u,v)$ of $\cL$, with $-\infty\le u<v\le +\infty$. Throughout the proof, we shall repeatedly use the following equalities: for all $s,t>0$ with $s<t$
\begin{equation}\label{eq-cL}
\int_{s}^{t}z\cL_{r}'(z)dz=\int_{s}^{t}z\cL_{l}'(z)dz=\int_{\cL(s)}^{\cL(t)}\cL^{-1}(u)du.
\end{equation}
In particular, by letting $s$ decrease to 0 and $t$ increase to infinity, we respectively obtain that for all $t>0$
\begin{equation}\label{eq-cL1}
\int_{0}^{t}z\cL_{r}'(z)dz=\int_{0}^{t}z\cL_{l}'(z)dz=\int_{u}^{\cL(t)}\cL^{-1}(u)du
\end{equation}
and for all $s>0$
\begin{equation}\label{eq-cL2}
\int_{s}^{+\infty}z\cL_{r}'(z)dz=\int_{s}^{+\infty}z\cL_{l}'(z)dz=\int_{\cL(s)}^{v}\cL^{-1}(u)du.
\end{equation}

We obtain~\eref{eq-cL} by using Fubini's theorem and the fact that a concave function is absolutely continuous on any compact interval $[s,t]\subset (0,+\infty)$. More precisely,
\begin{align*}
\int_{\cL(s)}^{\cL(t)}\cL^{-1}(u)du&=\int_{\cL(s)}^{\cL(t)}\cro{\int_{0}^{+\infty}\1_{x\le \cL^{-1}(u)}dx}du
=\int_{\cL(s)}^{\cL(t)}\cro{\int_{0}^{+\infty}\1_{\cL(x)\le u}dx}du\\
&=\int_{0}^{t}\cro{\int_{\cL(s)}^{\cL(t)}\1_{\cL(x)\le u}du}dx=\int_{0}^{t}\pa{\cL(t)-\cL(u)\vee \cL(s)}du\\
&=\int_{0}^{t}\pa{\cL(t)-\cL(u\vee s)}du=\int_{0}^{t}\pa{\int_{s\vee u}^{t}\cL_{l}'(z)dz}du=\int_{0}^{t}\pa{\int_{s\vee u}^{t}\cL_{r}'(z)dz}du\\
&=\int_{s}^{t}\cL_{r}'(z)\pa{\int_{0}^{z}du}dz=\int_{s}^{t}z\cL_{r}'(z)dz=\int_{s}^{t}z\cL_{l}'(z)dz.
\end{align*}

 Let $a,b\in\R\cup\{\pm \infty\}$, $a<b$, such that $\{p>0\}=(a,b)$. Since $p$ is $\cL$-concave, $\cL p$ is concave, hence continuous on $(a,b)$ and there exists $c\in \R\cup\{\pm \infty\}$, $a\le c\le b$, such that $\cL p$ is nondecreasing on $(a,c)$ and nonincreasing on $(c,b)$ (with the convention that a function is always monotone on the empty set). In particular, if $c$ belongs to $(a,b)$,  $\cL p$ is continuously nondecreasing on $(a,c]$, continuously nonincreasing on $[c,b)$ and the same holds for $p=\cL^{-1}\circ \cL p$ since $\cL^{-1}$ is increasing. Let us now consider the case where $c\in\{a,b\}$. Our aim is to show that $c$ is necessarily a finite number and that $p$ can be continuously extended at $c$. We may then conclude that in any case, $\cL p$ and $p$ are both continuous and monotone on the (possibly empty) intervals $(a,c]$ and $[c,b)$ for some real number $c\in [a,b]$.

Let us start with the case $c=a$. Then $\cL p$ and $p$ are both continuously nonincreasing on $(c,b)$ and since $p$ is a density, $c$ is necessarily finite. Let us assume that $p(x)$ tends to $+\infty$ when $x$ tends to $c$ and get a contradiction. Since  $\cL p$ is concave and nonincreasing on $(c,b)$, it has a finite limit at $c$ and since under our assumption $p$ grows to infinity at $c$,  
\[
v=\lim_{y\uparrow +\infty}\cL(y)=\lim_{x\downarrow c}\cL p(x)<+\infty.
\]
We may therefore extend $\cL p$ continuously at $x=c$ by setting $\cL p(c)=v\in\R$. Let $x_{0}\in (c,b)$ and set $s=p(x_{0})>0$. It follows from the monotonicity of $\cL$ that $\cL p(x_{0})=\cL(s)<\cL p(c)=v$. Besides, $\cL p$ lies above its chord on the interval $[c,x_{0}]$ since it is concave and we obtain that
\[
v>\cL p(x)\ge \theta(x-c)+v>\cL p(x_{0})>u \quad \text{for all $x\in (c,x_{0})$ with}\quad \theta=\frac{\cL(s)-v}{x_{0}-c}<0.
\]
Consequently, using that $\cL^{-1}$ is increasing, we get
\begin{align*}
1\ge \int_{c}^{x_{0}}p(x)dx&=\int_{c}^{x_{0}}\cL^{-1}\pa{\cL p(x)}dx\ge \int_{c}^{x_{0}}\cL^{-1}\pa{\theta(x-c)+v}dx\ge \frac{1}{|\theta|}\int_{\cL(s)}^{v}\cL^{-1}(u)du.
\end{align*}
In particular, $\cL^{-1}$ is integrable on $[\cL(s),v)$. This is in fact impossible under our assumption~\eref{eq-condcL}. If $\cL^{-1}$ were integrable, we would deduce from \eref{eq-cL2} and the fact that since $\cL$ is increasing  that
\begin{align*}
\int_{\cL(s)}^{v}\cL^{-1}(u)du&\ge \int_{\cL( \lambda s)}^{v}\cL^{-1}(u)du=\int_{\lambda s}^{+\infty}z\cL_{l}'(z)dz=\int_{s}^{+\infty}\lambda^{2}u\cL_{l}'(\lambda u)dz\\
&\ge \kappa\int_{s}^{+\infty}u\cL_{r}'(u)dz=\kappa\int_{\cL(s)}^{v}\cL^{-1}(u)du>\int_{\cL(s)}^{v}\cL^{-1}(u)du,
\end{align*}
which is the contradiction we are looking for. The case $c=b$ is obtained by applying this result to the density $x\mapsto p(-x)$ which is nonincreasing on $(-b,-a)$ and also belongs to $\overline \cM(\cL)$.

The density $p$ may therefore be continuously extended at $c$ so that it is nondecreasing on $(a,c]$ and nonincreasing on $[c,b)$. Let us now find an element $\overline p\in \overline{\cO}_{2D+3}(\cL)$ that satisfies \eref{eq-bd-approx-scv}. It is actually sufficient to restrict ourselves to the case $c=0$ for the following reasons. If $p$ belongs to $\overline \cM(\cL)$, so does the density $p_{0}:x\mapsto p(x+c)$ which is nondecreasing on $(a-c,0]$ and nonincreasing on $[0,b-c)$. If we can find $\overline p_{0}\in \overline{\cO}_{2D+3}(\cL)$ that satisfies \eref{eq-bd-approx-scv} with $p_{0}$ and $\overline p_{0}$ in place of $p$ and $\overline p$, then by the invariance property of the Lebesgue measure under translation, the same inequality holds for $p$ and $\overline p:x\mapsto p_{0}(x-c)$. It remains to notice that the so-defined function $\overline p$ is also an element of $\overline{\cO}_{2D+3}(\cL)$. From now on, we shall therefore assume that $c=0$. 

Let us assume that $b\ne 0$. We set $b_{1}=b$, $y_{j}=\lambda^{-j}p(0)$ for $j\in\N$, $x_{0}=0$ and for $j\in\N\setminus\{0\}$, $x_{j}$ a positive solution of the equation $p(x)=y_{j}$ when such a solution exists and $x_{j}=b_{1}>0$ otherwise. Since the density $p$ is continuous and monotone on $[0,b_{1})$, a positive solution $x_{j}$ always exists for each $j\ge 1$ when $b_{1}=+\infty$. The situation $x_{j}=b_{1}$ for some $j\ge 1$ can therefore only occur when $b_{1}$ is finite. In this case, we extend $p$ continuously at $b_{1}$ so that $p$ is well-defined and continuous on $[0,b_{1}]$. Note that in any case, $p(x_{j})\ge y_{j}>0$ for all $j\in\N$. 

In what follows, our aim is to find a nonnegative function $\overline q_{1}$ on $[0,+\infty)$ with the following properties. There exists some $J\ge 1$ such that for all $x\in [0,x_{J}]$, $0<\overline q_{1}(x)\le p(x)$, $\overline q_{1}(x)=0$ for all $x\not \in [0,x_{J}]$, $\cL\circ \overline q_{1}$ is a $(D+1)$-linear interpolation of $\cL p$ on $[0,x_{J}]$ and 
\begin{equation}\label{eq-cL-Approx}
\int_{0}^{b_{1}}\ab{p-\overline q_{1}}d\mu\le \frac{\gamma\zeta_{1}}{D^{2}}\quad \text{with}\quad \zeta_{1}=\int_{0}^{x_{1}}p(u)du>0.
\end{equation}

In what follows, we set 
\[
D_{1}=\PES{\frac{D\lambda \sqrt{(\lambda-1)}}{\sqrt{2\gamma \kappa}}}\ge \pa{\frac{D\lambda \sqrt{(\lambda-1)\zeta_{1}}}{\sqrt{2\gamma \kappa}}}\vee 1.
\]
Since $\cL p$ is nonincreasing and concave on $[0,b_{1})$, it admits a right and left derivative at all point $x\in (0,b_{1})$ as well as a right derivative $(\cL p)_{r}'(0)\in [0,+\infty)$ at 0.  If $b_{1}$ is finite, the limit $\lim_{x\uparrow b_{1}}(\cL p)'_{l}(x)\in [-\infty,0]$ exists and we shall denote it $(\cL p)'_{l}(b_{1})$. Throughout the proof, we consider a function $\phi$ such that for all $x\in (0,b_{1})$, $(\cL p)'_{r}(x)\le \phi(x)\le (\cL p)'_{l}(x)\le 0$, hence $|(\cL p)'_{l}(x)|\le |\phi(x)|\le |(\cL p)'_{r}(x)|$, and $\phi(b_{1})=(\cL p)'_{l}(b_{1})\in [-\infty,0]$ when $b_{1}<+\infty$. Since $\cL p$ is concave
\[
0\ge (\cL p)_{r}'(0)\ge \Delta_{1}=\frac{\cL p(x_{1})-\cL p(0)}{x_{1}}\ge (\cL p)_{l}'(x_{1})\ge \phi(x_{1}).
\]
If $|(\cL p)_{l}'(x_{1})|$ is finite, which is always the case when $x_{1}\ne b_{1}$, by using the convention $0/0=0$ and applying Proposition~\ref{prop-guer1}-(i) to $\cL p$, we may find a $D_{1}$-linear interpolation $\ell_{1}$ of $\cL p$ on the interval $[0,x_{1}]$ such that
\begin{align*}
\int_{0}^{x_{1}}\ab{\cL p-\ell_{1}}d\mu&\le  \frac{x_{1}^{2}}{2 D_{1}^{2}}\frac{\pa{(\cL p)_{r}'(0)-\Delta_{1}}\pa{\Delta_{1}-(\cL p)_{l}'(x_{1})}}{(\cL p)_{r}'(0)-(\cL p)_{l}'(x_{1})}\\
&= \frac{x_{1}^{2}}{2 D_{1}^{2}}\cro{1-\frac{\Delta_{1}-(\cL p)_{l}'(x_{1})}{(\cL p)_{r}'(0)-(\cL p)_{l}'(x_{1})}}\pa{\Delta_{1}-(\cL p)_{l}'(x_{1})}\\
&\le \frac{x_{1}^{2}|\Delta_{1}|}{2D_{1}^{2}}\pa{1-\frac{|\Delta_{1}|}{\ab{(\cL p)_{l}'(x_{1})}}}\le \frac{x_{1}\pa{\cL p(0)-\cL p(x_{1})}}{2 D_{1}^{2}}\pa{1-\frac{|\Delta_{1}|}{\ab{\phi(x_{1})}}}.
\end{align*}
It follows from Proposition~\ref{prop-guer1}-(iii) that the last inequality remains actually true when $x_{1}=b_{1}<+\infty$ and $\phi(b_{1})=(\cL p)_{l}'(b_{1})=-\infty$ with the convention $1/(+\infty)=0$. Since $\cL$ is concave on $(0,+\infty)$, it is Lispchitz on $[y_{1},y_{0}]\subset (0,+\infty)$ with Lipschitz constant not larger than $\cL_{r}'(y_{1})$. Using that $p(0)=y_{0}$ and $p(x_{1})\in [y_{1},y_{0}]$, we deduce  
\begin{align}
\int_{0}^{x_{1}}\ab{\cL p-\ell_{1}}d\mu&\le \frac{x_{1}\cL_{r}'(y_{1})(y_{0}-y_{1})}{2 D_{1}^{2}}\pa{1-\frac{|\Delta_{1}|}{\ab{\phi(x_{1})}}}=\frac{(\lambda-1)x_{1}\cL_{r}'(y_{1})y_{1}}{2 D_{1}^{2}}\pa{1-\frac{|\Delta_{1}|}{\ab{\phi(x_{1})}}}.\label{eq-pfL-01}
\end{align}
Since $\cL p$ is concave, nonincreasing and $p(x_{1})\ge y_{1}>0$, 
\[
u<\cL (y_{1})\le \cL p(x_{1})\le \ell_{1}(x)\le \cL p(x)\le \cL p(0)=\cL(y_{0})<v\quad \text{for all $x\in [0,x_{1}]$.}
\]
Since $\ell_{1}(x)\in (u,v)$ for $x\in  [0,x_{1}]$, we may define the function $q_{1}=\cL^{-1}\circ \ell_{1}$ on $[0,x_{1}]$. Since $\cL^{-1}$ is increasing, the function $q_{1}$ satisfies  $0<q_{1}(x)\le p(x)$ for all $x\in [0,x_{1}]$. Furthermore, since $\cL^{-1}$ is convex on $(u,v)$, it is Lipschitz on $[\cL(y_{1}),\cL(y_{0})]$ with a Lipschitz constant not larger than $(\cL^{-1})_{l}'(\cL(y_{0}))=(\cL_{l}'(y_{0}))^{-1}=(\cL_{l}'(\lambda y_{1}))^{-1}$. Using~\eref{eq-condcL} and \eref{eq-pfL-01}, we get
\begin{align*}
\int_{0}^{x_{1}}\ab{p(x)-q_{1}(x)}dx&=\int_{0}^{x_{1}}\ab{\cL^{-1}(\cL p(x))-\cL^{-1}(\ell_{1}(x))}dx\le \frac{1}{\cL _{l}'(\lambda y_{1})}\int_{0}^{x_{1}}\ab{\cL p(x)-\ell_{1}(x)}dx\\
&\le \frac{(\lambda-1)x_{1}\cL_{r}'(y_{1})y_{1}}{2 \cL _{l}'(\lambda y_{1}) D_{1}^{2}}\pa{1-\frac{|\Delta_{1}|}{\ab{\phi(x_{1})}}}\le \frac{\lambda^{2}(\lambda-1)x_{1}y_{1}}{2\kappa  D_{1}^{2}}\pa{1-\frac{|\Delta_{1}|}{\ab{\phi(x_{1})}}}.
\end{align*}
Since $p(x)\ge y_{1}$ for all $x\in [0,x_{1}]$, $x_{1}y_{1}\le \int_{0}^{x_{1}}p(u)du= \zeta_{1}\le 1$ and we obtain that for our choice of $D_{1}$ 
\begin{align}
\int_{0}^{x_{1}}\ab{p-q_{1}}d\mu&\le \frac{\lambda^{2}(\lambda-1)\zeta_{1}}{2\kappa  D_{1}^{2}}\pa{1-\frac{|\Delta_{1}|}{\ab{\phi(x_{1})}}}\le \frac{\gamma\zeta_{1}}{D^{2}}\pa{1-\frac{|\Delta_{1}|}{\ab{\phi(x_{1})}}}\le \frac{\gamma\zeta_{1}}{D^{2}}.\label{eq-cL-B1}
\end{align}
If $x_{1}=b_{1}$ the function $\overline q_{1}=q_{1}\1_{[0,b_{1}]}$ satisfies our requirements: inequality \eref{eq-cL-Approx} follows from~\eref{eq-cL-B1} and it follows from our definition \eref{eq-cL-defR} of $\gamma$ that $\cL \circ \overline q_{1}$ is a $D_{1}$-linear interpolation on $[0,x_{1}]$ with 
\[
D_{1}\le \frac{D\lambda \sqrt{(\lambda-1)}}{\sqrt{2\kappa \gamma}}+1\le D+1.
\]

Otherwise $x_{1}\ne b_{1}$, $p(x_{1})=y_{1}$ and 
\[
\Delta_{1}=\frac{\cL p(x_{1})-\cL p(0)}{x_{1}}=\frac{\cL(y_{1})-\cL(y_{0})}{x_{1}}<0.
\]
In such a situation, we set for $j\ge 1$ 
\[
v_{j}=\sqrt{(y_{j-1}-y_{j})\cro{\cL (y_{j-1})-\cL(y_{j})}},\quad w_{j}=\int_{y_{j}}^{y_{j-1}}z\cL'_{l}(z)dz
\]
for $j\ge 2$, 
\[
D_{j}=\PES{\frac{\lambda v_{j} D}{2 \sqrt{2\zeta_{1}\kappa \gamma |\Delta_{1}|}}}\ge \pa{\frac{\lambda v_{j} D}{2 \sqrt{2\zeta_{1}\kappa \gamma |\Delta_{1}|}}}\vee 1,
\]
and 
\[
J\et=\PES{\frac{2}{\log \kappa}\pa{\frac{\kappa}{\kappa-1}}^{1/2}\frac{D}{\sqrt{\gamma}}}\le \pa{\frac{2}{\log \kappa}\pa{\frac{\kappa}{\kappa-1}}^{1/2}\frac{D}{\sqrt{\gamma}}}+ 1.
\]
Since $\cL$ is concave, we deduce from \eref{eq-condcL} that for all $j\ge 1$, 
\begin{align*}
v_{j+1}^{2}&=(y_{j}-y_{j+1})\cro{\cL (y_{j})-\cL(y_{j+1})}=(y_{j}-y_{j+1})\int_{y_{j+1}}^{y_{j}}\cL_{r}'(t)dt=\frac{y_{j}-y_{j+1}}{\lambda}\int_{y_{j}}^{y_{j-1}}\cL_{r}'(u/\lambda)du\\
&\le \frac{\lambda(y_{j}-y_{j+1})}{\kappa}\int_{y_{j}}^{y_{j-1}}\cL_{l}'(u)du=\frac{1}{\kappa}\pa{y_{j-1}-y_{j}}\cro{\cL (y_{j-1})-\cL(y_{j})}=\frac{v_{j}^{2}}{\kappa}.
\end{align*}
By using \eref{eq-cL}, we obtain similarly that
\begin{align*}
w_{j+1}&= \int_{y_{j+1}}^{y_{j}}z\cL'_{r}(z)dz=\frac{1}{\lambda^{2}}\int_{y_{j}}^{y_{j-1}}u\cL'_{r}(u/\lambda)du\le \frac{1}{\kappa}\int_{y_{j}}^{y_{j-1}}u\cL'_{l}(u)du=\frac{w_{j}}{\kappa}.
\end{align*}
Besides, since $\cL p$ is concave, it lies above its chord on the interval $[0,x_{1}]$ which means that
\[
\cL p(x)\ge \frac{\cL (p(x_{1}))-\cL (p(0))}{x_{1}}x+\cL (p(0))=\Delta_{1}x+\cL (y_{0})\quad \text{for all $x\in [0,x_{1}]$}
\]
and consequently,
\begin{align*}
w_{1}&=\int_{y_{1}}^{y_{0}}z\cL'_{l}(z)dz=\int_{\cL(y_{1})}^{\cL(y_{0})}\cL^{-1}(u)du=|\Delta_{1}|\int_{0}^{x_{1}}\cL^{-1}\pa{\Delta_{1}x+\cL (y_{0})}dx\\
&\le |\Delta_{1}|\int_{0}^{x_{1}}\cL^{-1}\pa{\cL p(x)}dx=|\Delta_{1}|\int_{0}^{x_{1}}p(x)dx= |\Delta_{1}|\zeta_{1}.
\end{align*}
We deduce from these inequalities that 
\begin{equation}\label{eq-cL-vw}
\sum_{j\ge 2}v_{j}\le \frac{v_{1}}{\sqrt{\kappa}-1} \quad \text{and for all $J\ge 1$}\quad \sum_{j\ge J+1}w_{j}\le \frac{\kappa^{1-J}w_{1}}{\kappa-1}\le \frac{\kappa^{1-J}|\Delta_{1}|\zeta_{1}}{\kappa-1}.
\end{equation}
Since
\begin{align*}
\pa{\frac{v_{1}}{\sqrt{|\Delta_{1}|}}}^{2}&=\frac{x_{1}(y_{0}-y_{1})\cro{\cL (y_{0})-\cL(y_{1})}}{\cL(y_{0})-\cL(y_{1})}=(\lambda-1)x_{1}y_{1}\le (\lambda-1)\zeta_{1},
\end{align*}
we derive from~\eref{eq-cL-vw} that
\begin{align*}
D_{1}+\sum_{j=2}^{J\et}D_{j}&\le  \frac{D\lambda \sqrt{(\lambda-1)}}{\sqrt{2\kappa \gamma }}+1+\sum_{j=2}^{J\et}\cro{1+\frac{\lambda v_{j} D}{2\sqrt{2\zeta_{1}\kappa \gamma |\Delta_{1}|}}}\le J\et+\frac{\lambda D}{\sqrt{2 \kappa \gamma}}\pa{\sqrt{\lambda-1}+\frac{1}{2\sqrt{\zeta_{1}|\Delta_{1}|}}\sum_{j\ge 2}v_{j}}\\
&\le J\et+\frac{\lambda D}{\sqrt{2 \kappa \gamma}}\pa{\sqrt{\lambda-1}+\frac{v_{1}}{2(\sqrt{\kappa}-1)\sqrt{\zeta_{1}|\Delta_{1}|}}}\le  J\et+\frac{\lambda\sqrt{\lambda-1} D}{\sqrt{2 \kappa \gamma}}\pa{1+\frac{1}{2(\sqrt{\kappa}-1)}}\\
&\le \cro{\frac{2}{\log \kappa}\pa{\frac{\kappa}{\kappa-1}}^{1/2}+\frac{\lambda\sqrt{\lambda-1}
}{\sqrt{2 \kappa }}\pa{1+\frac{1}{2(\sqrt{\kappa}-1)}}} \frac{D}{\sqrt{\gamma}}+1.
\end{align*}
which, with the value of $\gamma$ given by~\eref{eq-cL-defR} leads to
\begin{equation}\label{eq-cL-D}
D_{1}+\sum_{j=1}^{J\et}D_{j}\le D+1.
\end{equation}

As long as $x_{j-1}\ne b_{1}$ with $j\ge 2$, the interval $[x_{j-1},x_{j}]$ does not reduce to a singleton and by  Proposition~\ref{prop-guer1}-(ii), we may find a $D_{j}$-linear interpolation $\ell_{j}$ of $\cL p$ on this interval that satisfies 
\begin{align}
\int_{x_{j-1}}^{x_{j}}\ab{\cL p-\ell_{j}}d\mu&\le \frac{\pa{\cL(y_{j-1})-\cL(y_{j})}^{2}}{8D_{j}^{2}}\ab{\frac{1}{(\cL p)'_{r}(x_{j-1})}-\frac{1}{(\cL p)'_{l}(x_{j})}}\nonumber\\
&= \frac{\pa{\cL(y_{j-1})-\cL(y_{j})}^{2}}{8D_{j}^{2}}\pa{\frac{1}{|(\cL p)'_{r}(x_{j-1})|}-\frac{1}{|(\cL p)'_{l}(x_{j})|}}\nonumber\\
&\le  \frac{\pa{\cL(y_{j-1})-\cL(y_{j})}^{2}}{8|\Delta_{1}|D_{j}^{2}}\pa{\frac{|\Delta_{1}|}{|\phi(x_{j-1})|}-\frac{|\Delta_{1}|}{|\phi(x_{j})|}}\label{eq-pfL-02}
\end{align}
Since $p(x_{j})\ge y_{j}>0$, 
\[
u<\cL(y_{j})\le \cL p(x_{j})\le \ell_{1}(x)\le \cL p(x)\le \cL p(x_{j-1})<v\quad \text{for all $x\in [x_{j-1},x_{j}]$}
\]
and as before, we may define the function $q_{j}=\cL^{-1}\circ \ell_{j}$ on $[x_{j-1},x_{j}]$. It satisfies $0<q_{j}(x)\le p(x)$ for all $x\in [x_{j-1},x_{j}]$ and since $\cL^{-1}$ is Lipschitz on $[\cL(y_{j}),\cL(y_{j-1})]$ with a Lipschitz constant not larger than $(\cL^{-1})_{l}'(\cL(y_{j-1}))=(\cL_{l}'(y_{j-1}))^{-1}=(\cL_{l}'(\lambda y_{j}))^{-1}$, we deduce from \eref{eq-pfL-02} that 
\begin{align*}
\int_{x_{j-1}}^{x_{j}}\ab{p-q_{j}}d\mu&=\int_{x_{j-1}}^{x_{j}}\ab{\cL^{-1}(\cL p(x))-\cL^{-1}(\ell_{j}(x))}dx\le \frac{\pa{\cL(y_{j-1})-\cL(y_{j})}^{2}}{8\cL_{l}'(\lambda y_{j})|\Delta_{1}|D_{j}^{2}}\pa{\frac{|\Delta_{1}|}{|\phi(x_{j-1})|}-\frac{|\Delta_{1}|}{|\phi(x_{j})|}}.
\end{align*}
Using \eref{eq-condcL} and the fact that $\cL$ is concave, 
\begin{align*}
\cL_{l}'(\lambda y_{j})\ge \frac{\kappa}{\lambda^{2}}\cL_{r}'(y_{j})\ge  \frac{\kappa}{\lambda^{2}}\frac{\cL(y_{j-1})-\cL(y_{j})}{y_{j-1}-y_{j}}= \frac{\kappa}{\lambda^{2}}\frac{\pa{\cL(y_{j-1})-\cL(y_{j})}^{2}}{v_{j}^{2}}
\end{align*}
which, with our choice of $D_{j}$, leads to
\begin{align}
\int_{x_{j-1}}^{x_{j}}\ab{p-q_{j}}d\mu&\le \frac{\lambda^{2} v_{j}^{2}}{8 \kappa |\Delta_{1}|D_{j}^{2}}\pa{\frac{|\Delta_{1}|}{|\phi(x_{j-1})|}-\frac{|\Delta_{1}|}{|\phi(x_{j})|}}=\frac{\gamma\zeta_{1}}{D^{2}}\pa{\frac{|\Delta_{1}|}{|\phi(x_{j-1})|}-\frac{|\Delta_{1}|}{|\phi(x_{j})|}}.\label{eq-cL-B2}
\end{align}

If $J\ge 1$ and $x_{J}\ne b_{1}$, $p(x_{J})=y_{J}$ and since $\cL p$ is concave, it lies under its tangent at $x_{J}$ (which is nonpositive) so that for all $x\in [x_{J},b_{1})$
\[
u<\cL p(x)\le \cL p(x_{J})+(\cL p)_{r}'(x_{J})(x-x_{J})\le \cL p(x_{J})=\cL (y_{J})<v.
\]
Doing the change of variables $t=\cL p(x_{J})+(\cL p)_{r}'(x_{J})(x-x_{J})\in (u,v)$ and using \eref{eq-cL1}, \eref{eq-cL-vw} and the monotonicity of $\cL^{-1}$ we obtain that 
\begin{align*}
\int_{x_{J}}^{b_{1}}p(x)dx&=\int_{x_{J}}^{b_{1}}\cL^{-1}\pa{\cL p(x)}dx\le \int_{x_{J}}^{b_{1}}\cL^{-1}\pa{\cL p(x_{J})+(\cL p)_{r}'(x_{J})(x-x_{J})}dx\\
&\le \frac{1}{|(\cL p)_{r}'(x_{J})|}\int_{u}^{\cL(y_{J})}\cL^{-1}(t)dt=\frac{1}{|(\cL p)_{r}'(x_{J})|}\int_{0}^{y_{J}}z\cL_{l}'(z)dt\\
&=\frac{1}{|(\cL p)_{r}'(x_{J})|}\sum_{j\ge J+1}w_{j}\le \frac{\kappa^{1-J}}{\kappa-1}\frac{|\Delta_{1}|\zeta_{1}}{|(\cL p)_{r}'(x_{J})|}.
\end{align*}
Note that this inequality is clearly satisfided when $x_{J}=b_{1}$, hence for all $J\ge 1$. It is in particular true for $J=J\et\ge 1$ and 
since by definition of $J\et$,
\begin{align*}
\kappa^{J\et}=\exp\pa{J\et\log \kappa}\ge \pa{1+\frac{1}{2}J\et\log \kappa}^{2}\ge \pa{\frac{1}{2}J\et\log \kappa}^{2}\ge \frac{\kappa}{\kappa-1}\frac{D^{2}}{\gamma}\quad  \text{ hence }\quad \frac{\kappa^{1-J\et}}{\kappa-1}\le \frac{\gamma}{D^{2}},
\end{align*}
we obtain that for $J=J\et$
\begin{align}
\int_{x_{J}}^{b_{1}}p(x)dx\le \frac{\gamma\zeta_{1}}{D^{2}}\frac{|\Delta_{1}|}{|(\cL p)_{r}'(x_{J})|}\le \frac{\gamma\zeta_{1}}{D^{2}}\frac{|\Delta_{1}|}{|\phi(x_{J})|}.\label{eq-cL-B3}
\end{align}
Let $\overline J=\inf\{j\ge 1,\; x_{j}=b_{1}\}$, with the convention $\inf\varnothing=+\infty$, and $J=J\et\wedge \overline J\ge 1$. The function
\[
\overline q_{1}=q_{1}\1_{[0,x_{1}]}+\sum_{j=2}^{J}q_{j}\1_{(x_{j-1},x_{j}]}
\]
satisfies $0<\overline q_{1}(x)\le p(x)$ for all $x\in [0,x_{J}]$, $\overline q_{1}(x)=0$ outside the interval $[0,x_{J}]$. Furthermore, it follows from \eref{eq-cL-D} that $\cL\circ \overline q_{1}$ is a $\overline D_{1}$-linear interpolation of $\cL p$ on $[0,x_{J}]$ with 
\[
\overline D_{1}\le D_{1}+\sum_{j=2}^{J}D_{j}\le D_{1}+\sum_{j=2}^{J\et}D_{j}\le D+1
\]
and by putting the inequalities \eref{eq-cL-B1}, \eref{eq-cL-B2} and  \eref{eq-cL-B3} together,  we obtain that 
\begin{align*}
\int_{0}^{b_{1}}\ab{p-\overline q_{1}}d\mu&\le \sum_{j=1}^{J}\int_{x_{j-1}}^{x_{j}}\ab{p-\overline q_{j}}d\mu+\int_{x_{J}}^{b_{1}}\ab{p-\overline q_{j}}d\mu\\
&\le \frac{\gamma\zeta_{1}}{D^{2}}\pa{1-\frac{|\Delta_{1}|}{|\phi(x_{1})|}+\sum_{j=2}^{J}\pa{\frac{|\Delta_{1}|}{|\phi(x_{j-1})|}-\frac{|\Delta_{1}|}{|\phi(x_{j})|}}+\frac{|\Delta_{1}|}{|\phi(x_{J})|}}= \frac{\gamma\zeta_{1}}{D^{2}}.
\end{align*}
The function $\overline q_{1}$ satisfies thus our requirements. 

If $a\ne 0$, by applying the previous result to the density $x\mapsto p(-x)$, which also belongs to $\overline \cM(\cL)$, we obtain that there exists numbers $x_{-J}\le x_{-1}<0$ and a function $\overline q_{-1}$ which vanishes outside $[x_{-J},0]$ that satisfies the following properties: $0<\overline q_{-1}(x)\le p(x)$ for all $x\in  [x_{-J},0]$, $\cL\circ \overline q_{-1}$ is a $\overline D_{-1}$-linear interpolation of $\cL p$ on $[x_{-J},0]$ with $\overline D_{-1}\le D+1$ and  
\[
\int_{a}^{0}\ab{p-\overline q_{-1}}d\mu\le \frac{\gamma\zeta_{-1}}{D^{2}}\quad \text{with}\quad \zeta_{-1}=\int_{x_{-1}}^{0}p(x)dx.
\]

By taking $\overline q=\overline q_{1}$ and $x_{-1}=x_{-J}=0$ when $a=0$,  $\overline q=\overline q_{-1}$ and $x_{1}=x_{J}=0$ when $b=0$ and $\overline q=q_{-1}\1_{[x_{-J},0]}+\overline q_{1}\1_{(0,x_{J}]}$ otherwise, we obtain a function with the following properties. For all $x\in [x_{-J},x_{J}]$, $0<\overline q(x)\le p(x)$ and $\overline q(x)=0$ for all $x\not \in [x_{-J},x_{J}]$. The function $\cL \circ \overline q$ is a $\overline D$-linear interpolation of $\cL p$ on the interval $[x_{-J},x_{J}]$ with $\overline D\le 2(D+1)$, it is in particular concave on this interval, and 
\[
\int_{\R}\ab{p-\overline q}d\mu=\int_{a}^{b}\ab{p-\overline q}d\mu\le \frac{\gamma}{D^{2}}\pa{\zeta_{-1}+\zeta_{1}}\le  \frac{\gamma}{D^{2}}.
\]
By construction, the linear interpolation $\cL \circ \overline q$ of $\cL p$ on $[x_{-J},x_{J}]$ is unimodal on this interval with a mode at 0. Since $\cL^{-1}$ is increasing, the same is true for $\overline q$. We finally define 
\[
\overline p(x)=\overline q(\eta x)\1_{(-x_{J},x_{J})}(\eta x)\quad \text{with}\quad \eta=\int_{-x_{J}}^{x_{J}}\overline q(u)du=\int_{\R}\overline q(u)du.
\]
Since $0<\overline q(x)\le p(x)$ for all $x\in [x_{-J},x_{J}]$, $\eta\in (0,1]$. By construction, $\overline p$ is a density for which the set $\{\overline p>0\}$ is the open interval $(x_{-J}/\eta,x_{J}/\eta)$. Since $\cL\circ \overline q$ is a $\overline D$-linear interpolation of $\cL p$ on $[x_{-J},x_{J}]$,  it is concave on this interval and $\cL\overline p=\cL\circ \overline q(\eta\cdot)$ is therefore concave on $\{\overline p>0\}=(x_{-J}/\eta,x_{J}/\eta)$ and affine on each element of a partition of this interval into at most $\overline D$ intervals. The function $\cL \overline p$ takes the value $-\infty$ on $(-\infty,x_{-J}/\eta]$ and $[x_{J}/\eta,+\infty)$ and is therefore affine or takes the value $-\infty$ on a partition of $\R$ into at most $\overline D+2$ intervals. The density $\overline p$ therefore belongs to $\overline \cO_{\overline D+1}(\cL)\subset 
\overline \cO_{2D+3}(\cL)$. Since $\overline q$ is nondecreasing on $(x_{-J},0)$ and nonincreasing on $(0,x_{J})$, 
Lemma~\ref{lem-scale_dens} applies with $f=\overline q$, $c=0$ and $\kappa=\eta$, which leads to 
\begin{align*}
\int_{\R}\ab{p-\overline p}d\mu\le 2\int_{\R}\ab{p-\overline q}d\mu\le \frac{2\gamma}{D^{2}}.
\end{align*}
This complets the proof of Theorem~\ref{thm-approx-scv}.
\end{proof}

\section*{Acknowledgments}
One of the authors thanks Lutz D\"umbgen for asking if we could establish a result with $d(P\et, \widehat P)$ in place  of $n^{-1}\sum_{i=1}^{n}d(P_{i}, \widehat P)$
 at a conference where a preliminary version of this work was presented. His question spurred the authors to improve their result and obtain this new version of Theorem~\ref{shape-estimation-th}.

\end{document}